\documentclass{amsart}
\usepackage{amsmath,amssymb,amsfonts, amsthm}
\usepackage{stmaryrd}
\usepackage{bbold}
\usepackage{mathtools}
\usepackage{comment}
\usepackage{xcolor}
\definecolor{darkgreen}{rgb}{0,0.45,0} 
\definecolor{darkishgreen}{rgb}{0,0.65,0} 
\definecolor{darkred}{rgb}{0.75,0,0}
\definecolor{darkblue}{rgb}{0,0,0.6} 
\usepackage[pdfborder=0,pagebackref,colorlinks,citecolor=darkgreen,linkcolor=darkgreen,urlcolor=darkblue]{hyperref}
\renewcommand*{\backref}[1]{}
\renewcommand*{\backrefalt}[4]{({%
    \ifcase #1 Not cited.%
          \or On p.~#2%
          \else On pp.~#2%
    \fi%
    })}
\usepackage{mathpartir}
\usepackage[all,2cell]{xy}\UseAllTwocells
\usepackage{tikz-cd}
\usepackage{color}
\usepackage{enumitem}
\usepackage{cleveref}
\crefname{thm}{Theorem}{Theorems}
\crefname{cor}{Corollary}{Corollaries}
\crefname{prop}{Proposition}{Propositions}
\crefname{defn}{Definition}{Definitions}
\crefname{rmk}{Remark}{Remarks}
\crefname{fig}{Figure}{Figures}
\crefname{ex}{Example}{Examples}
\crefname{lem}{Lemma}{Lemmas}
\crefname{ax}{Axiom}{Axioms}

\crefformat{section}{\S#2#1#3}
\Crefformat{section}{Section~#2#1#3}
\crefrangeformat{section}{\S\S#3#1#4--#5#2#6}
\Crefrangeformat{section}{Sections~#3#1#4--#5#2#6}
\crefmultiformat{section}{\S\S#2#1#3}{ and~#2#1#3}{, #2#1#3}{ and~#2#1#3}
\Crefmultiformat{section}{Sections~#2#1#3}{ and~#2#1#3}{, #2#1#3}{ and~#2#1#3}
\crefrangemultiformat{section}{\S\S#3#1#4--#5#2#6}{ and~#3#1#4--#5#2#6}{, #3#1#4--#5#2#6}{ and~#3#1#4--#5#2#6}
\Crefrangemultiformat{section}{Sections~#3#1#4--#5#2#6}{ and~#3#1#4--#5#2#6}{, #3#1#4--#5#2#6}{ and~#3#1#4--#5#2#6}

\crefformat{appendix}{Appendix~#2#1#3}
\Crefformat{appendix}{Appendix~#2#1#3}
\crefrangeformat{appendix}{Appendices~#3#1#4--#5#2#6}
\Crefrangeformat{appendix}{Appendices~#3#1#4--#5#2#6}
\crefmultiformat{appendix}{Appendices~#2#1#3}{ and~#2#1#3}{, #2#1#3}{ and~#2#1#3}
\Crefmultiformat{appendix}{Appendices~#2#1#3}{ and~#2#1#3}{, #2#1#3}{ and~#2#1#3}
\crefrangemultiformat{appendix}{Appendices~#3#1#4--#5#2#6}{ and~#3#1#4--#5#2#6}{, #3#1#4--#5#2#6}{ and~#3#1#4--#5#2#6}
\Crefrangemultiformat{appendix}{Appendices~#3#1#4--#5#2#6}{ and~#3#1#4--#5#2#6}{, #3#1#4--#5#2#6}{ and~#3#1#4--#5#2#6}

\crefname{figure}{Figure}{Figures}

\ifx\color@rgb\@undefined\else
  \definecolor{maroon}{rgb}{0.5,0,0}
\fi

\theoremstyle{plain}
\newtheorem{thm}{Theorem}[section] 
\newtheorem{cor}[thm]{Corollary}
\newtheorem{prop}[thm]{Proposition}
\newtheorem{lem}[thm]{Lemma}

\newtheorem{ax}[thm]{Axiom}

\theoremstyle{definition}
\newtheorem{defn}[thm]{Definition}

\theoremstyle{remark}
\newtheorem{ex}[thm]{Example}
\newtheorem{rmk}[thm]{Remark}

\makeatletter
\let\c@equation\c@thm
\makeatother
\numberwithin{equation}{section}

\newcommand{\tprod}{\textstyle\prod}
\newcommand{\tsum}{\textstyle\sum}
\newcommand{\exten}[4]{\left\langle\mathchoice{\textstyle\prod_{#1}}{\textstyle\prod_{#1}}{\scriptstyle\prod_{#1}}{\scriptscriptstyle\prod_{#1}} #2 \middle|^{#3}_{#4}\right\rangle}
\newcommand{\ndexten}[4]{\left\langle #1 \to #2 \middle|^{#3}_{#4}\right\rangle}
\newcommand{\ccexten}[4]{\langle\Pi[#1][#2] |^{#3}_{#4}\rangle}
\newcommand{\extfn}[1]{{#1}_\#}
\newcommand{\jdeq}{\equiv}
\newcommand{\defeq}{\coloneqq}

\newcommand{\types}{\vdash}

\newcommand{\type}{\;\mathsf{type}}

\newcommand{\unittype}{\ensuremath{\mathbf{1}}}
\newcommand{\booltype}{\ensuremath{\mathbf{2}}}
\newcommand{\univtype}{\mathcal{U}}

\newcommand{\cube}{\;\mathsf{cube}}
\newcommand{\tope}{\;\mathsf{tope}}
\newcommand{\shape}{\;\mathsf{shape}}
\newcommand{\ctx}{\;\mathsf{ctx}}
\newcommand{\sh}[2]{\{#1\mid #2\}}
\newcommand{\restr}[1]{\overline{#1}}
\newcommand{\pair}[1]{\langle #1\rangle}
\newcommand{\rec}{\mathsf{rec}}
\newcommand{\refl}{\mathsf{refl}}
\newcommand{\idtoiso}{\mathsf{idtoiso}}
\newcommand{\idtoarr}{\mathsf{idtoarr}}

\newcommand{\evid}{\mathsf{evid}}
\newcommand{\yon}{\mathsf{yon}}

\newcommand{\connmax}[1]{\mathsf{V}_{#1}}
\newcommand{\connmin}[1]{\mathsf{\Lambda}_{#1}}

\newcommand{\ct}{%
  \mathchoice{\mathbin{\raisebox{0.5ex}{$\displaystyle\centerdot$}}}%
             {\mathbin{\raisebox{0.5ex}{$\centerdot$}}}%
             {\mathbin{\raisebox{0.25ex}{$\scriptstyle\,\centerdot\,$}}}%
             {\mathbin{\raisebox{0.1ex}{$\scriptscriptstyle\,\centerdot\,$}}}
}

\def\homtwoshort#1(#2,#3,#4,#5,#6,#7){\hom_{#1}^2(#5,#6;#7)}
\def\homtwo#1{\hom_{#1}^2\homtwoarg1}

\def\homtwoarg#1(#2,#3,#4,#5,#6,#7){\left(
    \begin{tikzpicture}[baseline,xscale=#1]
      \node (a) at (0,0) {$\scriptstyle #2$};
      \node (b) at (1,.3) {$\scriptstyle #3$};
      \node (c) at (2,0) {$\scriptstyle #4$};
      \draw (a) -- node [auto] {$\scriptstyle #5$} (b);
      \draw (b) -- node [auto] {$\scriptstyle #6$} (c);
      \draw (a) -- node [auto,swap] {$\scriptstyle #7$} (c);
    \end{tikzpicture}\right)}

\def\nat#1#2{\underset{#1\to#2}{\hom}}
\def\nattwo#1#2{\underset{#1\to#2}{\hom}^2\homtwoarg1}

\newcommand{\lam}[1]{\lambda #1.\,}

\let\extlam\lam
\newcommand{\extapp}[2]{{#1}({#2})}

\newcommand{\idfunc}[1]{\mathsf{Id}_{#1}}

\newcommand{\two}{\mathbb{2}}
\newcommand{\idarr}[1]{\mathsf{id}_{#1}}
\newcommand{\iscomp}[2]{\mathsf{comp}_{#1,#2}}
\newcommand{\isiso}[1]{\mathsf{isiso}(#1)}

\newcommand{\isequiv}[1]{\mathsf{isEquiv}\Parens{#1}}

\newcommand{\transport}[3]{\mathsf{transport}^{#1}(#2,#3)}
\newcommand{\ap}{\mathsf{ap}}                             

\newcommand{\covtr}[1]{{#1}_*}  
\newcommand{\istrans}[2]{\mathsf{trans}_{#1,#2}}

\newcommand{\Parens}[1]{\Bigl(#1\Bigr)}
\let\xto\xrightarrow

\DeclareMathAlphabet{\mathbbe}{U}{bbold}{m}{n}

\newcommand{\Set}{\mathsf{Set}}
\newcommand{\sSet}{\mathsf{sSet}}
\newcommand{\ssSet}{\mathsf{ssSet}}
\newcommand{\DDelta}{\mathbbe{\Delta}}
\newcommand{\op}{\mathrm{op}}
\newcommand{\C}{\mathcal{C}}
\newcommand{\T}{\mathcal{T}}
\newcommand{\M}{\mathcal{M}}
\newcommand{\N}{\mathcal{N}}
\newcommand{\V}{\mathcal{V}}

\newcommand{\join}{\star}
\newcommand{\pojoin}{\mathbin{\widehat{\join}}}

\newcommand{\dI}{\mathbb{I}}
\newcommand{\fT}{\mathfrak{T}}

\newcommand{\maybe}[1]{\text{?`}\qquad #1 \qquad \text{?}}

\title{A type theory for synthetic $\infty$-categories}
\author{Emily Riehl and Michael Shulman}

\address{Dept.~of Mathematics\\Johns Hopkins U.\\ 3400 N Charles St.\\ Baltimore, MD 21218}
\email{eriehl@math.jhu.edu}

\address{Dept.~of Mathematics\\University of San Diego\\5998 Alcal\'{a} Park\\San Diego, CA 92110}
\email{shulman@sandiego.edu}

\thanks{This material is based on research sponsored by The United States Air Force Research Laboratory under agreement number FA9550-15-1-0053.  The U.S. Government is authorized to reproduce and distribute reprints for Governmental purposes notwithstanding any copyright notation thereon.  The views and conclusions contained herein are those of the author and should not be interpreted as necessarily representing the official policies or endorsements, either expressed or implied, of the United States Air Force Research Laboratory, the U.S. Government, or Carnegie Mellon University.  The first-named author is also grateful for support from the National Science Foundation via grants DMS-1551129 and DMS-1619569, the latter supporting a conference after which some of this work took place. A valuable discussion of the results of this paper took place at the 2017 Mathematics Research Community in Homotopy Type Theory coordinated by the AMS and supported by the National Science Foundation under grant number DMS-1321794. The post-publication v5 corrects typos pointed out by Bastiaan Cnossen, Nikolai Kudasov, Dan Licata, and Christian Sattler.}
\begin{document}
\maketitle

\begin{abstract}
  We propose foundations for a synthetic theory of $(\infty,1)$-categories within homotopy type theory.
  We axiomatize a directed interval type, then define higher simplices from it and use them to probe the internal categorical structures of arbitrary types.
  We define \emph{Segal types}, in which binary composites exist uniquely up to homotopy; this automatically ensures composition is coherently associative and unital at all dimensions.
  We define \emph{Rezk types}, in which the categorical isomorphisms are additionally equivalent to the type-theoretic identities --- a ``local univalence'' condition.
  And we define \emph{covariant fibrations}, which are type families varying functorially over a Segal type, and prove a ``dependent Yoneda lemma'' that can be viewed as a directed form of the usual elimination rule for identity types.
  We conclude by studying homotopically correct adjunctions between Segal types, and showing that for a functor between Rezk types to have an adjoint is a mere proposition.

  To make the bookkeeping in such proofs manageable, we use a three-layered type theory with shapes, whose contexts are extended by polytopes within directed cubes, which can be abstracted over using ``extension types'' that generalize the path-types of cubical type theory.
  In an appendix, we describe the motivating semantics in the Reedy model structure on bisimplicial sets, in which our Segal and Rezk types correspond to Segal spaces and complete Segal spaces.
\end{abstract}

\tableofcontents

\section{Introduction}
\label{sec:introduction}

Homotopy type theory~\cite{hottbook} is a new subject that augments Martin-L\"{o}f constructive dependent type theory with additional rules and axioms enabling it to be used as a formal language for reasoning about homotopy theory.
These rules and axioms are motivated by homotopy-theoretic models such as Voevodsky's simplicial set model~\cite{simplicial-model}.
In the latter, the \emph{types} of type theory are interpreted as simplicial sets in the Quillen model structure, which are a presentation of $\infty$-groupoids.
Thus, homotopy type theory can be viewed as a ``synthetic theory of $\infty$-groupoids'' and a foundational system for higher-categorical mathematics.

Of course, higher category theory is not just about $\infty$-groupoids, but also $n$-categories, $(\infty,1)$-categories, $(\infty,n)$-categories, and so on.
But a \emph{directed} type theory that could serve as a synthetic theory of such objects has proven somewhat elusive.
In particular, one of the advantages of homotopy type theory is that the single simple rule of identity-elimination automatically generates all the higher structure of $\infty$-groupoids, whereas (for instance) the 2-dimensional type theory of~\cite{2dtt} has to put in the categorical structure by hand, and thereby lacks as much advantage over explicit definitions of categories inside set theory.
Moreover, interpreting types directly as (higher) categories runs into various problems, such as the fact that not all maps between categories are exponentiable (so that not all $\Pi$-types exist), and that there are numerous different kinds of ``fibrations'' given the various possible functorialities and dimensions of categories appearing as fibers.

There is no reason in principle to think these problems insurmountable, and many possible solutions have been proposed.
However, in this paper we pursue a somewhat indirect route to a synthetic theory of higher categories, which has its own advantages, and may help illuminate some aspects of what an eventual more direct theory might look like.
Our approach is based on the following idea, which was also suggested independently by Joyal.

Homotopy type theory admits semantics not only in simplicial sets (hence $\infty$-groupoids), but in many other model categories.
In particular, as shown in~\cite{elreedy}, it can be interpreted in the Reedy model structure on \emph{bisimplicial} sets, also called \emph{simplicial spaces}.
This model structure, in turn, admits a left Bousfield localization called the \emph{complete Segal space} model structure~\cite{css}, which presents the homotopy theory and indeed also the category theory \cite{RV4} of $(\infty,1)$-categories.
We cannot interpret homotopy type theory (in its usual form) in the complete Segal space model structure directly (due to its lack of right properness among other things), but we can interpret it in the Reedy model structure and identify internally \emph{some} types that correspond to complete Segal spaces.
That is, in contrast to ordinary homotopy type theory where the basic objects (types) are exactly the ``synthetic $\infty$-groupoids'', in our theory the basic objects (types) are something more general, inside of which we identify a class that we regard as ``synthetic $(\infty,1)$-categories''.

The identification of these ``category-like types'', and the study of their properties, depends on adding certain structure to homotopy type theory that is characteristic of the bisimplicial set model.
The fundamental such structure is a ``directed interval'' type, which (thinking categorically) we denote $\two$.
The homotopy theoretic analysis of Joyal and Tierney \cite{JT} suggests that it is productive to think of bisimplicial sets as having a  ``spatial'' direction and ``categorical'' direction; simplicial sets can then be embedded in the categorical direction as \emph{discrete} simplicial spaces or in the spatial direction as \emph{constant} simplicial spaces. The semantics of the ``directed interval'' $\two$ as a bisimplicial set is the simplicial interval $\Delta^1$, placed in the ``categorical'' direction rather than the ``spatial'' direction.\footnote{Note that the ``spatial'' $\Delta^1$ is (weakly) contractible, whereas the ``categorical'' $\Delta^1$ is not.}
As it does in ordinary category theory, the directed interval detects arrows representably: that is, for any type $A$ the function type $\two\to A$ is the ``type of arrows in $A$''.

The directed interval $\two$ possesses a lot of useful structure.
The \emph{internal} incarnation of this structure, which is what is visible in the homotopy type theory of bisimplicial sets, is nicely summarized by saying that it is a \emph{strict interval}: a totally ordered set with distinct bottom and top elements (called $0$ and $1$).
In fact, there is a sense in which it possesses ``exactly this structure and no more'': the topos of simplicial sets is the \emph{classifying topos} for such strict interval objects.\footnote{This result was apparently first announced by Joyal at the Isle of Thorns; proofs can be found in~\cite{top-topos,maclane-moerdijk}.}
If we regard this classifying topos as sitting inside bisimplicial sets in the categorical direction (discrete in the spatial direction), then it is not hard to show that bisimplicial sets similarly present the ``classifying $(\infty,1)$-topos'' of strict intervals; but we will have no need of this.

The strict interval structure on $\two$ (i.e.\ $\Delta^1$) allows us to define the higher simplices from it internally, and hence the higher categorical structure of types.
For instance, \[\Delta^2 = \{ (s,t) : \two\times \two \mid t\le s \}.\] We regard a map $\alpha:\Delta^2\to A$ as a ``commutative triangle'' in $A$ witnessing that the composite of $\lam{t}\alpha(t,0):\Delta^1\to A$ and $\lam{t}\alpha(1,t):\Delta^1\to A$ is $\lam{t}\alpha(t,t):\Delta^1 \to A$.

Importantly, for a general type $A$, two given composable arrows --- i.e.\ two functions $f,g:\two\to A$ with $f(1)=g(0)$ --- may not have any such ``composite'', or they may have more than one.
If any two composable arrows have a unique composite in the homotopical sense that the type of such composites with their witnesses is contractible, we call $A$ a \emph{Segal type}.

Classically, a \emph{Segal space} is defined as a bisimplicial set $X$ for which \emph{all} the Segal maps $X_n \to X_1\times_{X_0}\cdots \times_{X_0} X_1$ are equivalences, thereby saying not only that any two composable arrows have a unique composite, but that any \emph{finite string} of composable arrows has a unique composite.
This ensures that composition is associative and unital up to all higher homotopies.
Our definition of Segal type appears to speak only about composable \emph{pairs}, but because it is phrased in the internal type theory of bisimplicial sets, semantically it corresponds to asserting not just that the Segal map $X_2 \to X_1\times_{X_0} X_1$ is an equivalence of \emph{simplicial sets}, but that the analogous map $X^{\Delta^2} \to X^{\Delta^1} \times_X X^{\Delta^1}$ is an equivalence of \emph{bisimplicial sets}.
Joyal conjectured that this is equivalent to the usual definition of a Segal space; in an appendix we prove this conjecture, justifying our terminology.
We can also prove internally in the type theory that composition in a Segal type is automatically associative and so on, so that it behaves just like a category.

Note the strong similarity to how ordinary homotopy type theory functions as a synthetic language for $\infty$-groupoids.
An explicit $\infty$-groupoid is a very complicated structure, but when working ``internally'' it suffices to equip every type with the single operation of identity-elimination.
It then automatically follows, as a meta-theorem, that every type \emph{internally} admits all the structure of an $\infty$-groupoid, as shown in~\cite{vdbg:oogpds,pll:oogpds}; but in practical applications we rarely need more than one or two levels of this structure, and we can just ``define it as we go''.
Similarly, a Segal space or $(\infty,1)$-category is a complicated structure with all higher coherences, but when working ``internally'' it suffices to assume a single contractibility condition to define a Segal type.
We do not prove an analogue of~\cite{vdbg:oogpds,pll:oogpds} for Segal types, but we conjecture that it should be possible; while in practice we generally seem to only need one or two levels that we can ``define as we go''.

If a Segal type satisfies a further condition analogous to Rezk's ``completeness'' condition for Segal spaces, we call it a \emph{Rezk type}.
These are the ones that semantically model $(\infty,1)$-categories.
However, for much of the theory it suffices to work with Segal types, which also have an $(\infty,1)$-categorical interpretation: they correspond to an $(\infty,1)$-category $\mathcal{A}$ equipped with a functor $\mathcal{G}\to \mathcal{A}$ where $\mathcal{G}$ is an $\infty$-groupoid.
In~\cite{fibrations} this is called a \emph{flagged $(\infty,1)$-category}; it can also be thought of as an ``$(\infty,1)$-double category'' with ``connections'' and one direction invertible.
The Rezk types correspond to the flagged $(\infty,1)$-categories for which $\mathcal{G}$ is the \emph{core} of $\mathcal{A}$, the locally full sub-$(\infty,1)$-category of invertible morphisms.
Note that the need for a ``completeness'' condition, or equivalently the fact that flagged categories must be defined before unflagged ones, also arises when \emph{defining} categories and higher categories inside homotopy type theory; see~\cite[Chapter 9]{hottbook} or~\cite{aks}.

The goal of this paper is to develop the basic category theory of Segal and Rezk types.
We discuss the behavior of ``functors'', which internally are simply functions between such types, and ``natural transformations'', which are simply functions $A\times \two\to B$.
We define what it means for a type family $C:A\to \univtype$ to be \emph{covariant} or \emph{contravariant}, and we prove a ``dependent Yoneda lemma'' that generalizes the usual Yoneda lemma and has the form of a ``directed'' version of the usual identity-elimination rule.

Many of the theorems are very similar to their versions in ordinary category theory and/or other forms of $(\infty,1)$-category theory.
In particular, when interpreted in the simplicial spaces model, our synthetic Yoneda lemma provides new proofs of the results that \cite{RV4,kv:yoneda-css,boavida:segr,rasekh:yoneda-ss} achieve semantically by working with simplicial spaces.
But often there is a significant ``internalization'' benefit, arising from the fact that all type-theoretic functions between Segal types are automatically ``functorial'' or ``natural''.
In this sense our theory achieves much of the expected benefit of a ``directed homotopy type theory'' for studying $(\infty,1)$-categories synthetically, with the added advantage that we have the full power of ordinary homotopy type theory to work with (including, for instance, all $\Pi$-types) and can draw on all of its results.
The presence of non-Segal types, whose category-theoretic meaning is somewhat unclear but which we can ignore whenever we wish, seems a small price to pay.

As a ``capstone'' application, we study \emph{adjunctions} between Segal and Rezk types, proving the equivalence of a diagrammatic ``unit and counit'' definition with an ``equivalence of homs'' definition.
As shown in~\cite{RVadj}, while a complete homotopy-coherent (diagrammatic) adjunction contains infinitely much data, it is uniquely determined by many finite subcollections of data {whenever the adjoint relationship exists,} such as: a single functor; both functors and the counit; both the unit and counit and a witness of one triangle identity; or witnesses of both triangle identities and a coherence between them (in the last case, no further existence assumptions are required).
We show that when transferred to the ``equivalence of homs'' definition, such subcollections correspond to the finitary coherent definitions of equivalence in homotopy type theory from~\cite[Chapter 4]{hottbook}.
Transferring the ``bi-invertibility'' definition of equivalence back across this comparison leads to a new way to fully characterize a homotopy coherent adjunction (with no further assumptions): two functors, a unit, and \emph{two} counits, one equipped with a witness that it satisfies one triangle identity and the other equipped with a witness of the other triangle identity.

There is one further technical device we will use, which is of some interest in its own right.
In principle, all of the above theory could be developed within ordinary homotopy type theory, simply by axiom\-atically assuming the type $\two$ and its strict interval structure.
However, we often want to talk about, given two points $x,y:A$, the ``type of arrows from $x$ to $y$'', i.e.\ the type of functions $f:\two\to A$ such that $f(0)=x$ and $f(1)=y$.
If we define this type internally in ordinary homotopy type theory, these latter equalities can only be points of the identity type, so we would have to define
\[ \hom_A(x,y) \defeq \sum_{f:\two\to A} (x=f(0)) \times (f(1)=y). \]
These equalities are then data, which have to be carried around everywhere.
This is quite tedious, and the technicalities become nearly insurmountable when we come to define commutative triangles, let alone commutative tetrahedra.

Intuitively, we would like $\hom_A(x,y)$ to be the type of functions $f:\two\to A$ such that $f(0)$ and $f(1)$ are \emph{strictly}, or \emph{judgmentally}, equal to $x$ and $y$ respectively.
Ordinary intensional type theory does not allow us to assert judgmental equalities as data, and the semantic reason for this is that it would not preserve fibrancy: judgmental equality on $A$ is interpreted by the diagonal $A\to A\times A$, which is not a fibration (unlike the path-object $P A \to A\times A$, which interprets the identity type).

However, in our motivating model of bisimplicial sets, the ``object of functions $f:\two\to A$ such that $f(0)\jdeq x$ and $f(1)\jdeq y$ strictly'' \emph{is} fibrant, because the inclusion $\booltype\to\two$ is a \emph{cofibration} and the Reedy model structure is cartesian monoidal.
The latter ensures that for any cofibration $A\to B$ and fibration $C\to D$, the ``pullback corner map'' or ``Leibniz hom''
\[ C^B \to C^A \times_{D^A} D^B \]
is a fibration.
Applied to the cofibration $\booltype\to \two$ and the fibration $A\to \unittype$ we obtain a fibration $A^\two \to A\times A$ representing the desired type family $\hom_A : A\times A \to \univtype$.

It is therefore natural to try to ``internalize'' this argument.
There are many possible ways to do this.
One ``brute force'' approach is to use a two-level type theory~\cite{hts,ack:two-level} in which there are both ``fibrant types'' and ``non-fibrant types'', with a non-fibrant ``strict equality type'' that reifies judgmental equality.
We could then define
\[ \hom_A(x,y) \defeq \sum_{f:\two\to A} (x\jdeq f(0)) \times (f(1)\jdeq y). \]
using strict equalities, and assert axiomatically that it is fibrant, since in general it would not be.

We will use instead a more refined approach that eliminates the need for strict equality and non-fibrant types.
We have a judgmental notion of \emph{cofibration}, and a new type former called an \emph{extension type}: if $i:A\rightarrowtail B$ is a cofibration and $C:B\to\univtype$ is a type family over its codomain with a section $d:\prod_{x:A} C(i(x))$ over its domain, then there is a type $\exten{y:B}{C(y)}{i}{d}$ of ``dependent functions $f:\prod_{y:B} C(y)$ such that $f(i(x))\jdeq d(x)$ for all $x:A$''.
This idea is due to Lumsdaine and the second author (unpublished).

We then have to give rules for what counts as a cofibration, in which we have to be careful to respect the semantics: it cannot simply be a map in any context that becomes a cofibration in the semantic slice category, since arbitrary slice categories are no longer cartesian monoidal model categories.
However, we need not only $\booltype\to\two$ to be a cofibration, but also the inclusion of the boundary of any simplex $\partial \Delta^n \to \Delta^n$, and we would like these to be constructible in a sensible and uniform way rather than axiomatically asserted.
One approach would be to keep the non-fibrant types with a notion of ``strict pushout'', and rules that cofibrations are closed under operations such as the ``pushout product'' or ``pushout join''.

We instead choose to keep all types fibrant (and hence all proofs more clearly homotopy-invariant), introducing rather a syntax for specifying cofibrations entirely separately from the rest of the type theory.
Pleasingly, this separate syntax is exactly the coherent theory of a strict interval.
We have a judgmental notion of \emph{shape}, representing the polytopes embedded in directed cubes that can be constructed in the theory of a strict interval, and we take the cofibrations to be the ``inclusions of sub-shapes''.
For instance, the boundary of $\Delta^2$ is the shape
\[ \partial\Delta^2 \defeq \sh{\pair{s,t}:\two\times\two}{(t\le s) \land (t\jdeq 0 \lor s\jdeq t \lor s\jdeq 1)} \]

This choice also makes the setup more flexible, since in principle any other suitable theory could be used instead.
For instance, using Joyal's theory of disks~\cite{disks} would presumably yield a type theory in which to study $(\infty,n)$-categories in the style of~\cite{rezk-theta}.
In an appendix we sketch how our setup should be interpretable semantically in bisimplicial sets.

\begin{rmk}
  Formally, our theory is very similar to the recent ``cubical type theories'' studied by~\cite{CCHM} and others, whose basic setup can also be regarded as an instance of ours, using the theory of a de Morgan algebra.
  The most substantial difference is that our interval $\two$ describes extra structure in an ``orthogonal'' direction to the native ``homotopy theory'' of homotopy type theory, whereas the cubical interval is rather a different way of describing that exact same native homotopy theory.
  This is why cubical type theory also includes the cubical Kan operations as rules of type theory; the closest analogue of this in our theory is the category structure of a Segal type induced by the contractibility of its composition spaces.
\end{rmk}

We introduce our basic type theory with shapes in \cref{sec:shape-type-theory}, and specialize to the simplicial type theory using the strict interval in \cref{sec:simplices}.
In \cref{sec:equiv-exten} we prove some basic results about extension types and how they commute with each other and with the other type constructors.

Then in \cref{sec:Segal-types} we give the basic definition of a Segal type and study is structure as a sort of ``category'', while in \cref{sec:2cat-segal} we study the corresponding behavior of ``functors'' and ``natural transformations''.
\Cref{sec:discrete-types} is devoted to a special kind of Segal type that we call ``discrete''; semantically these correspond to homotopically constant simplicial spaces; if Segal and Rezk types are the ``categories'', discrete types are the ``groupoids''.
Then in \cref{sec:covariant} we study covariant and contravariant type families, which are families of discrete types that vary functorially over a Segal type; these are the synthetic analogue of covariant and contravariant fibrations or presheaves.
In particular, they satisfy the Yoneda lemma, as we show in \cref{sec:yoneda-lemma}.

In \cref{sec:Rezk-types} we define Rezk types, which are Segal types satisfying a ``completeness'' or ``univalence'' condition identifying the type-theoretic identity type with the categorical isomorphisms.
And in \cref{sec:adjunctions} we study homotopy coherent adjunctions between Segal and Rezk types.

Finally, in \cref{sec:semantics} we briefly discuss the motivating semantics in bisimplicial sets and other ``model categories with shapes'', and show that our Segal and Rezk types correspond to Segal spaces and complete Segal spaces.
The analogous correspondence for covariant fibrations follows from recent work of~\cite{RV4,kv:yoneda-css,boavida:segr,rasekh:yoneda-ss}.

The authors with to thank the anonymous referee for a lengthy list of cogent suggestions and Arthur Azevedo de Amorim for catching a number of typos and imprecisions in the syntax for our type theory with shapes.

\section{Type theory with shapes}
\label{sec:shape-type-theory}

Our type theory has three layers.
The first two are basically ordinary coherent first-order logic, in which we express the theory of a strict interval; the third layer is then a homotopy type theory over the first two.
For clarity and generality, in this section we describe only the formal apparatus of the type theory; in \cref{sec:simplices} we will then add to it the axioms of a strict interval that we will use in the rest of the paper.

\subsection{Cubes, topes, shapes, and types}

The first layer is a simple intuitionistic type theory with finite product types and nothing else.
We call the types in this layer \textbf{cubes}; in our theory they will be finite powers of the one ``generating cube'' $\two$.
The formal rules for the cube layer are shown in \cref{fig:cubes}; here $\Xi$ is a context of variables belonging to cubes.

\begin{figure}
  \centering
  \begin{mathpar}
    \inferrule{ }{\unittype \cube}\and
    \inferrule{I\cube \\ J\cube}{I\times J\cube}\and
    \inferrule{(t:I)\in \Xi}{\Xi\types t:I}\and
    \inferrule{ }{\Xi\types\star:\unittype}\\
    \inferrule{\Xi\types s:I \\ \Xi\types t:J}{\Xi\types \pair{s,t}:I\times J}\and
    \inferrule{\Xi\types t:I\times J}{\Xi\types \pi_1(t):I}\and
    \inferrule{\Xi\types t:I\times J}{\Xi\types \pi_2(t):J}\and
  \end{mathpar}
  \caption{The cube layer}
  \label{fig:cubes}
\end{figure}

The second layer is an intuitionistic logic over the first.
We refer to its types as \textbf{topes}, thinking of them as polytopes embedded in a cube (namely, the ``cube context'' $\Xi$).
Topes admit operations of finite conjunction and disjunction, but not negation, implication, or either quantifier.\footnote{We could probably include the existential quantifier to obtain a full ``coherent logic'', and possibly even go beyond this, but for the theory of a strict interval we only need conjunction and disjunction.}
There is also a basic ``equality tope'', which we write with the symbol $\jdeq$, since it will be visible to the third layer as a ``strict'' or ``judgmental'' equality.
(In the theory of a strict interval introduced in \cref{sec:strict-interval}, there will also be an \emph{inequality} tope.)

\begin{figure}
  \centering
  \begin{mathpar}
    \inferrule{\phi\in \Phi}{\Xi\mid\Phi\types \phi}\and
    \inferrule{ }{\Xi\types \top\tope}\and
    \inferrule{ }{\Xi\mid \Phi \types \top}\and
    \inferrule{ }{\Xi\types \bot\tope}\and
    \inferrule{\Xi\mid\Phi\types \bot}{\Xi\mid\Phi\types \psi}\and
    \inferrule{\Xi\types \phi\tope \\ \Xi\types \psi\tope}{\Xi\types (\phi\land\psi)\tope}\and
    \inferrule{\Xi\mid\Phi\types \phi \\ \Xi\mid\Phi\types \psi}{\Xi\mid\Phi\types \phi\land\psi}\and
    \inferrule{\Xi\mid\Phi\types \phi\land\psi}{\Xi\mid\Phi\types \phi}\and
    \inferrule{\Xi\mid\Phi\types \phi\land\psi}{\Xi\mid\Phi\types \psi}\and
    \inferrule{\Xi\types \phi\tope \\ \Xi\types \psi\tope}{\Xi\types (\phi\lor\psi)\tope}\and
    \inferrule{\Xi\mid\Phi\types \phi}{\Xi\mid\Phi\types \phi\lor\psi}\and
    \inferrule{\Xi\mid\Phi\types \psi}{\Xi\mid\Phi\types \phi\lor\psi}\and
    \inferrule{\Xi\mid\Phi,\phi\types \chi \\ \Xi\mid\Phi,\psi\types\chi \\ \Xi\mid\Phi\types \phi\lor\psi}{\Xi\mid\Phi\types\chi}\and
    \inferrule{\Xi\types s:I \\ \Xi\types t:I}{\Xi\types (s\jdeq t)\tope}\and
    \inferrule{\Xi\types s:I}{\Xi\mid\Phi\types (s\jdeq s)}\and
    \inferrule{\Xi\mid\Phi\types (s\jdeq t)}{\Xi\mid\Phi\types (t\jdeq s)}\and
    \inferrule{\Xi\mid\Phi\types (s\jdeq t) \\ \Xi\mid\Phi\types (t\jdeq v)}{\Xi\mid\Phi\types (s\jdeq v)}\and
    \inferrule{\Xi\mid\Phi\types (s\jdeq t) \\ \Xi,x:I \types \psi\tope \\ \Xi\mid\Phi\types \psi[s/x]}{\Xi\mid\Phi\types \psi[t/x]}\and
    \inferrule{\Xi\types t:\unittype}{\Xi\mid\Phi\types t\jdeq\star}\and
    \inferrule{\Xi\types s:I \\ \Xi\types t:J}{\Xi\mid\Phi\types \pi_1(\pair{s,t}) \jdeq s}\and
    \inferrule{\Xi\types s:I \\ \Xi\types t:J}{\Xi\mid\Phi\types \pi_2(\pair{s,t}) \jdeq t}\and
    \inferrule{\Xi\types t:I\times J}{\Xi\mid\Phi\types t \jdeq \pair{\pi_1(t),\pi_2(t)}}\and
  \end{mathpar}
  \caption{The tope layer}
  \label{fig:topes}
\end{figure}

The formal rules of the tope layer are shown in \cref{fig:topes}; here $\Phi$ is a list of topes.
Note that we include the $\beta$ and $\eta$ rules for finite product cubes as introductions for equality topes.
We state all the rules in ``natural deduction style'', ensuring the admissibility of the usual structural rules like weakening, contraction, substitution, and cut.
For instance, here are the substitution rules:
\begin{mathpar}
  \inferrule{\Xi\types t:I \\ \Xi,x:I\types \phi\tope}{\Xi \types \phi[t/x] \tope}\and
  \inferrule{\Xi\types t:I \\ \Xi,x:I\mid\Phi\types \psi}{\Xi\mid\Phi[t/x] \types \psi[t/x]}\and
\end{mathpar}
and here is the cut rule for topes:
\[ \inferrule{\Xi\mid\Phi\types \psi \\ \Xi\mid\Phi,\psi\types\chi}{\Xi\mid \Phi\types \chi} \]

By a \textbf{shape} we will mean a cube together with a tope in the corresponding singleton context.
We could formalize this with a judgment and introduction rule such as the following:
\begin{equation}
  \inferrule{ I\cube \\ t:I \types \phi \tope}{\sh{t:I}{\phi} \shape}\label{eq:shape-judgment}
\end{equation}
The most important shapes for us will be the $n$-simplices and their boundaries and partial boundaries (such as horns).

Finally, the third layer is an ordinary intensional dependent type theory in which every judgment has additional contexts of cubes and topes.
All the usual type formers and rules leave these cube and tope contexts unchanged.
As in~\cite{hottbook}, we include $\Sigma$-types, $\Pi$-types with judgmental $\eta$-conversion, coproduct types, identity types $x : A, y : A \types x = y \type$, a universe $\univtype$ (but see \cref{sec:notation}), and so on.
We assume function extensionality as in~\cite[\S 2.9]{hottbook}, but we will not need any higher inductive types, nor the univalence axiom (although we expect that it, or at least directed analogues of it, will become important as the theory is developed further).

In addition, we have various rules that relate the first two layers to the third.
Firstly, we state all the rules in such a way that the following substitution/cut rules are admissible:
\begin{mathpar}
  \inferrule{\Xi\types t:I \\ \Xi,x:I\mid\Phi\mid \Gamma \types a:A}{\Xi\mid\Phi[t/x]\mid\Gamma[t/x] \types a[t/x] :A[t/x]}\and
  \inferrule{\Xi\mid\Phi\types \psi \\ \Xi\mid\Phi,\psi\mid \Gamma \types a:A}{\Xi\mid\Phi\mid \Gamma \types a:A}\and
\end{mathpar}
along with the obvious rules like weakening and contraction for the cube and tope contexts.

Secondly, we have rules ensuring that the type theory respects the ``tope logic'' in a strict judgmental way.
The appropriate sort of respect for $\top$ and $\land$ is already ensured by the cut and weakening rules.
For instance, we have the following derivations for $\land$:
\begin{mathpar}
  \inferrule*[Right=cut]{\inferrule*[Right=cut]{\inferrule*[Right=weak]{
        \Xi\mid\Phi,\phi,\psi\mid\Gamma\types a:A}
      {\Xi\mid\Phi,\phi\land\psi,\phi,\psi\mid\Gamma\types a:A}
    }{\Xi\mid\Phi,\phi\land\psi,\phi\mid\Gamma\types a:A}}
  {\Xi\mid\Phi,\phi\land\psi\mid\Gamma\types a:A}\and
  \inferrule*[Right=cut]{\inferrule*[Right=weak]{\inferrule*[Right=weak]{
      \Xi\mid\Phi,\phi\land\psi\mid\Gamma\types a:A}
    {\Xi\mid\Phi,\phi,\phi\land\psi\mid\Gamma\types a:A}}
  {\Xi\mid\Phi,\phi,\psi,\phi\land\psi\mid\Gamma\types a:A}}
{\Xi\mid\Phi,\phi,\psi\mid\Gamma\types a:A}
\end{mathpar}
But in the case of $\bot$ and $\lor$, we have to assert elimination and computation rules, as shown in \cref{fig:tope-or}.
Note that the rules for $\lor$ say that $\phi\lor\psi$ is a (strict) \emph{pushout} of $\phi$ and $\psi$ under $\phi\land\psi$, as is always the case in a coherent category.

\begin{figure}
  \centering
  \begin{mathpar}
    \inferrule{\Xi\mid\Phi\types \bot}{\Xi\mid\Phi\mid\Gamma\types \rec_\bot : A}\and
    \inferrule{\Xi\mid\Phi\types \bot \\ \Xi\mid\Phi\mid\Gamma\types a:A}{\Xi\mid\Phi\mid\Gamma\types a\jdeq \rec_\bot}\and
    \inferrule{
      \Xi\mid\Phi\types \phi\lor\psi \\
      \Xi\mid\Phi\mid\Gamma\types A\type \\
      \Xi\mid\Phi,\phi\mid\Gamma\types a_\phi:A \\
      \Xi\mid\Phi,\psi\mid\Gamma\types a_\psi:A \\
      \Xi\mid\Phi,\phi\land\psi\mid\Gamma\types a_\phi \jdeq a_\psi
    }{\Xi\mid\Phi\mid\Gamma\types \rec_\lor^{\phi,\psi}(a_\phi,a_\psi) : A}\and
    \inferrule{
      \Xi\mid\Phi\types \phi\lor\psi \\
      \Xi\mid\Phi\mid\Gamma\types A\type \\
      \Xi\mid\Phi,\phi\mid\Gamma\types a_\phi:A \\
      \Xi\mid\Phi,\psi\mid\Gamma\types a_\psi:A \\
      \Xi\mid\Phi,\phi\land\psi\mid\Gamma\types a_\phi \jdeq a_\psi
    }{\Xi\mid\Phi,\phi\mid\Gamma\types \rec_\lor^{\phi,\psi}(a_\phi,a_\psi) \jdeq a_\phi}\and
    \inferrule{
      \Xi\mid\Phi\types \phi\lor\psi \\
      \Xi\mid\Phi\mid\Gamma\types A\type \\
      \Xi\mid\Phi,\phi\mid\Gamma\types a_\phi:A \\
      \Xi\mid\Phi,\psi\mid\Gamma\types a_\psi:A \\
      \Xi\mid\Phi,\phi\land\psi\mid\Gamma\types a_\phi \jdeq a_\psi
    }{\Xi\mid\Phi,\psi\mid\Gamma\types \rec_\lor^{\phi,\psi}(a_\phi,a_\psi) \jdeq a_\psi}\and
    \inferrule{\Xi\mid\Phi\types \phi\lor\psi \\ \Xi\mid\Phi\mid\Gamma\types a:A}{\Xi\mid\Phi\mid\Gamma\types a \jdeq \rec_\lor^{\phi,\psi}(a,a)}\and
  \end{mathpar}
  \caption{Type elimination for tope disjunction}
  \label{fig:tope-or}
\end{figure}

We also require the following compatibility rule, saying that tope equality behaves like judgmental equality:
\begin{equation}
\inferrule{\Xi\mid\Phi\types (s\jdeq t) \\ \Xi,x:I\mid\Phi\mid\Gamma\types a:A}
{\Xi\mid\Phi\mid\Gamma[s/x] \types a[s/x]\jdeq a[t/x]}\label{eq:tope-eq-jdeq}
\end{equation}
Note that in the premise, $\jdeq$ refers to the equality tope in the second layer, while in the conclusion it refers to the judgmental equality of the third layer.
Also, inductively we have $\Gamma[s/x] \jdeq \Gamma[t/x]$, so both terms $a[s/x]$ and $a[t/x]$ in the conclusion are well-typed in the same context.

Unfortunately, because rule~\eqref{eq:tope-eq-jdeq} as stated has a not-fully-general context $\Gamma[s/x]$ in the conclusion, it breaks admissibility of substitution into judgmental equality:
\begin{equation}
  \inferrule{\Xi\mid\Phi\mid \Gamma \types a:A \\ \Xi\mid\Phi\mid\Gamma,x:A,\Delta\types b\jdeq b'}{\Xi\mid\Phi\mid\Gamma,\Delta[a/x]\types b[a/x]\jdeq b[a'/x]}\tag{$*$}\label{eq:sub-jdeq}
\end{equation}
There are various potential ways to resolve this, but since judgmental equality is proof-irrelevant and we are not considering implementation questions in this paper, the simplest is to take~\eqref{eq:sub-jdeq} as a primitive rule.
We thank Dan Licata for pointing this out.

\subsection{Extension types along cofibrations}

Finally, we come to the reason for introducing this whole three-layer theory: extension types along cofibrations.
As our notion of ``cofibration'' we use a \emph{shape inclusion}, i.e.\ a pair of shapes $\sh{t:I}{\phi}$ and $\sh{t:I}{\psi}$ in the same cube such that $t:I\mid\phi\types\psi$.
We will sometimes abbreviate this as $\sh{t:I}{\phi}\subseteq \sh{t:I}{\psi}$.

\begin{figure}
  \centering
  \begin{mathpar}
    \inferrule{\sh{t:I}{\phi} \shape \\ \sh{t:I}{\psi} \shape \\ t:I \mid \phi \types \psi \\
      \Xi\mid\Phi \types \Gamma \ctx \\
      \Xi,t:I \mid \Phi,\psi \mid \Gamma \types A\type \\
      \Xi,t:I \mid \Phi,\phi \mid \Gamma \types a:A
    }{\Xi\mid\Phi\mid\Gamma \types \exten{t:I \mid \psi}{A}{\phi}{a} \type}\and
    \inferrule{\sh{t:I}{\phi} \shape \\ \sh{t:I}{\psi} \shape \\ t:I \mid \phi \types \psi \\
      \Xi\mid\Phi \types \Gamma \ctx \\
      \Xi,t:I \mid \Phi,\psi \mid \Gamma \types A\type \\
      \Xi,t:I \mid \Phi,\phi \mid \Gamma \types a:A \\\\
      \Xi,t:I \mid \Phi,\psi \mid \Gamma \types b:A \\
      \Xi,t:I \mid \Phi,\phi \mid \Gamma \types b\jdeq a
    }{\Xi\mid\Phi\mid\Gamma \types \lam{t^{I\mid\psi}} b : \exten{t:I \mid \psi}{A}{\phi}{a}}\and
    \inferrule{\sh{t:I}{\phi} \shape \\ \sh{t:I}{\psi} \shape \\ t:I \mid \phi \types \psi \\\\
      \Xi\mid\Phi\mid\Gamma \types f:\exten{t:I \mid \psi}{A}{\phi}{a} \\
      \Xi\types s:I \\ \Xi\mid\Phi\types \psi[s/t]
    }{\Xi\mid\Phi\mid\Gamma \types f(s) : A}\and
    \inferrule{\sh{t:I}{\phi} \shape \\ \sh{t:I}{\psi} \shape \\ t:I \mid \phi \types \psi \\\\
      \Xi\mid\Phi\mid\Gamma \types f:\exten{t:I \mid \psi}{A}{\phi}{a} \\
      \Xi\types s:I \\ \Xi\mid\Phi\types \phi[s/t]
    }{\Xi\mid\Phi\mid\Gamma \types f(s) \jdeq a[s/t]}\and
    \inferrule{\sh{t:I}{\phi} \shape \\ \sh{t:I}{\psi} \shape \\ t:I \mid \phi \types \psi \\
      \Xi\mid\Phi \types \Gamma \ctx \\
      \Xi,t:I \mid \Phi,\psi \mid \Gamma \types A\type \\
      \Xi,t:I \mid \Phi,\phi \mid \Gamma \types a:A \\\\
      \Xi,t:I \mid \Phi,\psi \mid \Gamma \types b:A \\
      \Xi,t:I \mid \Phi,\phi \mid \Gamma \types b\jdeq a\\
      \Xi\types s:I \\ \Xi\mid\Phi\types \psi[s/t]
    }{\Xi\mid\Phi\mid\Gamma \types (\lam{t^{I\mid\psi}} b)(s) \jdeq b[s/t]}\and  
    \inferrule{\sh{t:I}{\phi} \shape \\ \sh{t:I}{\psi} \shape \\ t:I \mid \phi \types \psi \\\\
      \Xi\mid\Phi\mid\Gamma \types f:\exten{t:I \mid \psi}{A}{\phi}{a} \\
    }{\Xi\mid\Phi\mid\Gamma \types f \jdeq \lam{t^{I\mid\psi}} f(t)}\and
  \end{mathpar}
  \caption{Extension types}
  \label{fig:exten}
\end{figure}

The rules for extension types are shown in \cref{fig:exten}.
In the formation rule, the judgment $\Xi\mid\Phi \types \Gamma \ctx$ means that $\Gamma$ is a well-formed context of types relative to $\Xi\mid\Phi$.
The point is that $\Gamma$ is not allowed to depend on $t$ or $\psi$, and (implicitly) that $\Phi$ is also not allowed to depend on $t$.
The type $A$, however, is allowed to depend on $t$ and $\psi$, i.e.\ we allow ``dependent extensions''.
The rest of the rules say essentially that an extension type behaves like an ordinary dependent function type, with $\beta$ and $\eta$ rules, except that all its elements act like the supplied section $a:A$ whenever $\phi$ holds.

As with ordinary dependent function types, if the codomain type $A$ does not actually depend on the domain shape $\sh{t:I}{\psi}$, instead of $\exten{t:I \mid \psi}{A}{\phi}{a}$ we write $\ndexten{\sh{t:I}{\psi}}{A}{\phi}{a}$.

A different special case is when $\phi$ is $\bot$.
Then the section $a$ might as well be $\rec_\bot$, while all the required judgmental equalities also hold automatically by $\rec_\bot$, so the extension type behaves just like an ordinary (possibly dependent) function type whose domain is a shape and whose codomain is a type.
Thus, we omit the angle brackets, writing
\begin{align}
  \Parens{\tprod_{t:I|\psi} A} &\defeq \exten{t:I|\psi}{A}{\bot}{\rec_\bot}\label{eq:exten-bot}\\
  \Parens{\sh{t:I}{\psi} \to A} &\defeq \ndexten{\sh{t:I}{\psi}}{A}{\bot}{\rec_\bot}.\label{eq:ndexten-bot}
\end{align}

Having just introduced extension types and their notation, we now proceed to introduce an abuse of that notation.
The rules in \cref{fig:exten} are written in the usual formal type-theoretic way, with the dependent type $A$, tope $\phi$, and term $a:A$ being expression metavariables containing the variable $t:I$.
Note that the variable $t$ is bound in all three, i.e.\ its binding in $\prod_{t:I|\psi}$ scopes over the rest of the expression.

However, once we \emph{have} extension types, and when writing informally and internally to type theory (as we will do for most of the paper), it is more readable and natural to regard $A$ as a function into the universe and $a$ as a dependent function
\begin{mathpar}
  A:\sh{t:I}{\psi} \to \univtype\and
  a:\tprod_{t:I\mid\phi} A(t).
\end{mathpar}
(The types of $A$ and $a$ here are actually also extension types, with $\bot$ implicit as noted above.)
It is then natural to write the extension type as
\[ \exten{t:I\mid\psi}{A(t)}{\phi}{a} \qquad\text{or}\qquad \exten{t:I\mid\psi}{A(t)}{\phi}{\lam{t}a(t)}. \]
Once we introduce notations for important shapes, such as the simplices $\Delta^n$ (see \cref{sec:sub-simplices}), it will be natural to also use these in place of the tope $\phi$:
\[ \exten{t:I\mid\psi}{A(t)}{\sh{t:I}{\phi}}{a} \qquad\text{or}\qquad \exten{t:I\mid\psi}{A(t)}{\sh{t:I}{\phi}}{\lam{t}a(t)}. \]

\begin{rmk}\label{sec:notation}
This notation does technically require a universe type, to be the codomain of $A$.
Our primary motivating model of bisimplicial sets does have a universe by~\cite{elreedy}, but as we will see in \cref{sec:models-type-theory} there are other interesting models where universes are not known to exist.
However, our use of universes in this paper will be only for notational convenience; all the results could equally well be formulated in a type theory without universes.\footnote{For further development of the theory of synthetic $(\infty,1)$-categories, we expect various universes to be necessary though.}
\end{rmk}

\section{Simplicial type theory}
\label{sec:simplices}

This completes the specification of our general type theory with shapes and extension types.
As a special case, we now formulate the theory that we will work in for the rest of the paper, in which the cube and tope layers form the coherent theory of a strict interval. The simplices are then defined as particular shapes in this coherent theory.
In \cref{sec:Segal-types} we will define hom types by using these simplex shapes as ``probes''.

\subsection{The strict interval}
\label{sec:strict-interval}
To define the coherent theory of the strict interval we begin with the axiomatic cubes and terms:
\begin{mathpar}
  \two \cube \and
  0:\two \and
  1:\two
\end{mathpar}
and an axiomatic inequality tope
\begin{mathpar}
  x:\two, y:\two \types (x\le y) \tope
\end{mathpar}
together with the following strict interval axioms:
\begin{align*}
  x:\two \mid\cdot &\types (x\le x)\\
  x:\two, y:\two, z:\two \mid (x\le y), (y\le z) &\types (x\le z)\\
  x:\two, y:\two \mid (x\le y), (y\le x) &\types (x\jdeq y)\\
  x:\two, y:\two \mid\cdot &\types (x\le y) \lor (y\le x)\\
  x:\two \mid\cdot &\types (0\le x)\\
  x:\two \mid\cdot &\types (x\le 1)\\
  \cdot\mid (0\jdeq 1) &\types \bot
\end{align*}
(Technically, to maintain admissibility of cut, such axioms should be formulated as inference rules, but this version is much more readable.)

\begin{rmk}\label{rmk:duality}
  Note that this theory has a ``duality'' involution obtained by interchanging $0$ with $1$ and reversing the order of $\le$.
  Since the rules of our type theory are independent of the particular tope theory, it follows that our entire three-layer type theory has a duality involution.
  This is a \emph{syntactic} and ``meta-theoretic'' involution, transforming every proof into a dual proof; but it also corresponds semantically to the categorical duality involution on bisimplicial sets obtained by precomposing, in the ``categorical'' but not also in the ``spatial'' direction, with the functor $(-)^\circ \colon \DDelta \to \DDelta$ that reverses the direction on each ordinal.
  However, there can also be other models that lack such an involution, since it is not internalized in the syntax.
\end{rmk}

\begin{rmk}\label{rmk:cubical}
  The cubical type theory of~\cite{CCHM} can (at least approximately) be regarded as a different instance of our type theory with shapes, together with added ``Kan operations''.
  Instead of the coherent theory of a strict interval,~\cite{CCHM} uses the coherent theory of a nondegenerate de Morgan algebra, having the following axiomatic cubes and terms:
  \def\meet(#1,#2){#1\mathbin{\pmb{\land}}#2}
  \def\join(#1,#2){#1\mathbin{\pmb{\lor}}#2}
  \begin{mathpar}
    \dI\cube \and 0:\dI \and 1:\dI \\ t:\dI,s:\dI \types \join(t,s):\dI \and t:\dI,s:\dI \types \meet(t,s):\dI \and t:\dI \types \neg t:\dI.
  \end{mathpar}
  Note that here $\pmb{\land}$ and $\pmb{\lor}$ are cube term constructors denoting lattice operations, not to be confused with the logical conjunction and disjunction $\land$ and $\lor$ that act on topes.
  They are subject to the axioms of a distributive lattice:
  \begin{align*}
    t:\dI \mid\cdot &\types \join(t,t)\jdeq t\\
    t:\dI,s:\dI \mid\cdot &\types \join(t,s)\jdeq \join(s,t)\\
    t:\dI,s:\dI,u:\dI \mid\cdot &\types \join(t,(\join(s,u)))\jdeq \join((\join(t,s)),u)\\
    t:\dI \mid\cdot &\types \join(t,0)\jdeq t\\
    t:\dI \mid\cdot &\types \meet(t,t)\jdeq t\\
    t:\dI,s:\dI \mid\cdot &\types \meet(t,s)\jdeq \meet(s,t)\\
    t:\dI,s:\dI,u:\dI \mid\cdot &\types \meet(t,(\meet(s,u)))\jdeq \meet((\meet(t,s)),u)\\
    t:\dI \mid\cdot &\types \meet(t,1)\jdeq t\\
    t:\dI,s:\dI \mid \cdot &\types \join(t,(\meet(t,s))) \jdeq t\\
    t:\dI,s:\dI \mid \cdot &\types \meet(t,(\join(t,s))) \jdeq t\\
    t:\dI,s:\dI,u:\dI \mid\cdot &\types \join(t,(\meet(s,u))) \jdeq \meet((\join(t,s)),(\join(t,u)))\\
    \intertext{plus those of a de Morgan algebra:}
    t:\dI,s:\dI \mid\cdot &\types \neg(\meet(t,s)) \jdeq \join(\neg t, \neg s)\\
    t:\dI \mid\cdot &\types \neg\neg t\jdeq t\\
    \intertext{and finally distinctness of the top and bottom elements:}
    \cdot \mid (0\jdeq 1) &\types \bot.
  \end{align*}
  The cubical path-type $\mathsf{Path}_A(x,y)$ is the extension type $\exten{t:\dI}{A}{t\jdeq 0 \lor t\jdeq 1}{\rec_\lor(x,y)}$ analogous to our hom-types (see \cref{sec:hom-types}), and the ``face lattice'' $\mathbb{F}$ of~\cite{CCHM} corresponds to the coherent logic of topes, while the ``systems'' of~\cite{CCHM} correspond to the rules in \cref{fig:tope-or}.
  The de Morgan negation $\neg$ on $\dI$ yields a ``path reversal'' operation, while the minima and maxima operations yield path operations constructing ``connection squares'' such as
  \[
  \begin{tikzcd}
    x \arrow[r, "f"] \arrow[dr, "f" description] \arrow[d, "f"'] & y \arrow[d, equals] \\
    y \arrow[r, equals]  \arrow[ur, phantom, "\cdot" very near start, "\cdot" very near end] & y
  \end{tikzcd}
  \qquad\text{and}\qquad
  \begin{tikzcd}
    x \arrow[r, equals] \arrow[dr, "f" description] \arrow[d, equals] & x \arrow[d, "f"] \\
    x \arrow[r, "f"']  \arrow[ur, phantom, "\cdot" very near start, "\cdot" very near end] & y
  \end{tikzcd}
  \]
  The composition and Kan operations, however, are something extra in cubical type theory without any analogue in our general type theory with shapes; they force the cubical path-types to represent the ``internal'' homotopy theory of types, rather than moving in an ``orthogonal'' direction like our directed hom-types will.

  Of course, in our theory we do not want a negation, since it is explicitly supposed to be \emph{directed}.
  We do not include binary minima and maxima explicitly, but we could do so without changing the coherent theory, since a total order always has minima and maxima.
  However, as we will see, this same argument enables us to construct connection squares using $\rec_\lor$, so there is no need for minimum and maximum operations.

  There are many possible variations of ``cubical type theory'', corresponding to variations in the coherent theory chosen (for instance, leaving out negation and/or maxima and minima).
  There are also many possible variations of our directed type theory; for instance, the theory of \emph{discs} from~\cite{disks} should yield synthetic theories of $(\infty,n)$- or $(\infty,\infty)$-categories, corresponding semantically to the $\Theta$-spaces of~\cite{rezk-theta}.
  Moreover, we can combine theories: for instance, with a strict interval $\two$ and an unrelated de Morgan algebra $\dI$, we could obtain a ``directed cubical type theory'' in which the intrinsic homotopy theory is cubical but there is an additional orthogonal directed structure.
  Our type theory with shapes supplies a general context in which to investigate a large class of such theories.
\end{rmk}

\subsection{Simplices and their subshapes}
\label{sec:sub-simplices}

The interval type $\two$ allows us to define the \textbf{simplices} as the following shapes:\footnote{Formally, this should really be something like \[\sh{t:(\cdots((\two\times \two)\times \two)\cdots)}{\pi_2(t) \le \pi_2(\pi_1(t)) \le \cdots \le (\pi_1)^n(t)},\] but no problems will arise from this sort of abuse of notation.}
\[ \Delta^n \defeq \sh{\pair{t_1,\dots,t_n}:\two^n}{t_n\le\cdots\le t_1} \]
We note that this is a meta-theoretic definition, as there is internally no ``natural numbers'' that we can use to parametrize a ``family of shapes'' (there is a natural numbers \emph{type}, but shapes are not allowed to depend on types).
However, we will only need it for very small concrete values of $n$, in which case the meaning is clear:
\begin{align*}
  \Delta^0 &\defeq \sh{t:\unittype}{\top}\\
  \Delta^1 &\defeq \sh{t:\two}{\top}\\
  \Delta^2 &\defeq \sh{\pair{t_1,t_2}:\two\times\two}{t_2\le t_1}\\
  \Delta^3 &\defeq \sh{\pair{t_1,t_2,t_3}:\two\times\two\times\two}{t_3\le t_2 \le t_1}
\end{align*}
Moreover, as we shall discover in \cref{sec:join}, the higher dimensional simplices may be inductively defined from these lower dimensional ones by means of the join operation (though, again, the induction is external to the type theory).

\begin{rmk}
  The perhaps-surprising reversal of order in the coordinates is chosen so that $t_i$ parametrizes the $i^{\mathrm{th}}$ arrow in the spine of a simplex.
  For instance, in a 3-simplex $f:\Delta^3 \to A$ with the following boundary:
  \[
  \begin{tikzcd}
    & \cdot\ar[dr,"f_{23}"] \ar[dd,"f_2" description] &&&& \cdot \ar[dr,"f_{23}"]\\
    \cdot\ar[ur,"f_{1}"] \ar[dr,"f_{12}"'] && \cdot & \Rightarrow & \cdot\ar[ur,"f_{1}"] \ar[dr,"f_{12}"'] \ar[rr,"f_{123}" description] && \cdot\\
    & \cdot\ar[ur,"f_3"'] &&&& \cdot\ar[ur,"f_3"']
  \end{tikzcd}
  \]
  we have
  \begin{align*}
    f_1(t) &\jdeq f(t,0,0)\\
    f_2(t) &\jdeq f(1,t,0)\\
    f_3(t) &\jdeq f(1,1,t).
  \end{align*}
  The other three 1-simplices are given by
  \begin{align*}
    f_{12}(t) &\jdeq f(t,t,0)\\
    f_{23}(t) &\jdeq f(1,t,t)\\
    f_{123}(t) &\jdeq f(t,t,t).
  \end{align*}
\end{rmk}

The other face and degeneracy operations between simplices can be defined in analogous ways.
For instance, the four 2-simplex faces of a 3-simplex are obtained by requiring $0\jdeq t_3$, $t_3\jdeq t_2$, $t_2\jdeq t_1$, and $t_1\jdeq 1$ respectively.
These yield operations on extension types:
\begin{align*}
  \lam{f} \lam{\pair{t_1,t_2}} f\pair{t_1,t_2,0} &: (\Delta^3 \to A) \to (\Delta^2 \to A)\\
  \lam{f} \lam{\pair{t_1,t_2}} f\pair{t_1,t_2,t_2} &: (\Delta^3 \to A) \to (\Delta^2 \to A)\\
  \lam{f} \lam{\pair{t_1,t_2}} f\pair{t_1,t_1,t_2} &: (\Delta^3 \to A) \to (\Delta^2 \to A)\\
  \lam{f} \lam{\pair{t_1,t_2}} f\pair{1,t_1,t_2} &: (\Delta^3 \to A) \to (\Delta^2 \to A).
\end{align*}
Similarly, the two degenerate 2-simplices associated to a 1-simplex are given by ignoring one variable:
\begin{align*}
  \lam{f} \lam{\pair{t_1,t_2}} f(t_1) &: (\Delta^1 \to A) \to (\Delta^2 \to A)\\
  \lam{f} \lam{\pair{t_1,t_2}} f(t_2) &: (\Delta^1 \to A) \to (\Delta^2 \to A)
\end{align*}
and so on.

We will also use various sub-shapes of the simplices, particularly their \textbf{boundaries}:
\begin{align*}
  \partial\Delta^1 &\defeq \sh{t:\two}{(0\jdeq t) \lor (t\jdeq 1)}\\
  \partial\Delta^2 &\defeq \sh{\pair{t_1,t_2}:\two\times\two}{(0\jdeq t_2\le t_1) \lor (t_2\jdeq t_1) \lor (t_2\le t_1 \jdeq 1)}
\end{align*}
The elimination rules in \cref{fig:tope-or} ensure that terms depending on such a boundary can be ``glued together'' from terms depending on lower-dimensional simplices in the expected way.
For instance, to define a term $a:A$ in context $\partial\Delta^1$ (i.e.\ in context $t:\two \mid t\jdeq 0 \lor t\jdeq 1$), it is necessary and sufficient to give a term $a_0:A$ in context $t:\two \mid t\jdeq 0$ and a term $a_1:A$ in context $t:\two\mid t\jdeq 1$, such that if we assume $t\jdeq 0 \land t\jdeq 1$ then $a_0\jdeq a_1$.
But the last requirement is vacuous, since $t\jdeq 0 \land t\jdeq 1 \types \bot$ so that in that context both reduce to $\rec_\bot$.
Moreover, since tope equality acts like judgmental equality, assuming $t:\two$ and $t\jdeq 0$ is equivalent to assuming nothing at all, and similarly {for assuming $t\jdeq 1$.}

Thus, a term $a:A$ in context $\partial\Delta^1$ is equivalently two terms $a_0,a_1:A$ in no shape context, so that $\partial\Delta^1$ behaves like $\booltype$, the boolean type $\unittype + \unittype$.
Similarly, a term $a:A$ in context $\partial\Delta^2$ is equivalently three terms $a_0,a_1,a_2:A$ in context $t:\two$ such that $a_0[0/t] \jdeq a_1[0/t]$ and $a_0[1/t] \jdeq a_2[0/t]$ and $a_1[1/t] \jdeq a_2[1/t]$, i.e.\ a ``noncommutative triangle''.

More interestingly, $\Delta^1\times \Delta^1$ (i.e.\ the shape $\sh{t:\two\times\two}{\top}$) behaves like the pushout of two copies of $\Delta^2$ along their common diagonal boundary $\Delta^1_1 \defeq \sh{\pair{t,s} : \two \times \two}{t \jdeq s}$.
For since we have $t:\two,s:\two \types (t\le s)\lor (s\le t)$, a term $a:A$ in context $\Delta^1\times \Delta^1$ is equivalently a term $a_0:A$ in context $t:\two, s:\two \mid (t\le s)$ (which, up to tupling and permutation of variables, is just $\Delta^2$) and a term $a_1:A$ in context $t:\two, s:\two \mid (s\le t)$ (another copy of $\Delta^2$), such that if we assume $t\le s$ and $s\le t$ then $a_0\jdeq a_1$.
But $(t\le s),(s\le t)\types t\jdeq s$, so this context is just a copy of $\Delta^1$, embedded into the two copies of $\Delta^2$ as one of the boundary edges.

As an example application of this, recall that in \cref{rmk:cubical} we remarked that the cubical de Morgan algebra structure on shapes\footnote{Actually, a lattice structure suffices.} enables the construction of ``connection'' squares with the following faces, for any arrow $f$ from $x$ to $y$:
  \begin{equation}
  \begin{tikzcd}
    x \arrow[r, "f"] \arrow[dr, "f" description] \arrow[d, "f"'] & y \arrow[d, equals] \\
    y \arrow[r, equals]  \arrow[ur, phantom, "\cdot" very near start, "\cdot" very near end] & y
  \end{tikzcd}
  \qquad\text{and}\qquad
  \begin{tikzcd}
    x \arrow[r, equals] \arrow[dr, "f" description] \arrow[d, equals] & x \arrow[d, "f"] \\
    x \arrow[r, "f"']  \arrow[ur, phantom, "\cdot" very near start, "\cdot" very near end] & y
  \end{tikzcd}\label{eq:connection-squares}
  \end{equation}
We denote these squares by $\connmax f$ and $\connmin f$, respectively; in terms of the lattice operations $\pmb{\land}$ and $\pmb{\lor}$ they are defined by $\connmax f(t,s) = f(t\mathbin{\pmb{\lor}}s)$ and $\connmin f(t,s) = f(t\mathbin{\pmb{\land}}s)$.
Unlike in the cubical type theory of~\cite{CCHM}, we have not assumed $\mathbin{\pmb{\lor}}$ and $\mathbin{\pmb{\land}}$ as operations in our algebra of cubes, but we can nevertheless construct $\connmax f$ and $\connmin f$ using $\rec_\lor$.
Essentially, this is because a totally ordered set is automatically a lattice.

\begin{prop}\label{prop:connections}
  For any $f:\two\to A$, we have squares $\connmax f,\connmin f:\two\times\two\to A$ with the faces displayed in \eqref{eq:connection-squares}, i.e.\ such that
  \begin{alignat*}{2}
    \connmax f(0,s) &\jdeq f(s) &\qquad \connmin f(0,s) &\jdeq f(0)\\
    \connmax f(t,0) &\jdeq f(t) &\qquad \connmin f(t,0) &\jdeq f(0)\\
    \connmax f(1,s) &\jdeq f(1) &\qquad \connmin f(0,s) &\jdeq f(s)\\
    \connmax f(t,1) &\jdeq f(1) &\qquad \connmin f(t,0) &\jdeq f(t)\\
    \connmax f(t,t) &\jdeq f(t) &\qquad \connmin f(t,t) &\jdeq f(t).
  \end{alignat*}
\end{prop}
\begin{proof}
  We define
  \begin{align*}
    \connmax f(t,s) &\defeq \rec_\lor^{t\le s, s\le t}(f(s),f(t))\\
    \connmin f(t,s) &\defeq \rec_\lor^{t\le s, s\le t}(f(t),f(s)).
  \end{align*}
  In both cases, if $t\le s$ and $s\le t$, then $t\jdeq s$ and so $f(s)\jdeq f(t)$, so the compatibility condition is satisfied.
  Geometrically, $\connmax f$ glues two copies of the degenerate 2-simplex $\lam{t}\lam{s} f(t)$ along their common 1-face, while $\connmin f$ similarly glues two copies of the other degenerate 2-simplex $\lam{t}\lam{s} f(s)$.
\end{proof}

As a second application, we observe that, at least as far as maps out of it are concerned, we may suppose $\Delta^2$ to be a retract of $\Delta^1\times\Delta^1$.
This will be useful in a number of places to deduce 2-simplex information from 1-simplex assumptions.

\begin{prop}\label{prop:two-simp-as-retract}
  For any type $A$, the type $\Delta^2\to A$ is a retract of $\Delta^1\times\Delta^1\to A$.
\end{prop}
\begin{proof}
  The retraction is easy:\footnote{The apparently trivial $\eta$-expansion serves to create an element of an extension type with different domain.}
  \[ \lam{f} \lam{\pair{t,s}} f(t,s) : (\Delta^1\times\Delta^1\to A) \to (\Delta^2\to A). \]
  The section is where we have to use $\rec_\lor$:
  \[ \lam{f} \lam{\pair{t,s}} \rec_{\lor}^{t\le s,s\le t}(f(t,t),f(t,s)) : (\Delta^2\to A)\to (\Delta^1\times\Delta^1\to A) \]
  Again, if $t\le s$ and $s\le t$ then $t\jdeq s$ so $f(t,t)\jdeq f(t,s)$, so the compatibility condition holds.
  And if $\pair{t,s}:\Delta^2$ then $s\le t$, so the composite of section followed by retraction is the identity.
\end{proof}

Similar arguments apply in higher dimensions.
For instance, the 3-dimensional ``prism'' $\Delta^2\times\Delta^1 \jdeq \sh{\pair{\pair{t_1,t_2},t_3}}{t_2\le t_1}$ can be written as the union of three 3-simplices
\begin{align*}
  \Delta^3 &= \sh{\pair{\pair{t_1,t_2},t_3}}{t_3 \le t_2\le t_1}\\
  \Delta^3 &= \sh{\pair{\pair{t_1,t_2},t_3}}{t_2 \le t_3\le t_1}\\
  \Delta^3 &= \sh{\pair{\pair{t_1,t_2},t_3}}{t_2 \le t_1\le t_3}
\end{align*}
along their common boundary 2-simplices
\begin{align*}
  \Delta^2 &= \sh{\pair{\pair{t_1,t_2},t_3}}{t_3 \jdeq t_2\le t_1}\\
  \Delta^2 &= \sh{\pair{\pair{t_1,t_2},t_3}}{t_2\le t_3 \jdeq t_1}.
\end{align*}
This enables us to show:

\begin{prop}
  For any type $A$, the type $\Delta^3\to A$ is a retract of $\Delta^2\times\Delta^1 \to A$.
\end{prop}
\begin{proof}
  There are actually many such retractions.
  To illustrate the above decomposition, we describe one that isn't the simplest.
  For the retraction we evaluate on the ``middle'' 3-simplex of the prism:
  \[\lam{f} \lam{\pair{t_1,t_2,t_3}} f\pair{\pair{t_1,t_3},t_2} : (\Delta^2\times\Delta^1 \to A) \to (\Delta^3\to A). \]
  This is well-defined since in $\Delta^3$ we have $t_3\le t_2\le t_1$, hence in particular $t_3\le t_1$.
  The section is defined using a triple $\rec_\lor$, which we can write informally as a case split:
  \[ \lam{f}\lam{\pair{\pair{t_1,t_2},t_3}}
  \begin{cases}
    f(t_1,t_2,t_2) &\quad t_3\le t_2\\
    f(t_1,t_3,t_2) &\quad t_2\le t_3 \le t_1\\
    f(t_1,t_1,t_2) &\quad t_1\le t_3
  \end{cases}
  \]
  Here in all cases we have $t_2\le t_1$, so in each case the requirement is met for $f$ to be defined.
  The agreement on the boundary 2-simplices, when $t_3\jdeq t_1$ or $t_3\jdeq t_2$, is also obvious, as is the fact that this is a section of the above retraction.
\end{proof}

\subsection{Joins of simplices}
\label{sec:join}

In this paper we will only need shapes and simplices of very small dimension such as $n=2,3$.
However, in this subsection we briefly indicate how some important shapes such as $n$-dimensional simplicial spheres and horns can be defined, using an analogue of Joyal's non-symmetric monoidal ``join'' operation \cite{joyal-quasi}.

Informally, if given a pair of shapes
\[  \sh{t:\two^n}{\phi} \qquad\text{and}\qquad  \sh{s:\two^m}{\psi}, \]
in the $n$-cube and $m$-cube respectively, their \emph{join} is the shape in the $(n+1+m)$-cube whose tope is those $\pair{t_1,\dots,t_n,u, s_1,\ldots, s_m}:\two^n \times \two \times \two^m$ satisfying $\phi[u/0] \wedge \psi[u/1]$, i.e.~satisfying $\phi$ except with $u$ substituted for all occurrences of the term $0$ and also $\psi$ with $u$ substituted for all occurrences of the term 1.  However, this sort of ``substitution of a variable for a constant'' is not technically possible, so instead we construct joins as restrictions of an auxiliary ``gluing'' operation on shapes in cubes of a larger dimension.

\begin{defn}
Given shapes 
\begin{equation}
  A \defeq \sh{t:\two^{1+n+1}}{\phi} \qquad\text{and}\qquad  B \defeq\sh{s:\two^{1+m+1}}{\psi} \label{eq:gluing-data}
\end{equation}
where we write $t \jdeq \pair{t_-,t_1,\ldots, t_n,t_+} : \two^{1+n+1}$ and $s \jdeq \pair{s_-,s_1,\ldots, s_m,s_+} : \two^{1+m+1}$,
their \textbf{gluing} is the shape $A \oast B$ in $\two^{1+n+1+m+1}$ defined by
\[ A \oast B \defeq \sh{\pair{t_-, t_1, \ldots, t_{n}, u, s_1,\ldots, s_{m}, s_+} : \two^{1+n+1+m+1}}{\phi[u/t_{+}] \land \psi[u/s_-]} \]
\end{defn}

\begin{ex}
  For any $n,m\geq 0$, the gluing $\Delta^{1+n+1} \oast \Delta^{1+m+1}$ of simplices is the simplex $\Delta^{1+n+1+m+1}$.
\end{ex}

\begin{defn}
Given a shape 
\[  A \defeq \sh{\pair{t_-,t_1,\ldots,t_n,t_+}:\two^{1+n+1}}{\phi}, \]
its \textbf{restriction} is the shape
\[ \restr{A} \defeq \sh{\pair{t_1,\ldots,t_n}:\two^{n}}{\phi[1/t_-,0/t_+]}.
\]
\end{defn}

\begin{ex} For any $n \geq 0$, the $n$-simplex is the restriction of the $n+2$-simplex. 
\end{ex}

\begin{defn}[join]
  An \textbf{augmented shape} is a shape in $\two^n$ of the form $\restr{A}$ for some specified shape $A$ in $\two^{1+n+1}$.
  In other words, an augmented shape is really just an arbitrary shape in $\two^{1+n+1}$, but we regard it as its restriction in $\two^n$ equipped with extra data.

  Now given augmented shapes $\restr{A}$ in $\two^n$ and $\restr{B}$ in $\two^m$, their \textbf{join} is the augmented shape defined as the restriction of the gluing:
  \[ \restr{A} \join \restr{B} \defeq \restr{A \oast B}.\]
\end{defn}

\begin{ex}[joins of simplices]
  The simplex $\Delta^n$ has a canonical augmentation as $\Delta^n = \restr{\Delta^{1+n+1}}$; we always regard it as augmented in this way.
  Thus the join $\Delta^n \join \Delta^m$ is defined, and since $\Delta^{1+n+1} \oast \Delta^{1+m+1} = \Delta^{1+n+1+m+1}$ we have $\Delta^n \join \Delta^m = \Delta^{n+1+m}$.
  In particular, we can ``construct'' the simplices by iterated joins as $\Delta^n \defeq \Delta^{n-1} \star \Delta^0$, although this is not really a definition since to augment $\Delta^{n-1}$ we have to already have $\Delta^{n+1}$.
\end{ex}

\begin{defn}[pushout join]
  Given inclusions of augmented shapes $A\subseteq B$ and $C\subseteq D$, their \textbf{pushout join} is
  \[ A \pojoin B = (A \join D) \cup (B\join C) \subseteq B\join D \]
\end{defn}

\begin{ex}[boundaries of simplices]
  We can define the simplex boundaries (simplicial spheres) $\partial\Delta^n$ by recursive pushout joins.
  As base cases we define the augmented $0$- and $1$-spheres:
  \begin{alignat*}{2}
    \partial\Delta^0 &= \restr{\sh{\pair{t_-,t_+}:\two^2}{t_+\jdeq t_-}} && \subseteq \Delta^0\\
    \partial\Delta^1 &= \restr{\sh{\pair{t_-,t_1,t_+}:\two^3}{(t_+ \jdeq t_1 \le t_-) \lor (t_+ \le t_1 \jdeq t_-)}}&&\subseteq \Delta^1\\
    \intertext{and then define recursively}
    \partial\Delta^{n+1} &= \partial\Delta^n \pojoin \partial\Delta^0 && \subseteq \Delta^{n+1}.
  \end{alignat*}
  For instance, we have
  \begin{align*}
    \partial\Delta^1 \join \Delta^0 &= \restr{\sh{\pair{t_-,t_1,u,s_+}:\two^4}{(s_+ \le u \jdeq t_1 \le t_-) \lor (s_+ \le u \le t_1 \jdeq t_-)}}\\
    \Delta^1 \join \partial\Delta^0 &= \restr{\sh{\pair{t_-,t_1,u,s_+}:\two^4}{s_+ \jdeq u \le t_1 \le t_-}}
  \end{align*}
  and hence
  \begin{multline*}
    \partial\Delta^2 = (\partial\Delta^1 \join \Delta^0) \cup (\Delta^1 \join \partial\Delta^0) = \\
    \restr{\sh{\pair{t_-,t_1,u,s_+}:\two^4}{(s_+ \le u \jdeq t_1 \le t_-) \lor (s_+ \le u \le t_1 \jdeq t_-) \lor (s_+ \jdeq u \le t_1 \le t_-)}}
  \end{multline*}
  Restricting by substituting $s_+\jdeq0$ and $t_- \jdeq 1$, we obtain the usual definition of the boundary of the 2-simplex:
  \[ \sh{\pair{t_1,u}:\two^2}{(0 \le u \jdeq t_1 \le 1) \lor (0 \le u \le t_1 \jdeq 1) \lor (0 \jdeq u \le t_1 \le 1)} \]
\end{ex}

\section{Equivalences involving extension types}
\label{sec:equiv-exten}

In this section we collect several important equivalences involving extension types, mainly straightforward generalizations of standard facts about dependent function types.
Moreover, since our extension types, $\Pi$-types, and $\Sigma$-types have judgmental $\eta$-conversion as well as $\beta$-reduction, all of these equivalences are actually ``judgmental isomorphisms'', i.e.\ the composites in both directions are judgmentally equal to the identity.
We also formulate a function extensionality axiom for extension types.

Note that when a theorem is stated as an equivalence between two types $A \simeq B$, we will not scruple to later use the \emph{specific} equivalence constructed in its proof rather than the mere existence of such an equivalence.
  This is in accord with the ``proof-relevant'' philosophy of type theory and the propositions-as-types principle: a proof of a theorem is the construction of an inhabitant of some type, in this case a type of equivalences.

\subsection{Commutation of arguments and currying}
\label{sec:curry}

For ordinary function types we have an equivalence $(X\to (Y\to Z)) \simeq (Y\to (X\to Z))$, and similarly in the dependent case we have $\left(\prod_{x:X} \prod_{y:Y} Z(x,y)\right) \simeq \left(\prod_{y:Y} \prod_{x:X} Z(x,y)\right)$.
The following theorems are analogues of this for extension types.

\begin{thm}\label{thm:exten-prod-commute}
  If $t:I \mid \phi\types\psi$ and $X:\univtype$, while $Y:\sh{t:I}{\psi} \to X\to\univtype$ and $f:\prod_{t:I\mid\phi} \prod_{x:X} Y(t,x)$, then
  \[ \exten{t:I|\psi}{\Parens{\prod_{x:X} Y(t,x)}}{\phi}{f} \simeq
  \tprod_{x:X} \exten{t:I|\psi}{Y(t,x)}{\phi}{\lam{t}f(t,x)}.
  \]
\end{thm}
Note that $X$ cannot depend on $t$ or $\psi$, since otherwise the right-hand side of the equivalence would be ill-formed.
\begin{proof}
  As for ordinary dependent functions, this is just application and re-ab\-strac\-tion: from left to right $g\mapsto \lam{x} \lam{t} g(t,x)$, and from right to left $h \mapsto \lam{t}\lam{x} h(x,t)$.
  We just have to verify that the requisite judgmental equations hold: if $g(t)\jdeq f(t)$ assuming $\phi$, then $g(t,x)\jdeq f(t,x)$ assuming $\phi$; while if $h(x,t)\jdeq f(t,x)$ assuming $\phi$, then $\lam{t}\lam{x} h(x,t) \jdeq \lam{t}\lam{x}f(t,x) \jdeq f$ assuming $\phi$ (using $\eta$-conversion).
  The composites in both directions are judgmentally the identity, by $\eta$-conversion.
\end{proof}

\begin{thm}\label{thm:exten-curry}
  If $t:I \mid \phi\types\psi$ and $s:J \mid \chi\types\zeta$, while
  \[X:\sh{t:I}{\psi} \to \sh{s:J}{\zeta} \to \univtype\]
  and $f:\tprod_{\pair{t,s}:I\times J\mid (\phi\land\zeta)\lor(\psi\land\chi)} X(t,s)$, then
  \begin{align*} &\exten{t:I|\psi}{\exten{s:J|\zeta}{X(t,s)}{\chi}{\lam{s}f{\pair{t,s}}}}{\phi}{\lam{t}\lam{s}f{\pair{t,s}}}
  \\&\qquad\qquad\qquad\simeq \exten{\pair{t,s}:I\times J|\psi\land\zeta}{X(t,s)}{(\phi\land\zeta)\lor(\psi\land\chi)}{f}
\\&\qquad\qquad\qquad\simeq \exten{s:J|\zeta}{\exten{t:I|\psi}{X(t,s)}{\phi}{\lam{t}f{\pair{t,s}}}}{\chi}{\lam{s}\lam{t}f{\pair{t,s}}}. \end{align*}
\end{thm}
The equivalence of the sides to the middle in \cref{thm:exten-curry} is a version of \emph{currying}.
The shape $\sh{\pair{t,s}:I\times J}{(\phi\land\zeta)\lor(\psi\land\chi)}$ may be called the \textbf{pushout product} of the two inclusions $\sh{t:I}{\phi} \subseteq \sh{t:I}{\psi}$ and $\sh{s:J}{\chi}\subseteq \sh{s:J}{\zeta}$.
\begin{proof}
  We first check the well-formedness of the extension types.
  {Whenever $t : I$ is such that $\psi$ holds, we have for each $s : J$ such that $\chi$ holds a term $f\pair{t,s}:X(t,s)$, defining a function $\lam{s}f\pair{t,s} : \tprod_{s:J|\chi}X(t,s)$.; thus we can form
 $\exten{s:J|\zeta}{X}{\chi}{\lam{s} f\pair{t,s}}$. Now whenever $t : I$ is such that $\phi$ holds, we have for each $s : J$ such that $\zeta$ holds a term $f\pair{t,s}:X(t,s)$, which of course equals the first $f\pair{t,s}$ if we also have $\chi$, so we have the function $\lam{t}\lam{s}f\pair{t,s}:\tprod_{t:I|\phi}\exten{s:J|\zeta}{X}{\chi}{\lam{s} f\pair{t,s}}$ and thus can form the left-hand side.}
  The right-hand side is dual, and the middle is easy.
  Now the equivalence between the left- and right-hand types is again just application and re-abstraction, while the equivalence of both to the middle type is ordinary currying.
\end{proof}

Recall that when $\phi$ or $\chi$ is $\bot$, extension types behave like ordinary (dependent) function types, so we omit the angle brackets from the notation as in~\eqref{eq:exten-bot},~\eqref{eq:ndexten-bot}.
Thus, as special cases of \cref{thm:exten-curry} we have
\begin{align*}
  \exten{t:I|\psi}{\Parens{\tprod_{s:J|\zeta} X(t,s)}}{\phi}{\lam{t}\lam{s}f\pair{t,s}}
  &\simeq \exten{\pair{t,s}:I\times J|\psi\land\zeta}{X(t,s)}{\phi\land\zeta}{f}\\
  &\simeq \tprod_{s:J|\zeta}{\exten{t:I|\psi}{X(t,s)}{\phi}{\lam{s}f\pair{t,s}}}.\\
  \intertext{and}
  \tprod_{t:I|\psi} \tprod_{s:J|\zeta} X(t,s)
  &\simeq \tprod_{\pair{t,s}:I\times J|\psi\land\zeta} X(t,s)\\
  &\simeq \tprod_{s:J|\zeta} \tprod_{t:I|\psi} X(t,s).
  \intertext{with further notational specializations to the non-dependent case, such as}
  \Parens{\sh{t:I}{\psi} \to ( \sh{s:J}{\zeta} \to X)}
  &\simeq \Parens{\sh{t:I}{\psi}  \times  \sh{s:J}{\zeta} \to X}\\
  &\simeq  \Parens{\sh{s:J}{\zeta} \to (\sh{t:I}{\psi}  \to X)}.
\end{align*}

\subsection{Extending into $\Sigma$-types (the non-axiom of choice)}
\label{sec:non-choice}

For ordinary dependent functions we have the following equivalence~\cite[Theorem 2.15.7]{hottbook}:
\[ \Parens{\tprod_{x:X} \tsum_{y:Y(x)} Z(x,y)} \simeq \Parens{\tsum_{f:\prod_{x:X} Y(x)} \tprod_{x:X} Z(x,f(x)) } \]
The following is a version of this for extension types.

\begin{thm}\label{thm:exten-nonac}
  If $t:I \mid \phi\types\psi$, while $X:\sh{t:I}{\psi}\to\univtype$ and $Y:\tprod_{t:I\mid\psi} (X\to\univtype)$, while $a:\tprod_{t:I\mid\phi} X(t)$ and $b:\tprod_{t:I\mid\phi} Y(t,x(t))$, then
  \[ \exten{t:I|\psi}{\Parens{\tsum_{x:X(t)} Y(t,x)}}{\phi}{\lam{t}(a(t),b(t))} \simeq
\tsum_{f:\exten{t:I|\psi}{X(t)}{\phi}{a}} \exten{t:I|\psi}{Y(t,f(t))}{\phi}{b}. \]
\end{thm}
\begin{proof}
  As in the ordinary case, this is just composing the introduction and elimination rules.
  From left to right, $h\mapsto (\lam{t}\pi_1(h(t)), \lam{t}\pi_2(h(t)))$; while from right to left, $(f,g)\mapsto \lam{t} (f(t),g(t))$.
  It is easy to check that the required judgmental equalities are preserved, and the $\beta$- and $\eta$-conversion rules make these inverse equivalences.
\end{proof}

\subsection{Composites and unions of cofibrations}
\label{sec:compose-cofib}

These equivalences have no analogue for ordinary dependent functions.

\begin{thm}\label{thm:exten-compose}
  Suppose $t:I \mid \phi\types \psi$ and $t:I\mid \psi\types\chi$, and that $X:\sh{t:I}{\chi}\to\univtype$ and $a:\tprod_{t:I\mid\phi} X(t)$.
  Then
  \[
  \exten{t:I|\chi}{X}{\phi}{a} \simeq
  \Parens{\tsum_{f:\exten{t:I|\psi}{X}{\phi}{a}} \exten{t:I|\chi}{X}{\psi}{f}} \]
\end{thm}
\begin{proof}
  From left to right, $h\mapsto (\lam{t} h(t), \lam{t} h(t))$; the $\eta$-expansions indicate a re-packaging of the same term into a different extension type.
  Similarly, from right to left, $(f,g) \mapsto \lam{t} g(t)$.
\end{proof}

\begin{thm}\label{thm:exten-union}
  Suppose $t:I \types \phi\tope$ and $t:I \types \psi\tope$, and that we have $X:\sh{t:I}{\phi\lor\psi}\to\univtype$ and $a:\tprod_{t:I\mid\psi} X(t)$.
  Then
  \[ \exten{t:I|\phi\lor\psi}{X}{\psi}{a} \simeq
  \exten{t:I|\phi}{X}{\phi\land\psi}{\lam{t}a(t)}.
  \]
\end{thm}
\begin{proof}
  From left to right this is just re-packaging, $h\mapsto \lam{t} h(t)$.
  From right to left is a little less obvious: $g \mapsto \lam{t} \rec_\lor^{\phi,\psi}(g(t),a(t))$, which is well-defined since $g(t)\jdeq a(t)$ {for $t:I$ satisfying $\phi\land\psi$.}
  The required equalities are immediate, and the composites are the identity by the $\beta$- and $\eta$-conversion rules for $\rec_\lor$.
\end{proof}

\subsection{Relative function extensionality}
\label{sec:funext}

We will need to assume a function extensionality axiom for extension types, with respect to the homotopical identity types, which we write as equalities $x=y$.
In homotopy type theory there are at least three formulations of function extensionality for ordinary dependent functions, which turn out to be equivalent:
\begin{itemize}
\item For $f,g:\tprod_{x:A}B(x)$, if $\tprod_{x:A} (fx=gx)$, then $f=g$.
\item For $f,g:\tprod_{x:A}B(x)$, the canonical map $(f=g) \to \tprod_{x:A} (fx=gx)$ is an equivalence.
\item If each $B(x)$ is contractible, then so is $\tprod_{x:A} B(x)$.
\end{itemize}
The first is a na\"{i}ve statement of function extensionality uninformed by homotopy theory; the second is a homotopical refinement of it; and the third is an easy consequence of the second that was observed by Voevodsky to be equivalent to it. The equivalence in the second statement is meant  in the usual sense of~\cite[\S4.5]{hottbook}.

We do not know whether the analogues of these three formulations for extension types are still equivalent, so as our axiom of function extensionality we take one that we do know how to prove the others from.
Somewhat surprisingly, this is the \emph{third} rather than the second.

\begin{ax}[relative function extensionality\footnote{Or ``extension extensionality''.}]\label{ax:extfunext}
{ Supposing $t:I \mid \phi\types\psi$ and that $A:\sh{t:I}{\psi}\to\univtype$ is such that each $A(t)$ is contractible, and moreover $a:\tprod_{t:I\mid\phi} A(t)$, then $\exten{t:I\mid\psi}{A(t)}{\phi}{a}$ is contractible.}
\end{ax}

Now suppose given any $A:\sh{t:I}{\psi}\to\univtype$ and $a:\tprod_{t:I\mid\phi} A(t)$, and also $f,g:\exten{t:I \mid \psi}{A}{\phi}{a}$.
Then in the context of $t:I$ and $\psi$ we can form the identity type $f(t)=g(t)$, and thereby the extension type $\exten{t:I \mid \psi}{f(t)=g(t)}{\phi}{\lam{t}\refl}$.
Of course, we have $\lam{t^{I\mid\psi}} \refl : \exten{t:I \mid \psi}{f(t)=f(t)}{\phi}{\lam{t}\refl}$, so by identity elimination, we obtain a map
\begin{equation} (f=g) \to \exten{t:I \mid \psi}{f(t)=g(t)}{\phi}{\lam{t}\refl}.\label{eq:extension-extensionality}
\end{equation}
Analogously to ordinary function extensionality, we have:

\begin{prop}\label{thm:ext-funext}
  Assuming \cref{ax:extfunext}:
  \begin{enumerate}[label=(\roman*)]
  \item\label{itm:ext-funext-equiv} The map \eqref{eq:extension-extensionality} is an equivalence.
  \item\label{itm:ext-funext-naive}  In particular, for any $f,g:\exten{t:I \mid \psi}{A}{\phi}{a}$, if $\exten{t:I \mid \psi}{f(t)=g(t)}{\phi}{\lam{t}\refl}$ then $f=g$.
  \end{enumerate}
\end{prop}
\begin{proof}
  It suffices to prove that for each $f$ the induced map on total spaces
  \[ \Parens{\tsum_{g:\exten{t:I \mid \psi}{A(t)}{\phi}{a}} (f=g)} \to \Parens{\tsum_{g:\exten{t:I \mid \psi}{A(t)}{\phi}{a}} \exten{t:I \mid \psi}{f(t)=g(t)}{\phi}{\lam{t}\refl}}\]
  is an equivalence.
  But the domain of this map is contractible, as a based path space, so it suffices to prove that its codomain is also contractible.
  However, by \cref{thm:exten-nonac} this codomain is equivalent to
  \[ \exten{t:I\mid\psi}{\Parens{\tsum_{y:A(t)} (f(t) = y)}}{\phi}{\lam{t}(a,\refl)} \]
  which is contractible by \cref{ax:extfunext}, since each $\tsum_{y:A(t)} (f(t) = y)$ is a based path space and hence contractible.
\end{proof}

\begin{rmk}
Note that the identity type $f=g$ appearing as the domain of \eqref{eq:extension-extensionality} refers to identity \emph{in} the extension type $\exten{t:I \mid \psi}{A(t)}{\phi}{\lam{t}a}$, and \cref{thm:ext-funext}\ref{itm:ext-funext-equiv} identifies this with a type of ``relative homotopies'' or ``relative pointwise equalities'' that must restrict to reflexivity on the domain of the cofibration.
This explains the name ``relative function extensionality''.
Note also that when $\phi$ is $\bot$, relative function extensionality reduces to ordinary function extensionality for (dependent) function types whose domain is a shape rather than a type.
\end{rmk}

Another important consequence of \cref{ax:extfunext} is the following.

\begin{prop}[homotopy extension property]\label{thm:hep} Let $t:I \mid \phi\types\psi$.
  Assuming \cref{ax:extfunext},
  if we have $A:\sh{t:I}{\psi}\to\univtype$ and $b:\tprod_{t:I\mid\psi} A(t)$, and moreover $a:\tprod_{t:I\mid\phi} A(t)$ and $e:\tprod_{t:I\mid\phi} a(t) = b(t)$, then we have $a':\exten{t:I\mid\psi}{A(t)}{\phi}{a}$ and $e':\exten{t:I\mid\psi}{a'(t)=b(t)}{\phi}{e}$.
\end{prop}
\begin{proof}
  The extension type $\exten{t:I\mid\psi}{\Parens{\tsum_{y:A(t)} (y=b(t))}}{\phi}{\lam{t}(a(t),e(t))}$ is contractible by \cref{ax:extfunext}, hence inhabited.
  We obtain $a'$ and $e'$ by applying \cref{thm:exten-nonac}.
\end{proof}

We do not know how to derive \cref{ax:extfunext} or \cref{thm:hep} assuming only the conclusions of \cref{thm:ext-funext}\ref{itm:ext-funext-equiv} or \ref{itm:ext-funext-naive}, but we can show that  \cref{thm:ext-funext}\ref{itm:ext-funext-naive} and  \cref{thm:hep} imply  \cref{ax:extfunext}.

\begin{prop}
If \cref{thm:ext-funext}\ref{itm:ext-funext-naive} and the homotopy extension property hold, then the relative function extensionality axiom holds.
\end{prop}
\begin{proof}
  Suppose $A:\sh{t:I}{\psi}\to\univtype$ and $a:\tprod_{t:I\mid\phi} A(t)$ such that each $A(t)$ is contractible.
  The latter assumption supplies centers of contraction $b(t)$ for each $t:I$ such that $\psi$, hence a function $b:\tprod_{t:I\mid\psi} A(t)$.
  Contractibility of each $A(t)$ also shows that if $\phi$ then $a(t)=b(t)$, hence $e:\tprod_{t:I\mid\phi} a(t) = b(t)$.
  Thus, by \cref{thm:hep}, we have $a':\exten{t:I\mid\psi}{A(t)}{\phi}{a}$ and $e':\exten{t:I\mid\psi}{a'(t)=b(t)}{\phi}{e}$.

  It remains to show that any $f:\exten{t:I\mid\psi}{A(t)}{\phi}{a}$ is equal to $a'$.
  By \cref{thm:ext-funext}\ref{itm:ext-funext-naive}, for this it suffices to inhabit $\exten{t:I\mid\psi}{f(t)=a'(t)}{\phi}{\lam{t}\refl}$.
  Now since each $A(t)$ is contractible, we have $c:\tprod_{t:I\mid\psi} f(t)=a'(t)$, and moreover if $\phi$ then $c(t) = \refl$ since any two paths in a contractible type are equal.
  Thus, applying \cref{thm:hep} to $\lam{t} f(t)=a'(t)$ in place of $A$, with $c$ in place of $b$ and $\lam{t}\refl$ in place of $a$, we have an element of $\exten{t:I\mid\psi}{f(t)=a'(t)}{\phi}{\lam{t}\refl}$ as desired.
\end{proof}

A similar argument and an induction on $n$ shows:

\begin{prop}
  Assuming \cref{ax:extfunext}, if $A:\sh{t:I}{\psi}\to\univtype$ and $a:\tprod_{t:I\mid\phi} A(t)$ are such that each $A(t)$ is an $n$-type, then $\exten{t:I\mid\psi}{A(t)}{\phi}{a}$ is also an $n$-type.\qed
\end{prop}

For the rest of the paper, we will assume relative function extensionality, \cref{ax:extfunext}, without further comment.

\section{Segal types}
\label{sec:Segal-types}

The simplices defined in \cref{sec:simplices} are used to parametrize internal categorical structure in types satisfying an analogue of the famous Segal condition. Interestingly, because we express the Segal condition in the internal language, it has a more compact form than usual. We first introduce notation for ``hom'' types of various dimensions whose terms are ``morphisms'' or ``compositions'' in another type. We then state our Segal type axiom and prove the somewhat surprising fact that a single low-dimensional condition suffices to establish the expected categorical properties.

\subsection{The Segal condition}
\label{sec:hom-types}

We introduce the following notation.

\begin{defn}\label{defn:hom}
Given $x,y:A$, determining a term $[x,y]:A$ in context $\partial\Delta^1$, we define \[\hom_A(x,y) \defeq \ndexten{\Delta^1}{A}{\partial\Delta^1}{[x,y]}.\]
We refer to an element of $\hom_A(x,y)$ as an \textbf{arrow} from $x$ to $y$ in $A$.
\end{defn}

This plays the role of the directed hom-space of $A$.
Note that every $f : \hom_A(x,y)$ is a kind of \emph{function} from $\two$ to $A$, with the property that $f(0) \jdeq x$ and $f(1)\jdeq y$.

\begin{defn}
Similarly, for $x,y,z:A$ and $f:\hom_A(x,y)$, $g:\hom_A(y,z)$, and $h:\hom_A(x,z)$ we have an induced term $[x,y,z,f,g,h]:A$ in context $\partial\Delta^2$, and thereby an extension type that we denote
\[\homtwo{A}(x,y,z,f,g,h) \defeq \ndexten{\Delta^2}{A}{\partial\Delta^2}{[x,y,z,f,g,h]}. \]
or $\homtwoshort{A}(x,y,z,f,g,h)$ when space is at a premium.
\end{defn}

In a few places we will use 3-simplices with specified boundaries, but we will not introduce a particular notation for them.
We will not have any need for $n$-simplices with $n>3$.

\begin{defn}\label{defn:segal-type}
A \textbf{Segal type} is a type $A$ such that for all $x,y,z:A$ and $f:\hom_A(x,y)$ and $g:\hom_A(y,z)$, the type
\begin{equation}
  \sum_{h:\hom_A(x,z)} \homtwo{A}(x,y,z,f,g,h)\label{eq:composites}
\end{equation}
is contractible.
\end{defn}

In particular,~\eqref{eq:composites} is inhabited, and the first component of this inhabitant we call $g\circ f : \hom_A(x,z)$, the \textbf{composite} of $g$ and $f$.
The second component of this inhabitant is a 2-simplex in $\homtwoshort{A}(x,y,z,f,g,g\circ f)$, which we consider a ``witness that $g\circ f$ is the composite of $g$ and $f$''; we denote it by $\iscomp g f$.
The contractibility of~\eqref{eq:composites} implies that composites are unique in the following sense: given $h:\hom_A(x,z)$ and any witness $p:\homtwoshort{A}(x,y,z,f,g,h)$, we have $(h,p) = (g\circ f, \iscomp g f)$, and hence in particular $h = g\circ f$.

We can usefully reformulate~\eqref{eq:composites} as a single extension type of functions $\Delta^2 \to A$ that restrict to $f$ and $g$ on the 2-1-horn defined by \[\Lambda^2_1 = \sh{\pair{s,t}:\two\times\two}{s=1 \lor t=0}.\]
 This observation is the key step in the proof of the following alternate characterization of Segal types (cf.~Proposition \ref{prop:joyal}).

\begin{thm}\label{thm:segal-global}
  A type $A$ is Segal if and only if the {restriction} map
  \[ (\Delta^2 \to A) \to (\Lambda^2_1 \to A) \]
  is an equivalence.
\end{thm}
\begin{proof}
If $\Delta^1_1$ denotes the diagonal 1-face $\sh{\pair{s,t}:\two\times\two}{s=t}$ of $\Delta^2$, then we have $\Lambda^2_1 \cap \Delta^1_1 = \partial\Delta^1_1$ and $\Lambda^2_1 \cup \Delta^1_1 = \partial\Delta^2$.
Therefore, by \cref{thm:exten-union}, to extend a map $\Lambda^2_1 \to A$ to $\partial\Delta^2$ is equivalent to extending its restriction to $\partial\Delta^1_1$ to $\Delta^1_1$.
This provides the second in the following chain of equivalences; the third is \cref{thm:exten-compose}.
\begin{align*}
  \sum_{h:\hom_A(x,z)} \homtwo{A}(x,y,z,f,g,h)
  &\jdeq \sum_{h:\ndexten{\Delta^1}{A}{\partial\Delta^1}{[x,z]}} \ndexten{\Delta^2}{A}{\partial\Delta^2}{[x,y,z,f,g,h]}\\
  &\simeq \sum_{\ell:\ndexten{\partial\Delta^2}{A}{\Lambda^2_1}{[x,y,z,f,g]}} \ndexten{\Delta^2}{A}{\partial\Delta^2}{\ell} \\
  &\simeq \ndexten{\Delta^2}{A}{\Lambda^2_1}{[x,y,z,f,g]}.
\end{align*}
In other words,~\eqref{eq:composites} is the type of functions $\Delta^2 \to A$ that restrict to $f$ and $g$ on the 2-1-horn.
\cref{defn:segal-type} asserts that $A$ is a Segal type if for any composable $f$ and $g$ there is a unique such extension.

Using \cref{thm:exten-compose} again, we have
\begin{align*}
  (\Delta^2 \to A)
  &\simeq \sum_{k:\Lambda^2_1 \to A} \ndexten{\Delta^2}{A}{\Lambda^2_1}{k}.
\end{align*}
Therefore, $\Delta^2 \to A$ is the total space of a type family over $\Lambda^2_1 \to A$ whose fibers are the types~\eqref{eq:composites}.
Since the projection from a total space is an equivalence exactly when all the fibers are contractible, the result follows.
\end{proof}

\begin{cor}\label{thm:segal-expidl}
  If $X$ is either a type or a shape and $A:X\to\univtype$ is such that each $A(x)$ is a Segal type, then the dependent function type $\tprod_{x:X} A(x)$ is a Segal type.
\end{cor}
\begin{proof}
  Applying \cref{thm:exten-prod-commute} or \cref{thm:exten-curry} to rearrange function types, we have $(\Delta^2 \to \tprod_{x:X} A(x)) \simeq \tprod_{x:X} (\Delta^2 \to A(x))$ and similarly for $\Lambda^2_1$.
  Since $\tprod_{x:X}$ preserves fiberwise equivalences (using function extensionality or relative function extensionality), the result follows from \cref{thm:segal-global}.
\end{proof}

In the rest of this section we show that Segal types behave like categories, or more precisely $(\infty,1)$-categories.

\subsection{Identity}

Identities in a Segal type are obtained as constant maps.

\begin{defn} For any $x : A$, define a term $\idarr x : \hom_A(x,x)$ by $\idarr x(s) \jdeq x$ for all $s:\two$.
\end{defn}

The pair of ``degenerate'' 2-simplices witness that identities behave as identities in a Segal type:

\begin{prop}\label{prop:identity} If $A$ is a Segal type with terms $x,y : A$, then for any $f : \hom_A(x,y)$ we have 
$\idarr y\circ f = f$ and $f\circ \idarr x = f$.
\end{prop}
\begin{proof}
For any $f:\hom_A(x,y)$ we have a canonical 2-simplex:
\[ \lam{s,t} f(s) : \homtwo{A}(x,y,y,f,\idarr y,f).\]
To check that this has the right boundary, we see that $(s,0) \mapsto f(s)$ and $(s,s) \mapsto f(s)$, while $(1,t) \mapsto f(1) = y$.
Thus, when $A$ is Segal, by uniqueness of composites, $\idarr y\circ f = f$; and similarly $f\circ \idarr x = f$.
\end{proof}

\subsection{Associativity}

We now prove that composition in a Segal type is associative.
At first this may be surprising, since the definition of Segal type refers only to 2-simplices; but its ``uniformity'' allows us to apply it pointwise to arrows and use the fact that products of simplices contain higher-dimensional simplices.

\begin{prop}\label{prop:segal-assoc} If $A$ is a Segal type with terms $x,y,z,w : A$, then for any $f:\hom_A(x,y)$, $g:\hom_A(y,z)$, $h:\hom_A(z,w)$ we have $(h \circ g) \circ f = h \circ (g \circ f)$. 
\end{prop}
\begin{proof}
By \cref{thm:segal-expidl}, if $A$ is Segal then so is $A^\two \defeq (\two \to A)$.\footnote{For notational conciseness, we sometimes abbreviate the function type $X \to A$ as $A^X$, particularly in the case $X=\two$.}
Thus, for any $f:\hom_A(x,y)$, $g:\hom_A(y,z)$, $h:\hom_A(z,w)$, the type
\[
\sum_{p:\hom_{A^\two}(f,h)} \homtwo{A^\two}(f,g,h,\iscomp g f, \iscomp h g,p)
\] 
is contractible, and in particular inhabited. Here $\iscomp{g}{f}$ is abusive notation for the function $\two \to A^\two$ that is built from two copies of $\iscomp{g}{f} : \Delta^2 \to A$ using the equivalence  $\Delta^1 \times \Delta^1 \simeq \Delta^2 \cup_{\Delta^1_1} \Delta^2$ discussed in \cref{sec:sub-simplices}.

The second component of this inhabitant is a 2-simplex witness $\Delta^2 \times \two \to A$. There is a function 
\[ \lambda (t_1,t_2,t_3).((t_1,t_3),t_2) : \Delta^3 \to \Delta^2 \times \two\] that picks out the ``middle shuffle''.  The 1st and 2nd faces are identified with further restrictions
\[ \lambda (s,t). (s,s,t) : \Delta^2 \to \Delta^3 \qquad \lambda(s,t).(s,t,t) : \Delta^2 \to\Delta^3,\] with a common edge
\[ \lambda t. (t,t,t) : \Delta^1 \to \Delta^3.\]
This edge defines an inhabitant $\ell : \hom_A(x,w)$, while the pair of 2-simplices define witnesses that $\ell$ is the composite of $h \circ g$ and $f$, and that $\ell$ is the composite of $h$ with $f \circ g$, respectively. In particular, $(h \circ g) \circ f = h \circ (g \circ f)$.  
\end{proof}

\subsection{Homotopies}
\label{sec:homotopies}

Let $A$ be a Segal type with terms $x,y :A$.
Given two arrows $f,g:\hom_A(x,y)$, there are two ways to say that $f$ and $g$ are the same:
\begin{itemize}
\item we might have a path $p:f=_{\hom_A(x,y)}g$, or 
\item we might have a 2-simplex $q:\homtwo{A}(x,x,y,\idarr x,f,g)$.\footnote{A third case is presented by a 2-simplex in the ``dual'' type $\homtwoshort{A}(x,y,y,f, \idarr y, g)$, but an analogous argument will work for those; see Remark \ref{rmk:duality}.}
\end{itemize}
We demonstrate that these two types are in fact equivalent:

\begin{prop}\label{prop:htpy-is-htpy} For any $f,g : \hom_A(x,y)$ in a Segal type $A$, the natural map  
\begin{equation}
  (f=g) \to \homtwo{A}(x,x,y,\idarr x,f,g)\label{eq:idtosimplex}
\end{equation}
is an equivalence.
\end{prop}
\begin{proof}
The map \eqref{eq:idtosimplex} is defined by path induction, since for any $f$ we have a ``degenerate'' 2-simplex \[ \lam{s,t}f(t): \homtwo{A}(x,x,y,\idarr x,f,f)\] defined to be constant on one input.
To show that~\eqref{eq:idtosimplex} is an equivalence, it suffices to show that the map of total spaces
\[ \Parens{\sum_{g:\hom_A(x,y)}(f=g)} \to \Parens{\sum_{g:\hom_A(x,y)} \homtwo{A}(x,x,y,\idarr x,f,g)}\]
is an equivalence.
But here both types are contractible; the first since it is a based path space, and the second since $A$ is a Segal type.
\end{proof}

More generally, we can say:

\begin{prop}\label{thm:comp-htpy}
  For $f:\hom_A(x,y)$ and $g:\hom_A(y,z)$ and $h:\hom_A(x,z)$ in a Segal type $A$, the natural map
  \[ (g\circ f = h) \to \homtwo{A}(x,y,z,f,g,h) \]
  is an equivalence.
\end{prop}
\begin{proof}
  The map is defined by path induction, since when $h\jdeq g\circ f$ the codomain is inhabited by $\iscomp g f$.
  Now we again show it to be an equivalence in the general case by summing over $h$ and noting that both types become contractible.
\end{proof}

However defined, the homotopies between arrows in a Segal type behave like a 2-category up to homotopy.
For instance, given $p:f=_{\hom_A(x,y)} g$ and $q:g=_{\hom_A(x,y)} h$, we can concatenate them as equalities to get $p\ct q : f=_{\hom_A(x,y)} h$, a ``vertical'' composite. 
We can also compose them ``horizontally'':

\begin{prop}
  Given $p:f=_{\hom_A(x,y)} g$ and $q:h=_{\hom_A(y,z)} k$ in a Segal type $A$, there is a concatenated equality $q \circ_2 p : h\circ f =_{\hom_A(x,z)} k\circ g$.
\end{prop}
\begin{proof}
  By path induction on $p$ and $q$, defining $\refl_h \circ_2 \refl_f \defeq \refl_{h\circ f}$.
\end{proof}

In particular, taking one of $p$ or $q$ to be $\refl$ but not the other, we obtain ``whiskering'' operations.
These have another useful characterization:

\begin{prop}
  Given $p:f=_{\hom_A(x,y)} g$ and $h:\hom_A(y,z)$ and $k:\hom_A(w,x)$ in a Segal type $A$, we have
  \begin{align*}
    \refl_h \circ_2 p &= \ap_{(h\circ -)}(p)\\
    p \circ_2 \refl_k &= \ap_{(-\circ k)}(p).
  \end{align*}
\end{prop}
\begin{proof}
  By path induction on $p$.
\end{proof}

Of course, we have the usual middle-four interchange law:

\begin{prop}
  We have the following equality in a Segal type whenever it makes sense:
  \[ (q'\ct p') \circ_2 (q\ct p) = (q' \circ_2 q) \ct (p' \circ_2 p). \]
\end{prop}
\begin{proof}
  By path induction on all four equalities.
\end{proof}

On the other hand, if we view homotopies as 2-simplices, then a natural way to compose them is by filling 3-dimensional horns, as in a quasicategory.
We can express this in terms of whiskering and concatenation of equalities.

\begin{prop}\label{thm:32horn-concat}
  In a Segal type $A$, suppose given arrows $f,g,h,k,\ell,m$ and equalities
  \begin{mathpar}
    p:g\circ f =_{\hom_A(x,z)} k \and
    q:h\circ g =_{\hom_A(z,w)} \ell \and
    r:h\circ k =_{\hom_A(x,w)} m
  \end{mathpar}
  corresponding to 2-simplices that fill out the following horn $\Lambda^3_2 \to A$:
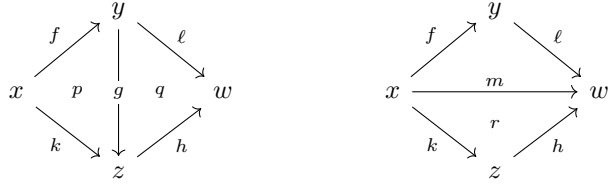
  \[
  \begin{tikzcd}
    & y \ar[dd,"g" description] \ar[dr,"\ell"] &&&& y \ar[dr,"\ell"] \\
    x \ar[ur,"f"] \ar[dr,"k"'] \ar[rr,phantom,"\scriptstyle p" near start, "\scriptstyle q" near end] && w && x \ar[ur,"f"] \ar[dr,"k"'] \ar[rr,"m"] && w \\
    & z \ar[ur,"h"'] &&&& z \ar[ur,"h"'] \ar[uu,phantom,"\scriptstyle r" near start]
  \end{tikzcd}
  \]
  Then the horn has a filler $\Delta^3 \to A$ in which the missing 2-face is the 2-simplex corresponding to the concatenated equality
  \begin{equation}\label{eq:32horn-concat}
    \ell \circ f \overset q= (h\circ g) \circ f
    = h \circ (g\circ f)
    \overset p= h\circ k
    \overset r= m.
  \end{equation}
  where $p$ and $q$ are whiskered by $h$ and $f$ respectively.
\end{prop}
\begin{proof}
  First we do path induction on $p$ and $q$, enabling us to assume $k\jdeq g\circ f$ and $\ell\jdeq h\circ g$.
  The 2-simplices corresponding to $p\jdeq \refl$ and $q\jdeq \refl$ are now $\iscomp g f$ and $\iscomp h g$, while~\eqref{eq:32horn-concat} reduces to
  \begin{equation}
    (h\circ g) \circ f = h \circ (g\circ f) \overset r= m.\label{eq:32horn-concat2}
  \end{equation}
  Now the proof of \cref{prop:segal-assoc} constructed a 3-simplex of the form
  
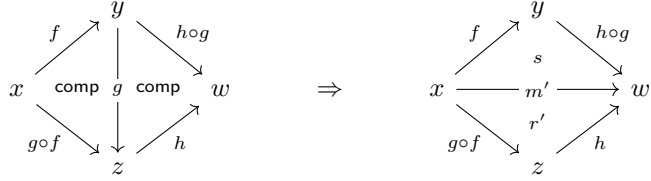
\[
  \begin{tikzcd}
    & y \ar[dd,"g" description] \ar[dr,"h\circ g"] &&&& y \ar[dr,"h\circ g"] \\
    x \ar[ur,"f"] \ar[dr,"g\circ f"'] \ar[rr,phantom,"\scriptstyle \mathsf{comp}" near start, "\scriptstyle \mathsf{comp}" near end] && w & \Rightarrow & x \ar[ur,"f"] \ar[dr,"g\circ f"'] \ar[rr,"m'" description] && w \\
    & z \ar[ur,"h"'] &&&& z \ar[ur,"h"'] \ar[uu,phantom,"\scriptstyle r'" near start,"\scriptstyle s" near end]
  \end{tikzcd}
  \]
  By the contractibility of the type of composites for $h$ and $g\circ f$, we have $(m',r') = (m,r)$.
  And the type of 3-2-horns can be decomposed as
  \[ (\Lambda^3_2\to A) \simeq \sum_{\alpha : \Delta^2 \cup_{\Delta^1} \Delta^2 \to A} \ndexten{\Delta^2}{A}{\Lambda^2_1}{\alpha} \]
  where $\Delta^2 \cup_{\Delta^1} \Delta^2$ denotes the left-hand half of the 3-simplex drawn above, with $\Lambda^2_1$ sitting inside it as the lower two 1-simplices.
  Thus, the equality $(m',r') = (m,r)$ in $\ndexten{\Delta^2}{A}{\Lambda^2_1}{[h,g\circ f]}$ yields an equality of 3-2-horns
  \[[\iscomp g f,\iscomp h g,r'] = [\iscomp g f,\iscomp h g,r].\]
  The above 3-simplex lives in $\ndexten{\Delta^3}{A}{\Lambda^3_2}{[\iscomp g f,\iscomp h g,r']}$, so we can transport it across this equality to get a 3-simplex in $\ndexten{\Delta^3}{A}{\Lambda^3_2}{[\iscomp g f,\iscomp h g,r]}$, which is what we wanted.

  Finally, by the naturality of path transport, the missing 2-simplex face is the transport of $s$ along the equality $m'=m$, which is equal to the concatenation $m' \overset {r'}= h\circ (g\circ f) \overset{r}= m$ of the two equalities induced by $r$ and $r'$.
  Thus, the equality corresponding to this face is
  \[ (h\circ g) \circ f \overset s= m' \overset {r'}= h\circ (g\circ f) \overset{r}= m. \]
  But the concatenation of the first two of these equalities was the \emph{definition} of associativity $(h\circ g) \circ f = h \circ (g\circ f)$ in \cref{prop:segal-assoc}, so this is equal to~\eqref{eq:32horn-concat2}.
\end{proof}

\subsection{Anodyne maps}
\label{sec:anodyne}

The definition of Segal type says that any 2-1-horn has a unique filler, i.e.\ that any extension type of the form $\ndexten{\Delta^2}{A}{\Lambda^2_1}{h}$ is contractible.
This is sufficient to imply that many other cofibrations have the same property.

\begin{defn}
  An inclusion of shapes $t:I \mid \phi\types\psi$ is \textbf{inner anodyne} if for any Segal type $A$ and any $h:\sh{t:I}{\phi}\to A$, the extension type $\ndexten{\sh{t:I}{\psi}}{A}{\phi}{h}$ is contractible.
\end{defn}

This can only be a meta-theoretic definition, but we will not worry too much about exactly how it should be made precise; our only intent is to exhibit certain other maps as inner anodyne.

\begin{prop}\label{thm:pop-anodyne}
  If $t:I \mid \phi\types\psi$ is inner anodyne and $s:J \mid \chi \types \zeta$ is any cofibration, then the pushout product
  \[ {\pair{t,s}:I\times J} \mid {(\phi\land\zeta)\lor(\psi\land\chi)} \types \psi\land \zeta \]
  is inner anodyne.
\end{prop}
\begin{proof}
  By \cref{thm:exten-curry}, for any Segal type $A$ and any $h$, we have
  \begin{align*}
&    \ndexten{\sh{I\times J}{\psi\land\zeta}}{A}{(\phi\land\zeta)\lor(\psi\land\chi)}{h}\\
    &\qquad\qquad\qquad\simeq \exten{s:J|\zeta}{\ndexten{\sh{t:I}{\psi}}{A}{\phi}{\lam{t}h{\pair{t,s}}}}{\chi}{\lam{s}\lam{t}h{\pair{t,s}}}.
  \end{align*}
  But the latter is an extension type of a contractible family, hence contractible by relative function extensionality (\cref{ax:extfunext}).
\end{proof}

\begin{prop}\label{thm:3horn-anodyne}
  In the cube context $    \pair{t_1,t_2,t_3}:\two^3$, the 3-1-horn and 3-2-horn inclusions
  \begin{align*}
(0\jdeq t_3\le t_2\le t_1) \lor (t_3\le t_2\jdeq t_1) \lor (t_3\le t_2\le t_1\jdeq 1) &\types t_3\le t_2\le t_1\\
 (0\jdeq t_3\le t_2\le t_1) \lor (t_3\jdeq t_2\le t_1) \lor (t_3\le t_2\le t_1\jdeq 1) &\types t_3\le t_2\le t_1
  \end{align*}
  are inner anodyne.
\end{prop}
\begin{proof}
  This is essentially the same argument as Joyal's lemma~\cite[2.3.2.1]{HTT}, but written using our ``interval'' description of simplices rather than the ``finite ordered set'' version.
  Let $\Lambda^3_2 \to \Delta^3$ be the 3-2-horn inclusion; the 3-1 case is analogous.
  By \cref{thm:pop-anodyne}, the pushout product of this inclusion with the 2-1-horn $\Lambda^2_1\to\Delta^2$:
  \[ (\Lambda^3_2 \times \Delta^2) \cup_{(\Lambda^3_2 \times \Lambda^2_1)} (\Delta^3\times\Lambda^2_1) \longrightarrow \Delta^3\times\Delta^2 \]
  is inner anodyne.
  For brevity, let $X$ be the domain of this pushout product.
  We will show that for any $h:\Lambda^3_2 \to A$, there is an $\hat h:X\to A$ such that $\ndexten{\Delta^3}{A}{\Lambda^3_2}{h}$ is a retract of $\ndexten{\Delta^3\times\Delta^2}{A}{X}{\hat h}$; thus when $A$ is Segal, the former is contractible since the latter is.

  The retraction will be defined by
  \[
  \begin{array}{ccc}
    \ndexten{\Delta^3\times\Delta^2}{A}{X}{\hat h} & \to &\ndexten{\Delta^3}{A}{\Lambda^3_2}{h} \\
    f & \mapsto & \lam{\pair{t_1,t_2,t_3}} f\pair{\pair{t_1,t_2,t_3},\pair{t_1,t_2}}.
  \end{array}
  \]
  This will be well-defined as long as we define $\hat h$ such that
  \[ \hat h\pair{\pair{t_1,t_2,t_3},\pair{t_1,t_2}} \jdeq h\pair{t_1,t_2,t_3}. \]
  whenever $\pair{t_1,t_2,t_3} : \Lambda^3_2$, i.e.\ whenever we have
  \[ (0\jdeq t_3\le t_2\le t_1) \lor (t_3\jdeq t_2\le t_1) \lor (t_3\le t_2\le t_1\jdeq 1). \]
  Note that $\pair{t_1,t_2,t_3} : \Lambda^3_2$ implies $\pair{\pair{t_1,t_2,t_3},\pair{t_1,t_2}}:X$, so this condition makes sense.

  If we had maximum and minimum operations $\pmb{\lor}$ and $\pmb{\land}$ in our theory, the section would be defined by
  \[
  \begin{array}{ccc}
    \ndexten{\Delta^3}{A}{\Lambda^3_2}{h} & \to &\ndexten{\Delta^3\times\Delta^2}{A}{X}{\hat h}\\
    g & \mapsto & \lam{\pair{\pair{t_1,t_2,t_3},\pair{s_1,s_2}}} g\pair{t_1 \mathbin{\pmb{\lor}} s_1, t_2 \mathbin{\pmb{\land}} s_2, t_3 \mathbin{\pmb{\land}} s_2}
  \end{array}
  \]
  and we could say that $\Lambda^3_2 \to \Delta^3$ is literally a retract of $X \to \Delta^3\times \Delta^2$, before we even form the extension types.
  Instead, we have to define the section at the level of extension types only, using $\rec_\lor$ and a sextuple case split:
  \[
  \begin{array}{r@{}l}
    \ndexten{\Delta^3}{A}{\Lambda^3_2}{h}\  &\to  \ndexten{\Delta^3\times\Delta^2}{A}{X}{\hat h}\\
    g\  &\mapsto  \lam{\pair{\pair{t_1,t_2,t_3},\pair{s_1,s_2}}}
    \begin{cases}
      g\pair{t_1,t_2,t_3} & (s_1\le t_1) \land (t_2\le s_2)\\
      g\pair{s_1,t_2,t_3} & (t_1\le s_1) \land (t_2\le s_2)\\
      g\pair{t_1,s_2,t_3} & (s_1\le t_1) \land (t_3\le s_2\le t_2)\\
      g\pair{s_1,s_2,t_3} & (t_1\le s_1) \land (t_3\le s_2\le t_2)\\
      g\pair{t_1,s_2,s_2} & (s_1\le t_1) \land (s_2\le t_3)\\
      g\pair{s_1,s_2,s_2} & (t_1\le s_1) \land (s_2\le t_3).
    \end{cases}
  \end{array}
  \]
  It is easy to see that this is indeed a section of the above retraction, as long as it is well-defined.
  We can make both the section and retraction well-defined if we define $\hat h : X \to A$ by the same sextuple case split from $h:\Lambda^3_2\to A$, as long as we check that this makes sense, i.e.\ that if $\pair{\pair{t_1,t_2,t_3},\pair{s_1,s_2}}:X$ then all the arguments to which the case split tries to apply $g$ in fact live in $\Lambda^3_2$.

  It is clear that in all cases if $\pair{\pair{t_1,t_2,t_3},\pair{s_1,s_2}}:\Delta^3\times\Delta^2$ then the arguments of $g$ lie in $\Delta^3$.
  The additional condition imposed by $\pair{\pair{t_1,t_2,t_3},\pair{s_1,s_2}}:X$ is
  \begin{equation*}
    (0\jdeq t_3) \lor (t_3\jdeq t_2) \lor (t_1 \jdeq 1) \lor (0\jdeq s_2) \lor (s_1\jdeq 1).
  \end{equation*}
  and we have to show that each of these five cases implies
  \[ (0\jdeq (t_3 \mathbin{\pmb{\land}} s_2)) \lor
  ((t_3 \mathbin{\pmb{\land}} s_2) \jdeq (t_2 \mathbin{\pmb{\land}} s_2)) \lor
  ((t_1 \mathbin{\pmb{\lor}} s_1)\jdeq 1)
  \]
  But this is easy:
  \begin{itemize}
  \item If $0\jdeq t_3$ then $0\jdeq (t_3 \mathbin{\pmb{\land}} s_2)$.
  \item If $t_3\jdeq t_2$ then $(t_3 \mathbin{\pmb{\land}} s_2) \jdeq (t_2 \mathbin{\pmb{\land}} s_2)$.
  \item If $t_1 \jdeq 1$ then $(t_1 \mathbin{\pmb{\lor}} s_1)\jdeq 1$.
  \item If $0\jdeq s_2$ then $0\jdeq (t_3 \mathbin{\pmb{\land}} s_2)$.
  \item If $s_1\jdeq 1$ then $(t_1 \mathbin{\pmb{\lor}} s_1)\jdeq 1$.
  \end{itemize}
  Technically, lacking $\pmb{\lor}$ and $\pmb{\land}$ this must be proven by splitting into $30 = 5\times 6$ cases, but we leave that for automation in a proof assistant.
\end{proof}

If we gave a formal definition of all the horns $\Lambda^n_k$ (by meta-theoretic induction on $n$ and $k$), then we could presumably generalize \cref{thm:3horn-anodyne} to a theorem-schema that all the inner horn inclusions $\Lambda^n_k\to \Delta^n$ ($0<k<n$) are inner anodyne.

\begin{cor}\label{thm:32horn-is-concat}
  In the situation of \cref{thm:32horn-concat}, if we also have $s:\ell\circ f = m$, then the type of 3-simplices with given boundary $\ndexten{\Delta^3}{A}{\partial\Delta^3}{[p,q,r,s]}$ is equivalent to the type of equalities from $s$ to~\eqref{eq:32horn-concat}.
\end{cor}
\begin{proof}
  \cref{thm:32horn-concat} gives a map from equalities to 3-simplices.
  As in \cref{prop:htpy-is-htpy}, if we sum over $s$ we get contractible types on both sides: one being a based path-space, the other by \cref{thm:3horn-anodyne}.
  Thus, the map is an equivalence.
\end{proof}

\section{The 2-category of Segal types}
\label{sec:2cat-segal}

The collection of all categories is a 2-category, and similarly the collection of all $(\infty,1)$-categories is an $(\infty,2)$-category.
In this section we introduce the 2-categorical structure on the collection of Segal types.
Unlike in case of Segal types themselves, where we had to introduce an extra condition to characterize those types that ``behave like categories'', every function between Segal types is automatically a ``functor'', and every arrow between such functors is automatically a ``natural transformation''.
Thus, Segal types really do behave like ``synthetic $(\infty,1)$-categories''.

\subsection{Functoriality}
\label{sec:functoriality}


Given a function $\phi : A \to B$, we have an induced function
\[ \extfn \phi: \hom_A(x,y) \to \hom_{B}(\phi x, \phi y)\]
defined by postcomposition: $(\extfn \phi(f))(s) = \phi(f(s))$.
We usually write $\extfn \phi$ abusively as simply $\phi$.
In the case where $A$ and $B$ are Segal types, the following observations justify our referring to any function $\phi : A \to B$ as a \textbf{functor}.

\begin{prop}\label{prop:segal-functors} Any function $\phi : A \to B$ between Segal types preserves identities and composition.
\end{prop}
\begin{proof}
In the case of $\idarr x : \hom_A(x,x)$, $\phi \idarr x : \hom_{B}(\phi x, \phi x)$ is defined by $\phi \idarr x(s) \jdeq \phi x$ for all $s : \two$. As $\idarr {\phi x}$ has the same definition, we conclude that $\idarr {\phi x} \jdeq \phi(\idarr x)$, i.e.\~that functors preserve identities.

Similarly, $\phi$ gives rise to a postcomposition function
\[ \extfn\phi : \homtwo{A}(x,y,z,f,g,h) \to \homtwo{B}(\phi x, \phi y, \phi z, \phi f , \phi g, \phi h).\]
In particular, when $A$ is a Segal type, we have a term
\[ \extfn\phi(\iscomp{g}{f}): \homtwo{B}(\phi x, \phi y, \phi z, \phi f, \phi g, {\phi (g \circ f)})\]
witnessing the fact that $\phi (g \circ f)$ is a composite of $\phi g$ and $\phi f$. If $B$ is also a Segal type, we have $(\phi g \circ \phi f, \iscomp{\phi g}{\phi f}) = (\phi (g \circ f),\phi (\iscomp{g}{f}))$, which implies in particular that $\phi g \circ \phi f = \phi (g \circ f)$.
\end{proof}

Similarly, it can be shown that any functor between Segal types preserves all the higher-dimensional operations defined in \cref{sec:homotopies}.

\subsection{Naturality}
\label{sec:naturality}

\begin{defn}
Given Segal types $A$ and $B$ and $f,g : A\to B$, we refer to a term $\alpha : \nat A B(f,g)$ as a \textbf{natural transformation from $f$ to $g$}.

Given such an $\alpha$ we can apply it to any $s : \two$ to yield a term $\extapp{\alpha}{s}:A\to B$ so that $\extapp{\alpha}{0} \jdeq f$ and $\extapp{\alpha}{1}\jdeq g$.
In particular, for each $a : A$, we have
\[ \lambda s. \extapp{\alpha}{s}(a) : \two \to B\]
which we can alternatively abstract as
\[\extlam{s}{\extapp{\alpha}{s}(a)} : \hom_{B}(fa,ga).\]
We refer to the latter as the \textbf{component of $\alpha$ at $a$} and denote it by $\alpha_a$. 
\end{defn}

As expected, ``a natural transformation is defined by its components''.

\begin{prop}\label{prop:transformation-extensionality}
For any type $B$ and any type or shape $A$, and any $f,g : A \to B$, the function
\begin{equation}
  \nat AB(f,g) \to \prod_{a : A} \hom_{B}(fa,ga)\label{eq:transf-ext}
\end{equation}
that carries a natural transformation to its components is an equivalence.
\end{prop}

\begin{proof}
Writing out the hom-types as extension types, this function becomes
\[ \exten{t:\two}{(A\to B)}{\booltype}{[f,g]} \to \tprod_{a:A} \exten{t:\two}{B}{\booltype}{[fa,ga]} \]
and is therefore an equivalence, being an instance of \cref{thm:exten-prod-commute} or \cref{thm:exten-curry} depending on whether $A$ is a type or a shape.
\end{proof}

We sometimes refer to the equivalence of \cref{prop:transformation-extensionality} as \textbf{transformation extensionality}, since it is a directed analogue of function extensionality for homotopies.
Since it is not just an equivalence but a judgmental isomorphism (due to the $\beta$- and $\eta$-rules for dependent function types and extension types), we generally blur the line between a natural transformation and its components.

All other structure on natural transformations is also performed componentwise.
For instance, we have:

\begin{prop}\label{prop:componentwise-composition}
  If $B$ is a Segal type and $A$ is a type or a shape, then for $\alpha:\nat AB(f,g)$ and $\beta: \nat AB(g,h)$ and $x:A$ we have
  \[ (\beta\circ\alpha)_x = \beta_x \circ \alpha_x \qquad\text{and}\qquad (\idarr{f})_x = \idarr{f(x)}. \]
\end{prop}
\begin{proof}
  By definition, we have a 2-simplex witness
  \[\iscomp \beta\alpha : \nattwo AB(f,g,h,\alpha,\beta,\beta\circ\alpha).\]
  By commuting arguments and evaluating at $x:A$ this gives a 2-simplex
  \[(\iscomp \beta\alpha)_x : \homtwo B(f(x),g(x),h(x),\alpha_x,\beta_x,{(\beta\circ\alpha)_x}).\]
  But we also have a 2-simplex
  \[\iscomp{\beta_x}{\alpha_x} : \homtwo B(f(x),g(x),h(x),\alpha_x,\beta_x,\beta_x\circ\alpha_x)\]
  so since $B$ is Segal we have $(\beta\circ\alpha)_x = \beta_x \circ \alpha_x$.
  The second equality is in fact judgmental, since $\idarr{f}$ is defined to be constant at $f$.
\end{proof}

We now demonstrate that natural transformations really are natural.

\begin{prop}\label{prop:nat-are-nat} For any natural transformation $\alpha : \nat AB(f,g)$ for which $B$ is a Segal type and $k : \hom_A(x,y)$, $\alpha_y \circ fk = gk \circ \alpha_x$. 
\end{prop}
\begin{proof}
By the usual associativity of non-dependent function types, we can also rearrange $\alpha : \nat AB(f,g)$ into a function
\[ \alpha \jdeq \lam{a}\lam{s}\extapp{\alpha}{s}(a) : A \to B^\two.\]
Now the functoriality of the extension type construction yields a map
\[ \extfn\alpha : \hom_A(x,y) \to \hom_{B^\two}(\alpha_x,\alpha_y).\]
Given $k :\hom_A(x,y)$, the term $\extfn\alpha(k)$ is a function from $\two$ to $B^\two$, and evaluating it at $0,1:\two$ we get $\extfn\alpha(k)(0) \jdeq \alpha_x : \hom_{B}(fx,gx)$ and $\extfn\alpha(k)(1)\jdeq\alpha_y : \hom_{B}(fy,gy)$.
We can also postcompose $\extfn\alpha(k) : \two \to B^\two$ with the evaluation functions $\mathsf{ev}_0, \mathsf{ev}_1 : B^\two \to B$ to yield $fk : \hom_{B}(fx,fy)$ and $gk : \hom_{B}(gx,gy)$.

Put differently, $\extfn\alpha(k)$ can be regarded (technically, by application and re-ab\-stract\-ion) as a function $\two \to B^\two$, and thereby uncurried to obtain a function $\two\times \two \to B$, and the preceding paragraph identifies the values of this function on four of the nondegenerate 1-simplices in $\two\times\two$.
If we call its value on the fifth ``diagonal'' 1-simplex $\alpha_k$
\[
\begin{tikzcd}
fx \arrow[r, "\alpha_x"] \arrow[dr, "\alpha_k" description] \arrow[d, "fk"'] & gx \arrow[d, "gk"] \\ fy \arrow[r, "\alpha_y"']  \arrow[ur, phantom, "\cdot" very near start, "\cdot" very near end] & gy
\end{tikzcd}
\]
then the two 2-simplices contained in $\two\times\two$ yield witnesses that $\alpha_y \circ fk = \alpha_k$ and $gk \circ \alpha_x = \alpha_k$, hence $\alpha_y \circ fk = gk \circ \alpha_x$. This demonstrates that $\alpha$ really is natural.
\end{proof}

\subsection{Horizontal composition}
\label{sec:horiz}

Given
\begin{mathpar}
f,g : A \to B \and j,k : B \to C \and \alpha : \nat AB(f,g) \and \beta : \nat BC(j,k)
\end{mathpar}
where $C$ at least is a Segal type, we can define a \textbf{horizontal composite} natural transformation
\[
\beta\ast\alpha :\nat CA(jf,kg)
\]
as follows. Define $\beta\ast\alpha$ to be
\[
\lam{a}\beta_{\alpha_a} : A \to C^\two.
\]
That is, the component $(\beta\ast\alpha)_a : \hom_C(jfa, kga)$ is defined to be the ``diagonal'' 1-simplex of the term obtained by applying the map
\[ \extfn\beta : \hom_{B}(fa,ga) \to \hom_{C^\two}(\beta_{fa}, \beta_{ga})\] to $\alpha_a : \hom_{B}(fa,ga)$.
\[
\begin{tikzcd}
jfa \arrow[r, "\beta_{fa}"] \arrow[dr, "\beta_{\alpha_a}" description] \arrow[d, "j\alpha_a"'] & kfa \arrow[d, "k\alpha_a"] \\ jga \arrow[r, "\beta_{ga}"']  \arrow[ur, phantom, "\cdot" very near start, "\cdot" very near end]  & kga
\end{tikzcd}
\]
The square $\two \times \two \to C$ witnesses the ``Gray interchanger'', a homotopy between the two ways to define a ``horizontal composite'' in terms of whiskering.

\section{Discrete types}
\label{sec:discrete-types}

In this section we will show that any type $A$ that satisfies a natural ``discreteness'' condition is a Segal type.  In a discrete type, the hom-types of \cref{defn:hom} are equivalent to its identity types, so in these types the Segal structure can be thought of as another incarnation of the weak $\omega$-groupoid structure possessed by any type (see~\cite{vdbg:oogpds,pll:oogpds}). For more on the ``groupoid'' interpretation of discrete types, see \cref{prop:discrete-are-groupoid}.

\begin{defn}\label{defn:discrete}
  For any type $A$, there is a map
  \[ \idtoarr: \prod_{x,y :A} (x =_A y) \to \hom_A(x,y)\]
  defined by path induction and the equation $\idtoarr_{x,x}(\refl_x) \defeq \idarr x$.
  We say that $A$ is a \textbf{discrete type} if $\idtoarr_{x,y}$ is an equivalence for all $x,y:A$.
\end{defn}

\begin{prop}\label{thm:forall-discrete} If $A:X\to\univtype$ is a family of discrete types, then $\prod_{x:X} A(x)$ is also discrete.
\end{prop}
\begin{proof}
  Let $f,g:\prod_{x:X} A(x)$.
  By function extensionality and \cref{thm:exten-prod-commute},
  \begin{align*}
    (f=g) &\simeq \prod_{x:X} (f(x)=g(x))\\
    \hom_{\prod_{x:X} A(x)}(f,g) &\simeq \prod_{x:X}\hom_A(f(x),g(x)).
  \end{align*}
  Thus, applying $\prod_{x:X}$ to the fiberwise equivalence $(f(x)=g(x))\simeq \hom_A(f(x),g(x))$ arising from the discreteness of $A$, we obtain the result.
\end{proof}

\begin{prop} If $A$ is a discrete type then $A$ is a Segal type.
\end{prop}
\begin{proof}
By \cref{thm:forall-discrete}, the type $A^\two$ is also discrete; thus we know that for every $f : \hom_A(x,y)$ and $g : \hom_A(z,w)$, the map
\[ \idtoarr : (f =_{A^\two} g) \to \hom_{A^\two}(f,g) \] is an equivalence. Since the type of arrows $A^\two$ is equivalent to the dependent sum $\sum_{x :A} \sum_{y : A} \hom_A(x,y)$, its identity types are characterized as dependent sums as well by~\cite[Theorem 2.7.2]{hottbook}, and \cref{thm:exten-nonac} similarly characterizes its hom-types.
Thus we have an equivalence
\[ \Parens{ \sum_{e_1: x = z} \sum_{e_2 :y = w} f =^{(e_1,e_2)}_{\hom_A} g } \simeq \sum_{h : \hom_A(x,z)} \sum_{k : \hom_A(y,w)} \ndexten{\Delta^1\times\Delta^1}{A}{\partial(\Delta^1\times\Delta^1)}{[h,f,k,g]}\]
where the right-hand extension type is the type of extensions from the square boundary
\[ \begin{tikzcd} x \arrow[d, "f"'] \arrow[r, "h"] & z \arrow[d, "g"] \\ y \arrow[r, "k"'] & w
\end{tikzcd}
\] into a diagram $\Delta^1 \times \Delta^1 \to A$. This equivalence projects onto the equivalence
\[ \idtoarr : (x = z) \times (y = w) \to \hom_A(x,z) \times \hom_A(y,w),\] inducing an equivalence on the fibers over any pair of terms. Specializing to the case of $(\refl_x, \refl_y) : (x =x ) \times (y=y)$ we see that the types
\[ (f =_{\hom_A(x,y)} g ) \simeq  \ndexten{\Delta^1\times\Delta^1}{A}{\partial(\Delta^1\times\Delta^1)}{[\idarr x, f,\idarr y, g]}\] are equivalent. Hence, there is an equivalence
\[  \sum_{g : \hom_A(x,y)} (f =_{\hom_A(x,y)} g ) \simeq  \sum_{g : \hom_A(x,y)} \ndexten{\Delta^1\times\Delta^1}{A}{\partial(\Delta^1\times\Delta^1)}{[\idarr x, f,\idarr y, g]} ,\] and since the left-hand type is a based path space, both types are contractible.

Applying \cref{thm:exten-compose}, the right-hand type is equivalent to a single extension type
\begin{equation}
\ndexten{\Delta^1\times\Delta^1}{A}{d}{[\idarr x, f,\idarr y]} \label{eq:discrete-retract}
\end{equation} 
where $d$ is the ``cubical horn''
  \[
\left( \begin{tikzcd} \cdot \arrow[r, "{\idarr x}"] \arrow[d, "f"'] & \cdot \\ \cdot \arrow[r, "{\idarr y}"'] & \cdot \end{tikzcd} \right) \longrightarrow \left( \begin{tikzcd} \cdot \arrow[r, "{\idarr x}"] \arrow[d, "f"']   \arrow[dr] & \cdot \arrow[d] \arrow[dl, phantom, "\cdot" very near start, "\cdot" very near end] \\ \cdot\arrow[r, "{\idarr y}"']  & \cdot \end{tikzcd} \right)
 \]

Now to show that $A$ is a Segal type, we must show that for all $x,y,w:A$ and $f:\hom_A(x,y)$ and $k:\hom_A(y,w)$, the type
\[
  \sum_{\ell:\hom_A(x,w)} \homtwo{A}(x,y,w,f,k,\ell)
\]
is contractible. Since $A$ is discrete, the hom types are equivalent to identity types, and so by path induction we may reduce to the case $y \jdeq w$ and $k \jdeq \idarr y$.
In this case we have

\[
  \Parens{\sum_{\ell:\hom_A(x,y)} \homtwo{A}(x,y,y,f,\idarr y,\ell)} \simeq \ndexten{\Delta^2}{A}{\Lambda^2_1}{[f,\idarr y]}.
\]
To show that this type is contractible, we observe that it is a retract of the type \eqref{eq:discrete-retract}, using the construction of \cref{prop:identity} to construct the inclusion of a 2-simplex to a diagram of shape $\Delta^1 \times \Delta^1$ in which one of the new edges is an identity. 
\end{proof}

\begin{rmk}
  The $(\infty,1)$-topos of simplicial $\infty$-groupoids, which is presented by our motivating model of bisimplicial sets, is a \emph{cohesive $(\infty,1)$-topos} in the sense of~\cite{dcct}, i.e.\ its global sections functor to $\infty$-groupoids has both a right adjoint (``codiscrete objects'') and a left adjoint (``discrete objects'') that has a further product-preserving left adjoint (see also~\cite{lcoh}).
  In this model, the discrete objects defined above coincide with those in the image of the ``discrete objects'' functor.
  It would thus be natural to enhance our type theory with modalities representing the discrete reflection, discrete coreflection, and codiscrete coreflection, as in~\cite{bfp}.
  The discrete reflection, in particular, should be constructible by ``localizing at $\two$''.
\end{rmk}

\section{Covariantly functorial type families}
\label{sec:covariant}

Let $C:A\to \univtype$ be a type family.
Given $x,y:A$ and $f:\hom_A(x,y)$ and $u:C(x)$ and $v:C(y)$, we define the \textbf{dependent hom-type} from $u$ to $v$ over $f$ to be
\begin{equation}
  \hom_{C(f)}(u,v) \defeq \exten{s:\two}{C(f(s))}{\partial\Delta^1}{[u,v]}.\label{eq:dependent-arrows}
\end{equation}
Intuitively, $\hom_{C(f)}(u,v)$ is the type of arrows from $u$ to $v$ in the total space of $C$ that lie over $f$.
In particular, we see that $C$ associates to every arrow $f:\hom_A(x,y)$ a \emph{span}
\[ \begin{tikzcd}
 & {\textstyle\sum_{u:C(x)}\sum_{v:C(y)} \hom_{C(f)}(u,v)} \arrow[dl] \arrow[dr] & \\
  C(x) && C(y) 
  \end{tikzcd}
\]
In general, the types $C(x)$ do not depend \emph{functorially} on $x:A$ in the usual sense, so $C$ cannot be regarded as a functor from $A$ to a category of groupoids or categories.
One might hope that it could be regarded as a functor to a category whose morphisms are spans, but this also fails; if in the context of
\begin{mathpar}
  u:C(x) \and v:C(y) \and w:C(z) \and
  k:\hom_{C(f)}(u,v) \and m:\hom_{C(g)}(v,w) \and n:\hom_{C(h)}(u,w) \and
  t:\homtwo{A}(x,y,z,f,g,h)
\end{mathpar}
we define similarly
\[ \homtwo{C(t)}(u,v,w,k,m,n) \defeq \exten{s:\Delta^2}{C(t(s))}{\partial\Delta^2}{[k,m,n]} \]
we see that when $h=g\circ f$, the span $C(h)$ is not necessarily the composite of the spans $C(f)$ and $C(g)$, but is only related to them by a ``higher span''.

We will say that $C$ is \emph{covariant} if each $C(x)$ is a groupoid, i.e.~discrete in the sense of \cref{defn:discrete}, and moreover all of these spans are suitably representable, so that these $\infty$-groupoids do depend functorially on $A$.\footnote{One can also consider more general \emph{cocartesian} dependent types where the fibers are \emph{categories} (i.e.\ Segal or Rezk types) depending functorially, but we leave those for later work.}
Fortunately, as with Segal types, it turns out to be sufficient to ask for one contractibility condition, which we introduce presently; {in \cref{thm:covar-discrete}, we see that this condition implies that the fibers of a covariant fibration are necessarily discrete.} We prove that the total space of a covariant family over a Segal type is itself a Segal type.  
{ We show that any fiberwise map between covariant type families induces a ``natural transformation,'' commuting with the functorial actions of the arrows in the base type.} 
Then we turn our attention to the question of multivariable functoriality. 

The prototypical example of a covariant family is the ``representable'' type family associated to a term $a : A$ in a Segal type.
In \cref{sec:yoneda-lemma} we will state and prove versions of the Yoneda lemma involving this notion of representable family.

\subsection{Covariant fibrations}

See \cref{rmk:boavida} for a semantic justification of the following definition in the bisimplicial set model.

\begin{defn}\label{defn:covariant-family}
  We say that a type family $C:A\to\univtype$ is \textbf{covariant} if for every $f:\hom_A(x,y)$ and $u:C(x)$, the type \[\sum_{v:C(y)} \hom_{C(f)}(u,v)\] is contractible.
\end{defn}

Dually, $C$ is \textbf{contravariant} if for every $f:\hom_A(x,y)$ and $v:C(y)$, the type \[\sum_{u:C(x)} \hom_{C(f)}(u,v)\] is contractible; see Remark \ref{rmk:duality}.
Often we will assume that $A$ is a Segal type.

\begin{rmk}\label{rmk:covariant-pullbacks}
Note that the condition that characterizes a covariant fibration is stable under substitution (i.e.\ precomposition or reindexing).
That is, if $g : B \to A$ is a function and  $C : A \to \univtype$ is a covariant type family, then $\lam{b}C(g(b)) : B \to \univtype$ is also covariant.
\end{rmk}

As for Segal types, we can reformulate this using \cref{thm:exten-compose}:

\begin{prop}\label{prop:covariance-as-extension-type} A type family $C : A \to \univtype$ is covariant if and only if for all $f : \hom_A(x,y)$ and $u : C(x)$ there is a unique lifting of $f$ that starts at $u$; {i.e.\ the type 
\(\exten{t:\two}{C(f(t))}{0}{u}\)
is contractible.}
\end{prop}
\begin{proof}
By \cref{defn:hom} and \cref{thm:exten-compose}
\begin{align*}
  \tsum_{v:C(y)} \hom_{C(f)}(u,v)
  &\jdeq \tsum_{v:C(y)} \exten{t:\two}{C(f(t))}{0\lor 1}{[u,v]}\\
  &\simeq \exten{t:\two}{C(f(t))}{0}{u}.
\end{align*}
Thus, \cref{defn:covariant-family} asserts that $C$ is covariant if and only if the type of extensions of $u$ over $f$ is contractible.
\end{proof}

On the other hand, for a more global view of covariance, let us write $\tilde C \defeq {\sum_{z:A} C(z)}$ and denote the projection by $\pi : \tilde C \to A$.

\begin{thm}\label{thm:covariance-as-pullback} A type family  $C : A \to \univtype$ is covariant if and only if the square
\[ 
\begin{tikzcd}
 \tilde C^\two \arrow[r, "{\pi^\two}"] \arrow[d, "{\mathsf{ev}_0}"'] &
  A^\two \arrow[d, "{\mathsf{ev}_0}"] \\
  \tilde C \arrow[r, "\pi"'] &
  A 
  \end{tikzcd}\]
is a (homotopy) pullback.
\end{thm}
\begin{proof}
For each $(x,u) : \tilde C$, we have a projection map
\begin{equation}
  \Parens{\sum_{(y,v):\tilde C} \Parens{\sum_{f:\hom_A(x,y)} \hom_{C(f)}(u,v)}} \to \Parens{\sum_{y:A} \hom_A(x, y)}\label{eq:cvpbproj}
\end{equation}
whose fibers are the types $\sum_{v:C(y)} \hom_{C(f)}(u,v)$.
Thus, \cref{defn:covariant-family} is equivalent to the condition that this projection is an equivalence.
If we substitute the equivalence  \[\hom_{\tilde C}((x,u),(y,v)) \simeq \Parens{\sum_{f:\hom_A(x,y)} \hom_{C(f)}(u,v)}\]
established by \cref{thm:exten-nonac} into~\eqref{eq:cvpbproj}, and write $p=(x,u)$ and $q=(y,v)$, it becomes
\begin{equation}
  (\pi,\extfn\pi) : \Parens{\sum_{q:\tilde C} \hom_{\tilde C}(p,q)} \to \Parens{\sum_{y:A} \hom_A(\pi p, y)}.\label{eq:cvpbproj2}
\end{equation}
Thus, \cref{defn:covariant-family} is equivalent to saying that~\eqref{eq:cvpbproj2} is an equivalence for each $p:\tilde{C}$.
And since a fiberwise map is a fiberwise equivalence if and only if it induces an equivalence on total spaces, this is equivalent to asking that
\[ (\pi,\pi,\extfn\pi) : \Parens{\sum_{p : \tilde C} \sum_{q:\tilde C} \hom_{\tilde C}(p,q)} \to \Parens{\sum_{p : \tilde C}\sum_{y : A} \hom_A(\pi p, y)}\]
is an equivalence.
Finally, since $\sum_{p : \tilde C} \sum_{q:\tilde C} \hom_{\tilde C}(p,q)$ is equivalent to $\tilde C^\two$ by \cref{thm:exten-compose}, and similarly
\begin{align*}
  \Parens{\sum_{p : \tilde C}\sum_{y : A} \hom_A(\pi p, y)}
  &\simeq \sum_{p : \tilde C}\sum_{x:A} \sum_{e:\pi p=x}\sum_{y : A} \hom_A(\pi p, y)\\
  &\simeq \sum_{p : \tilde C}\sum_{x:A}\sum_{y : A}\sum_{f:\hom_A(\pi p, y)} (\pi p=x)\\
  &\simeq \sum_{p : \tilde C} \sum_{f:A^\two} (\pi p=\mathsf{ev}_0(x))\\
  &\jdeq \tilde C \times_{A} A^\two
\end{align*}
this is equivalent to saying that the square in the statement is a (homotopy) pullback.
\end{proof}

\begin{thm}\label{thm:covariant-segal}
  If $A$ is a Segal type and $C:A\to\univtype$ is covariant, then $\tilde C \defeq \sum_{z:A} C(z)$ is also a Segal type.
\end{thm}
We will give two proofs, or rather two versions of the same proof; one in type-theoretic language and one in category-theoretic language.
\begin{proof}[Type-theoretic proof]
  By \cref{thm:exten-nonac}, we have
  \begin{align*}
    (\Lambda^2_1 \to \tilde C)
    &\simeq \tsum_{\phi:\Lambda^2_1 \to A} \tprod_{t:\Lambda^2_1} C(\phi(t))
  \end{align*}
  We want to show that the type of extensions of any $(\phi,\psi)$ in this type to $\Delta^2$ is contractible.
  This type of extensions is similarly equivalent to
  \[ \sum_{\mu:\ndexten{\Delta^2}{A}{\Lambda^2_1}{\phi}} \exten{t:\Delta^2}{C(\mu(t))}{\Lambda^2_1}{\psi}. \]
  Since $A$ is Segal, the base of this dependent sum is contractible, so the sum itself is equivalent to
  \[ \exten{\pair{t,s}:\Delta^2}{C(\iscomp g f\pair{t,s})}{\Lambda^2_1}{\psi} \]
  where $f:\hom_A(x,y)$ and $g:\hom_A(y,z)$ are the arrows making up $\phi$.
  Now since $C$ is covariant, $\exten{s:\two}{C(g(s))}{0}{\psi(1,0)}$ is contractible, so it suffices to show that
  \[ \exten{\pair{t,s}:\Delta^2}{C(\iscomp g f\pair{t,s})}{s\jdeq 0}{\bar f} \]
  is contractible, where $\bar f : \tprod_{t:\two} C(f(t))$ is what is left of $\psi$.
  However, given any $\nu$ in this type, we can extend it to $\two\times\two$ with $\rec_\lor$, sending $\pair{t,s}$ to $\nu(t,s)$ if $s\le t$ and $\nu(t,t)$ if $s\ge t$; cf.~the proof of \cref{prop:two-simp-as-retract}.
  Thus, this type is a retract of 
  \[ \exten{\pair{t,s}:\two\times\two}{C(c\pair{t,s})}{s\jdeq 0}{\bar f} \]
  where $c:\two\times\two\to A$ is a similar extension of $\iscomp g f$; and this is equivalent to
  \[ \tprod_{t:\two}\exten{s:\two}{C(c\pair{t,s})}{s\jdeq 0}{\bar f(t)} \]
  and hence contractible, since it is a product of types that are each contractible by covariance of $C$.
\end{proof}
\begin{proof}[Categorical proof]
  It suffices to show that $\tilde C^{\Delta^2} \to \tilde C^{\Lambda^2_1}$ is an equivalence.
  Consider the following pair of squares:
  \begin{equation}
\begin{tikzcd}
    \tilde C^{\Delta^2} \arrow[r] \arrow[d] &
    \tilde C^{\Lambda^2_1} \arrow[d] \arrow[r, "d^2"] &
    \tilde C^{\two} \arrow[d]\\
    A^{\Delta^2} \arrow[r] &
    A^{\Lambda^2_1} \arrow[r,"{d^2}"] &
    A^{\two} 
    \end{tikzcd}\label{eq:cvsegrect}
  \end{equation}
  We have $A^{\Lambda^2_1} \simeq \sum_{p:A^\two} \sum_{z:A} \hom_A(p(1),z)$ and similarly for $\tilde C$, so covariance of $\pi$ implies that the right-hand square is a pullback.
  Since $A$ is a Segal type, the bottom-left arrow $A^{\Delta^2} \to A^{\Lambda^2_1}$ is an equivalence; thus it will suffice to show that the outer rectangle is also a pullback.

  Now we also have $A^{\two\times\two} \simeq \sum_{p:A^\two} \prod_{s:\two} \sum_{z:A} \hom_A(p(s),z)$ and similarly for $\tilde C$, so covariance of $\pi$ also implies that the square below is a pullback:
  \begin{equation}
\begin{tikzcd}[column sep=large]
    \tilde C^{\two\times\two}\arrow[d] \arrow[r, "{(\idfunc \two,[0])}"] &
    \tilde C^{\two} \arrow[d] \\
    A^{\two\times\two} \arrow[r, "{(\idfunc \two,[0])}"] &
    A^{\two} 
    \end{tikzcd}\label{eq:cvsegsq}
  \end{equation}
  Now we recall again that $A^{\Delta^2}$ is a retract of $A^{\two\times\two}$.
  Since this retraction is natural, the outer rectangle in~\eqref{eq:cvsegrect} is a retract of the square~\eqref{eq:cvsegsq}; hence it is also a pullback.
\end{proof}

\begin{rmk}[dependent composition]\label{rmk:dependent-comp}
Since $\sum_{z:A} C(z)$ is a Segal type, we can compose arrows in it.
However, it is often more useful to compose ``dependent arrows'' in the following sense.
Given $f:\hom_A(x,y)$ and $g:\hom_A(y,z)$ and also $k:\hom_{C(f)}(u,v)$ and $m:\hom_{C(g)}(v,w)$, we showed that for any $h:\hom_A(x,z)$ and ${t:\homtwoshort{A}(x,y,z,f,g,h)}$ the type
  \begin{equation}
    \sum_{n:\hom_{C(h)}(u,w)} \homtwo{C(t)}(u,v,w,k,m,n)\label{eq:dep-comp}
  \end{equation}
is contractible.
In the case when $h\defeq g\circ f$ and $t\defeq \iscomp g f$, we will write the specified inhabitant of~\eqref{eq:dep-comp} as $(m\circ k, \iscomp{m}{k})$.
Note that $m \circ k$ is, by definition, an arrow ``over'' $g\circ f$, and similarly $\iscomp m k$ is a 2-simplex over $\iscomp g f$.
\end{rmk}

We now show that each element of a Segal type gives rise to a covariant ``representable'' type family {and in fact the covariance condition on this type family characterizes Segal types}.

\begin{prop}\label{prop:covar-rep}
{  Let $A$ be a type and fix $a : A$. Then the type family \[\lam{x} \hom_A(a,x) : A\to \univtype\] is covariant if and only if $A$ is a Segal type.}
  \end{prop}
  \begin{proof}
  {The condition of \cref{defn:covariant-family} asserts that for each $b,c:A$, $f:\hom_A(a,b)$, and $g:\hom_A(b,c)$, the type
    \begin{equation}
      \sum_{h:\hom_A(a,c)} \exten{s:\two}{\hom_A(a,g(s))}{\partial\Delta^1}{[f,h]}
      \tag{$\dagger$}\label{eq:covar-rep-type}
    \end{equation}
  is contractible.}  Applying \cref{thm:exten-compose}, this is easily seen to be equivalent to
  \[ \ndexten{\two\times\two}{A}{d}{[\idarr a,f,g]} \]
  where $d$ is the ``cubical horn''
  \[
\left( \begin{tikzcd} \cdot \arrow[r, "\idarr a"] \arrow[d, "f"'] & \cdot \\ \cdot \arrow[r, "g"'] & \cdot \end{tikzcd} \right) \longrightarrow \left( \begin{tikzcd} \cdot \ar[r, "\idarr a"]\arrow[d, "f"']  \arrow[dr] & \cdot \arrow[d] \arrow[dl, phantom, "\cdot" very near start, "\cdot" very near end] \\ \cdot \arrow[r, "g"'] & \cdot \end{tikzcd} \right)
 \]
  But since $\two\times\two$ is the pushout of two copies of $\Delta^2$ over their diagonal faces, the type~\eqref{eq:covar-rep-type} is now also equivalent to
  \[ \sum_{k:\hom_A(a,c)} \left(\homtwo{A}(a,b,c,f,g,k) \times \sum_{h:\hom_A(a,c)} \homtwo{A}(a,a,c,\idarr a,h,k)\right) \]
  which by reassociating is equivalent to the dependent sum type
  \begin{equation}
    \sum\nolimits_{\sum\nolimits_{k:\hom_A(a,c)} \homtwoshort{A}(a,b,c,f,g,k)} \sum_{h:\hom_A(a,c)} \homtwo{A}(a,a,c,\idarr a,h,k)
    \tag{$\ddagger$}\label{eq:covar-rep-type-2}
  \end{equation}

  Now if $A$ is Segal, then by \cref{prop:htpy-is-htpy}, we have
  \[\left(\sum_{h:\hom_A(a,c)} \homtwo{A}(a,a,c,\idarr a,h,k)\right) \simeq \sum_{h:\hom_A(a,c)} (h=k),\]
  which is contractible. Thus, the dependent sum~\eqref{eq:covar-rep-type-2} is equivalent to its base type
  \begin{equation}
    \sum_{k:\hom_A(a,c)} \homtwo{A}(a,b,c,f,g,k)
    \tag{$**$}\label{eq:covar-rep-base-type}
  \end{equation}
  which is contractible, again because $A$ is Segal.

  Conversely, if $\lam{x} \hom_A(a,x)$ is covariant, then the dependent sum type~\eqref{eq:covar-rep-type-2} is contractible. While we cannot appeal to \cref{prop:htpy-is-htpy} (which was pointed out to us a few years ago by Bastiaan Cnossen and rediscovered more recently when formalizing these results), nevertheless the image of $\refl_k$ under the map \eqref{eq:idtosimplex} defines an inhabitant in each fiber of this dependent sum. This makes the base type~\eqref{eq:covar-rep-base-type} into a retract of the contractible~\eqref{eq:covar-rep-type-2}, which proves that $A$ is Segal.
\end{proof}

Of course, by duality, $\lam{x} \hom_A(x,a)$ is contravariant.

\subsection{Functoriality}
\label{sec:covar-funct}

If $C:A\to\univtype$ is covariant, then the arrows of $A$ act on $C$ in the following way.
Given $f:\hom_A(x,y)$ and $u:C(x)$, by assumption $\sum_{v:C(y)} \hom_{C(f)}(u,v)$ is contractible, and in particular inhabited.
We write its specified inhabitant as $(\covtr f u, \istrans f u)$.

\begin{ex}\label{ex:representable-transport}
  In the case of the covariant representable $\lam{x} \hom_A(a,x) : A\to \univtype$, suppose $e:\hom_A(a,x)$ and $f:\hom_A(x,y)$.
  Then the proof of \cref{prop:covar-rep} shows that $\covtr f e = f\circ e$.
\end{ex}

We have an analogue of \cref{prop:htpy-is-htpy}.
Note that this is also a directed version of the usual characterization of ``dependent paths'' in homotopy type theory as paths whose domain is a transport~\cite[(6.2.2)]{hottbook}.

\begin{lem}\label{thm:covtr-is-eq}
  If $C:A\to\univtype$ is covariant and $f:\hom_A(x,y)$ and $u:C(x)$ and $v:C(y)$, then
  \[ \hom_{C(f)}(u,v) \simeq \Parens{\covtr f u =_{C(y)} v} \]
\end{lem}
\begin{proof}
  Given $g:\hom_{C(f)}(u,v)$, we have $(v,g) : \sum_{w:C(y)}\hom_{C(f)}(u,w)$, hence $(\covtr f u,\istrans f u) = (v,g)$ by contractibility and so $\covtr f u = v$.
  This gives a map from left to right.
  To show that it is an equivalence, we observe that the induced map on total spaces
  \[ \Parens{\tsum_{v:C(y)} \hom_{C(f)}(u,v)} \to \Parens{\tsum_{v:C(y)}(\covtr f u = v)} \]
  is an equivalence, since both types are contractible.
\end{proof}

We now argue that the operation that takes $f:\hom_A(x,y)$ and $u:C(x)$ and produces $\covtr f u : C(y)$ is ``functorial'' in the expected sense:

\begin{prop}\label{prop:functorial-fibration} Suppose $A$ is a Segal type and $C:A\to\univtype$ is covariant. Then given $f:\hom_A(x,y)$, $g:\hom_A(y,z)$, and $u:C(x)$, we have 
\[\covtr g (\covtr f u) = \covtr{(gf)} u \qquad \mathrm{and} \qquad \covtr{(\idarr x)}u = u.\]
\end{prop}
\begin{proof}
Given $f:\hom_A(x,y)$ and $u:C(x)$, we have
\[ (\covtr f u, \istrans{f}{u}) : \sum_{v:C(y)} \hom_{C(f)}(u,v).\] 
Now suppose given also $g:\hom_A(y,z)$.
Then we have also
\[(\covtr g (\covtr f u), \istrans g {\covtr f u}) : \sum_{w:C(z)} \hom_{C(g)}(\covtr f u,w)\]
and
\[(\covtr{(gf)} u, \istrans{gf}{u}) : \sum_{w:C(z)} \hom_{C(gf)}(u,w). \]
On the other hand, the dependent composition $\istrans g {\covtr f u} \circ \istrans f u$ discussed in Remark \ref{rmk:dependent-comp} lies in the type $\hom_{C(gf)}(u,\covtr g (\covtr f u))$, and so we have
\[(\covtr g (\covtr f u),\istrans g {\covtr f u} \circ \istrans f u) : \sum_{w:C(z)} \hom_{C(gf)}(u,w). \]
Thus, since this type is contractible, we have $\covtr g (\covtr f u) = \covtr{(gf)} u$.

The case of identities is even easier.
Given $u:C(x)$, by definition we have
\[(\covtr{(\idarr x)}u,\istrans {\idarr x}u):\sum_{v:C(x)}\hom_{C(\idarr x)}(u,v)\]
but we also have a dependent identity arrow $\idarr u : \hom_{C(\idarr x)}(u,u)$ and so
\[(u,\idarr u) : \sum_{v:C(x)}\hom_{C(\idarr x)}(u,v) \]
By contractibility, therefore, $\covtr{(\idarr x)}u = u$.
\end{proof}

\subsection{Naturality}
\label{sec:covar-nat}

Any fiberwise map between two covariant fibrations $C,D:A\to\univtype$ defines a ``natural transformation'', commuting with the functorial action of \cref{prop:functorial-fibration}:

\begin{prop}\label{prop:natural-fibration}
Suppose given two covariant $C,D:A\to\univtype$ and a fiberwise map $\phi:\prod_{x:A} C(x) \to D(x)$.
Then for any $f:\hom_A(x,y)$ and $u:C(x)$, \[\phi_y(\covtr f u)= \covtr f (\phi_x(u)).\]
\end{prop}
\begin{proof}
We can apply $\phi$ to $\istrans f u$ to obtain
\[ (\phi_y(\covtr f u), \lam{t}\phi_{f(t)}(\istrans f u(t))) : \sum_{v:D(\phi_y(u))} \hom_{D(f)}(\phi_x(u),v). \]
But of course we also have
\[ (\covtr f (\phi_x(u)), \istrans f {\phi_x(u)}) : \sum_{v:D(\phi_y(u))} \hom_{D(f)}(\phi_x(u),v) \]
so by contractibility $\phi_y(\covtr f u)= \covtr f (\phi_x(u))$.
\end{proof}



\subsection{Discrete fibers}

The fibers of a covariant fibration over a Segal type are discrete types.

\begin{prop}\label{thm:covar-discrete} If $A$ is a Segal type and $C \colon A \to \univtype$ is a covariant type family, then for each $x : A$, the type $C(x)$ is discrete.
\end{prop}
\begin{proof}
  We must show that $\idtoarr: (u=v) \to \hom_{C(x)}(u,v)$ is an equivalence for all $u,v:C(x)$.
  It suffices to show that the induced map on total spaces
  \[ \Parens{\tsum_{v:C(x)} (u=v)} \to \Parens{\sum_{v:C(x)} \hom_{C(x)}(u,v)} \]
  is an equivalence.
  But its domain is contractible since it is a based path space, and its codomain is contractible by covariance of $C$ applied to $\idarr x : \hom_A(x,x)$.
\end{proof}

\begin{cor}\label{thm:segal-hom-discrete}
  If $A$ is a Segal type and $x,y:A$, then $\hom_A(x,y)$ is discrete.
\end{cor}
\begin{proof}
  By \cref{prop:covar-rep,thm:covar-discrete}.
\end{proof}

With \cref{thm:forall-discrete} it follows that various other types are discrete.
For instance, if $B,C:A\to \univtype$ are covariant, then the type $\prod_{a:A} (B(a) \to C(a))$ of ``natural transformations'' from $B$ to $C$ is discrete.
We also have:

\begin{cor}\label{thm:discrete-eq-discrete}
  If $A$ is discrete then so is $x=_A y$ for any $x,y:A$.
\end{cor}
\begin{proof}
  Since $A$ is discrete, it is Segal, hence $\hom_A(x,y)$ is discrete.
  But since $A$ is discrete, $x=_A y$ is equivalent to $\hom_A(x,y)$ and hence also discrete.
\end{proof}

\subsection{Multivariable covariance}
\label{sec:mult-covar}

We say that a type family dependent on multiple types is \textbf{covariant} if it is covariant in the ordinary sense when regarded as dependent on the $\Sigma$-type that collects all its arguments.
For instance, $C:A\to B\to \univtype$ is covariant if its uncurried version $C' : A\times B \to \univtype$ is covariant.
In this case we have:

\begin{prop}\label{thm:multivar-covar}
  $C:A\to B\to \univtype$ is covariant if and only if $C(a,-)$ is covariant for each $a:A$ and $C(-,b)$ is covariant for each $b:B$.
\end{prop}
\begin{proof}
  ``Only if'' follows from \cref{rmk:covariant-pullbacks}.
  For ``if'', note first that by \cref{thm:exten-nonac} we have $\hom_{A\times B}((a,b),(a',b'))\simeq \hom_A(a,a')\times \hom_B(b,b')$. For  $f:\hom_A(a,a')$ and $g:\hom_B(b,b')$, we write $(f,g):\hom_{A\times B}((a,b),(a',b'))$ by abuse of notation.
{ By \cref{prop:covariance-as-extension-type} we must show that for any $u:C(a,b)$ the type
 \begin{equation}
    \exten{t:\two}{C(f(t),g(t))}{0}{u}.\label{eq:2varcontr-exten}
  \end{equation}
  is contractible.}
  We will show that this is a retract of
  \begin{equation}
    \exten{\pair{t,s}:\two\times\two}{C(f(t),g(s))}{\pair{0,0}}{u}\label{eq:2varcov-contr}
  \end{equation}
  and that~\eqref{eq:2varcov-contr} is contractible.
  For the latter, we rewrite~\eqref{eq:2varcov-contr} using \cref{thm:exten-curry,thm:exten-compose} as
  \[ \tsum_{\phi:\exten{t:\two}{C(f(t),b)}{0}{u}} \tprod_{t:\two} \exten{s:\two}{C(f(t),g(s))}{0}{\phi(t)} \]
  Now ${\exten{t:\two}{C(f(t),b)}{0}{u}}$ is contractible since $C(-,b)$ is covariant, with center $\istrans f u$.
  So~\eqref{eq:2varcov-contr} is equivalent to
  \[ \tprod_{t:\two} \exten{s:\two}{C(f(t),g(s))}{0}{(\istrans f u)(t)} \]
  But since $C(f(t),-)$ is covariant, $\exten{s:\two}{C(f(t),g(s))}{0}{(\istrans f u)(t)}$ is contractible for any $t:\two$; thus~\eqref{eq:2varcov-contr} is contractible by relative function extensionality.

  It remains to show that~\eqref{eq:2varcontr-exten} is a retract of~\eqref{eq:2varcov-contr}.
  The retraction is just evaluation on the diagonal: $\phi \mapsto \lam{t} \phi(t,t)$.
  For the section, suppose given $\phi:\exten{t:\two}{C(f(t),g(t))}{0}{u}$.  We want to define an element of $\prod_{\pair{t,s}} C(f(t),g(s))$; this will be defined by gluing together a pair of 2-simplices defined for $t \le s$ and $s \le t$ respectively that restrict judgmentally to  $\phi(t)$ on the 1-simplex $t = s$.

  Recall from \cref{prop:connections} that we have a connection square $\connmax g$ with the following faces:
  \[
  \begin{tikzcd}
    b \arrow[r, "g"] \arrow[dr, "g" description] \arrow[d, "g"'] & b' \arrow[d, "\idarr{b'}"] \\ b' \arrow[r, "\idarr{b'}"']  \arrow[ur, phantom, "\cdot" very near start, "\cdot" very near end] & b'
  \end{tikzcd}
  \]
  We define
  \[ g_t \defeq \lam{s} \connmax g(t,s) : \hom_B(g(t),b'). \]
  Thus if $s\le t$ then $g_t(s)\jdeq g(t)$, while if $t\le s$ then $g_t(s)\jdeq g(s)$.
  Thus we have the covariant lifting arrow $\istrans{g_t}{\phi(t)} : \hom_{C(f(t),g_t)}(\phi(t),\covtr{(g_t)}{(\phi(t))})$ with respect to the type family $C(f(t),-)$, and evaluating it at $s$ we have $\istrans{g_t}{\phi(t)}(s): C(f(t),g_t(s))$.

  Similarly, we define $f_s \defeq \lam{t} \connmax f(t,s)$, yielding $\istrans{f_s}{\phi(s)}(t): C(f_s(t),g(s))$.
Since $C(f(t),g_t(s)) \jdeq C(f(t),g(s))$ for $t\le s$ and $C(f_s(t),g(s)) \jdeq C(f(t),g(s))$ for $s\le t$, we would like to paste these together with $\rec_\lor$ to get
  \[ \maybe{\psi(t,s) =
  \begin{cases}
    \istrans{g_t}{\phi(t)}(s) &\quad t\le s\\
    \istrans{f_s}{\phi(s)}(t) &\quad s\le t
  \end{cases}}
  \]
  But unfortunately we do not know that these two values agree when $s\jdeq t$.
  We know that $\istrans{g_t}{\phi(t)}(0) \jdeq \phi(t)$ and $\istrans{f_s}{\phi(s)}(0)\jdeq \phi(s)$, but although $g_t$ is constant for $0\le s\le t$ it doesn't follow that $\istrans{g_t}{\phi(t)}$ is constant on that range, so that $\istrans{g_t}{\phi(t)}(t)$ might not equal $\phi(t)$.
  Put differently, we have
  \begin{equation}
    \lam{\pair{t,s}} \istrans{g_t}{\phi(t)}(s) : \exten{\pair{t,s}\mid t\le s}{C(f(t),g(s))}{t\jdeq s}{\lam{\pair{t,s}}\istrans{g_t}{\phi(t)}(t)}\label{eq:tweakfrom}
  \end{equation}
  whereas we need something in
  \begin{equation}
    \exten{\pair{t,s}\mid t\le s}{C(f(t),g(s))}{t\jdeq s}{\lam{\pair{t,s}}\phi(t)}.\label{eq:tweakto}
  \end{equation}

  Consider the 2-simplex that we would like to be degenerate but isn't, $\istrans{g_t}{\phi(t)}(s)$ for $s\le t$.
  We can use $\rec_\lor$ to put this together with a 2-simplex $\istrans{g_t}{\phi(t)}(t)$ for $t \le s$ that is degenerate:
  \[
  h(t,s) \defeq
  \begin{cases}
    \istrans{g_t}{\phi(t)}(s)&\quad s\le t\\
    \istrans{g_t}{\phi(t)}(t) &\quad t\le s
  \end{cases}
  \]
  Then $h(t,s) : C(f(t),g(t))$ for all $t$ and $s$, since $g_t(s)\jdeq g(t)$ for $s\le t$.
  Thus for each $t$, we have an arrow $\lam{s} h(t,s) : \hom_{C(f(t),g(t))}(\phi(t),\istrans{g_t}{\phi(t)}(t))$.
  Since \cref{thm:covar-discrete} proves that $C(f(t),g(t))$ is discrete, this yields an equality $\phi(t) = \istrans{g_t}{\phi(t)}(t)$, and thus  an equality $(\lam{\pair{t,s}}\phi(t)) = (\lam{\pair{t,s}}\istrans{g_t}{\phi(t)}(t))$ by relative function extensionality.
  Therefore, we can transport~\eqref{eq:tweakfrom} along this equality to get an element of~\eqref{eq:tweakto} as desired.
  We argue similarly on the opposite side to obtain $\phi(s)= \istrans{f_s}{\phi(s)}(s)$ and another 2-simplex in
  \[
      \exten{\pair{t,s}\mid s\le t}{C(f(t),g(s))}{t\jdeq s}{\lam{\pair{t,s}}\phi(s)}
      \] that we can glue with this one, giving the desired section.
\end{proof}

We can also consider such families where one variable depends on another one.
For instance, $C : \prod_{a:A} (B(a) \to\univtype)$ is covariant if its uncurried version $C' : \left(\sum_{a:A} B(a)\right) \to \univtype$ is covariant.
A fundamental example is the following.

\begin{thm}\label{thm:eq-covar}
  Suppose $C:A\to\univtype$ is covariant.
  Then
  \[\lam{a}\lam{u}\lam{v} (u=v) : \prod_{a:A} (C(a) \to C(a) \to \univtype)\]
  is also covariant.
\end{thm}
\begin{proof}
  The family $(\lam{a} C(a)\times C(a)) : A \to \univtype$ is covariant, so an arrow in its total space is uniquely determined by an arrow $f:\hom_A(a,a')$ and a lift $(u,v) : C(a)\times C(a)$ of its domain.
  We denote the resulting uniquely determined arrow by
  \[ \phi^f_{u,v} : \hom_{\sum_{a:A} C(a)\times C(a)}((a,u,v),(a',\covtr f u,\covtr f v)). \]
  By \cref{thm:exten-nonac}, the type of $\phi^f_{u,v}$ is equivalent to
  \[ \sum_{f:\hom_A(a,a')} \hom_{C(f)}(u,\covtr f u) \times \hom_{C(f)}(v,\covtr f v) \]
  and under this equivalence $\phi^f_{u,v}$ corresponds to the triple $(f,\istrans f u, \istrans f v)$.

  Now suppose $p:u=v$; we want to show that the following type is contractible:
  \[ \exten{t:\two}{\istrans f u(t) =_{C(f(t))} \istrans f v(t)}{0}{p}. \]
  By path induction, we are free to assume that $v$ is $u$ and that $p$ is $\refl$.
  However, now by relative function extensionality we have
  \[ \exten{t:\two}{\istrans f u(t) =_{C(f(t))} \istrans f u(t)}{0}{\refl_u}
  \simeq
  \left(\istrans f u =_{\exten{t:\two}{C(f(t))}{0}{u}} \istrans f u\right)
  \]
  and the latter is contractible since it is a path space in a type that is itself contractible, since $C$ is covariant.
\end{proof}

It is also useful to identify the covariant transport in such a family.

\begin{prop}\label{thm:eq-covtr-ap}
  With notation as in the proof of \cref{thm:eq-covar}, for any equality $e:u=_{C(a)} v$ we have $\covtr {(\phi^f_{u,v})}{e} = \ap_{\covtr f }(e)$.
\end{prop}
\begin{proof}
  By path induction, we assume $v\jdeq u$ and $e\jdeq \refl$.
  But then $\lam{t} \refl_{\istrans f u(t)}$ is a lift of $\phi^f_{u,u}$ starting at $\refl_u$ and ending at $\refl_{\covtr f u}$, so $\covtr {(\phi^f_{u,u})}{\refl_u} = \refl_{\covtr f u}$, which is by definition $\ap_{\covtr f }(\refl_u)$.
\end{proof}

\subsection{Two-sided discrete fibrations}

We now consider type families dependent on multiple types with opposite variance.

\begin{defn}\label{defn:two-sided-discrete}
  Let $A$ and $B$ be Segal types and let $C : A \to B \to \univtype$ be a type family. We say that $C$ is \textbf{contravariant over $A$ and covariant over $B$} if for all $a : A$ and $b : B$ the type families
  \[ \lam{y}C(a,y) : B \to \univtype \qquad \mathrm{and} \qquad \lam{x}C(x,b) : A \to \univtype\] respectively define a covariant family over $B$ and a contravariant family over $A$.
\end{defn}

In classical category theory, fibrations of the form of \cref{defn:two-sided-discrete} are called \textbf{two-sided discrete fibrations}. The prototypical example is given by \cref{prop:covar-rep} and its dual:

\begin{prop}\label{prop:hom-bifunctor}
If $A$ is a Segal type, then the type family
\[ \lam{x}\lam{y} \hom_A(x,y) : A \to A \to \univtype\] is a two-sided discrete fibration.\qed
\end{prop}

\subsection{Closure properties of covariance}
\label{sec:closure-of-covariance}

Mapping into a covariant family (even dependently) preserves covariance; while mapping out of a covariant family, at least into a discrete type, yields a contravariant family.
The former is easy to prove, but the latter is rather trickier.

\begin{thm}\label{thm:prod-covar}
  Let $C:A\to B\to \univtype$ be such that each $C(-,b)$ is covariant.
  Then $\lam{a} \prod_b C(a,b) : A \to \univtype$ is also covariant.
\end{thm}
\begin{proof}
  By \cref{prop:covariance-as-extension-type} we must show that every $f:\hom_A(a,a')$ has a unique lifting that starts at $g:\prod_{b:B} C(a,b)$. By \cref{thm:exten-prod-commute}, the type of such extensions, displayed below-left, is equivalent to the dependent function type displayed below-right:
  \begin{align*}
    \exten{t:\two}{\Parens{\tprod_{b:B} C(f(t),b)}}{0}{g}
    &\simeq \tprod_{b:B} \exten{t:\two}{C(f(t),b)}{0}{g(b)}.
  \end{align*}
  Since $C(-,b)$ is covariant, each $\exten{t:\two}{C(f(t),b)}{0}{g(b)}$ is contractible, hence so is the right-hand side.
\end{proof}

\begin{thm}\label{thm:covar-into-discrete}
  Let $C:A\to\univtype$ be covariant and let $Y$ be discrete.
  Then $\lam{a} (C(a)\to Y): A \to \univtype$ is contravariant.
\end{thm}
\begin{proof}
  Fix $f:\hom_A(a,a')$ and $v:C(a')\to Y$; we must show that
  \begin{equation}
    \exten{t:\two}{(C(f(t))\to Y)}{1}{v}\label{eq:covar-disc-contr}
  \end{equation}
  is contractible.
  Recall from \cref{prop:connections} that $f$ gives rise to a square $\connmax f$ with the following faces:
  \[
  \begin{tikzcd}
    a \arrow[r, "f"] \arrow[dr, "f" description] \arrow[d, "f"'] & a' \arrow[d, "\idarr{a'}"] \\ a' \arrow[r, "\idarr{a'}"']  \arrow[ur, phantom, "\cdot" very near start, "\cdot" very near end] & a'
  \end{tikzcd}
  \]
  As in \cref{thm:multivar-covar}, we define
  \[ f_t \defeq \lam{s} \connmax f(t,s) : \hom_A(f(t),a'). \]
  Thus if $s\le t$ then $f_t(s)\jdeq f(t)$, while if $t\le s$ then $f_t(s)\jdeq f(s)$.
  Then for any $t$ and any $c:C(f(t))$ we have $\covtr{(f_t)} c : C(a')$ and hence $v(\covtr{(f_t)} c) : Y$, so~\eqref{eq:covar-disc-contr} is inhabited by $\lam{t} \lam{c} v(\covtr{(f_t)} c)$.

  It remains to show that any element of~\eqref{eq:covar-disc-contr} is equal to $\lam{t} \lam{c} v(\covtr{(f_t)} c)$.
  Thus, let $\phi:\exten{t:\two}{(C(f(t))\to Y)}{1}{v}$; by relative function extensionality it suffices to fix $t:\two$ and $c:C(f(t))$ and show $\phi(t,c) = v(\covtr{(f_t)} c)$. And in fact, since $Y$ is discrete, it suffices to define an arrow in $\hom_Y(\phi(t,c), v((f_t)_*c))$.

  Now we have $\istrans{f_t}{c} : \hom_{C(f_t)}(c,\covtr{(f_t)} c)$, and thus for any $s:\two$ we have $\istrans{f_t}{c}(s) : C(f_t(s))$.
  Thus if $s\le t$, we can write $\phi(t,\istrans{f_t}{c}(s))$; while we can always write $\phi(t,\istrans{f_t}{c}(t))$ since $f_t(t)\jdeq f(t)$ for any $t$.
    Using $\rec_\lor$, we may paste these 2-simplices together to define
  \[
  k(t,s) \defeq
  \begin{cases}
    \phi(t,\istrans{f_t}{c}(s)) &\quad s \le t \\
    \phi(t,\istrans{f_t}{c}(t)) &\quad t \le s.
  \end{cases}
  \]
   For each $t$, this provides an arrow $\lam{s} k(t,s): \hom_Y(\phi(t,c),    \phi(t,\istrans{f_t}{c}(t)))$ 
     since $\istrans{f_t}{c}(0)\jdeq c$.

Similarly, using $\rec_\lor$, we may paste  together a pair of 2-simplices to define
  \[
  h(t,s) \defeq
  \begin{cases}
    \phi(t,\istrans{f_t}{c}(t)) &\quad s \le t \\
    \phi(s,\istrans{f_t}{c}(s)) &\quad t \le s.
  \end{cases}
  \]
For each $t$, this provides an arrow $\lam{s} h(t,s) : \hom_Y( \phi(t,\istrans{f_t}{c}(t)) , v((f_t)_*c))$ since $    \phi(1,\istrans{f_t}{c}(1)) \jdeq \phi(1, (f_t)_*(c)) \jdeq v((f_t)_*c)$. Thus $\phi(t,c) =  \phi(t,\istrans{f_t}{c}(t))$ and $ \phi(t,\istrans{f_t}{c}(t)) =  v((f_t)_*c)$ since $Y$ is discrete.
\end{proof}

\section{The Yoneda lemma}
\label{sec:yoneda-lemma}

Let $C:A\to \univtype$ be covariant, and fix $a:A$.
Then we have maps
\begin{alignat*}{2}
  \evid^C_a &\defeq \lam{\phi}\phi(a,\idarr a) &\quad:\quad& \Parens{\prod_{x:A} (\hom_A(a,x) \to C(x))} \to C(a)\\
  \yon^C_a &\defeq \lam{u} \lam{x}\lam{f} \covtr f u &\quad:\quad& C(a) \to \Parens{\prod_{x:A} (\hom_A(a,x) \to C(x))}.
\end{alignat*}

\begin{thm}[Yoneda lemma]\label{thm:yoneda}
  If $A$ is a Segal type, then for any covariant $C:A\to \univtype$ and $a : A$, the maps $\evid^C_a$ and $\yon^C_a$ are inverse equivalences.
\end{thm}
\begin{proof}
In one direction, given $u : C(a)$ we have
\[ (\lam{x}\lam{f} \covtr f u)(a,\idarr a) = \covtr {(\idarr a)}u = u\] by \cref{prop:functorial-fibration}. 
In the other direction, we want to compare the fiberwise map $\phi$ to $\lam{x}\lam{f} \covtr f (\phi(a,\idarr a))$.
By function extensionality, we can evaluate both of them at some $x:A$ and $f:\hom_A(a,x)$, in which case we have
\[ \covtr f (\phi(a,\idarr a)) = \phi(x,\covtr f \idarr a) = \phi(x,f \circ \idarr a) = \phi(x,f) \]
by \cref{prop:natural-fibration} and \cref{ex:representable-transport}.
\end{proof}

This is of course just the usual proof of the Yoneda lemma.
However, note that we do not need to manually check the naturality of $\yon_a^C(u)$; this is automatic simply by its being defined as a fiberwise map.
Similarly, because its domain and codomain are both covariant in $a:A$ (the domain by \cref{thm:prod-covar,thm:covar-into-discrete} --- or simply by the fact of being fiberwise equivalent to the codomain), the Yoneda equivalence is itself automatically natural in $a$.
Naturality in $C$ is not similarly automatic, but is easy to prove:

\begin{lem}\label{thm:yoneda-nat}
  If $A$ is a Segal type and $a:A$, while $C,D:A\to\univtype$ are covariant and $\psi:\prod_{x:A} (C(x) \to D(x))$, then we have
  \begin{align*}
    \psi_a \circ \evid^C_a &\jdeq \evid^D_a \circ (\lam{\phi}\lam{x}\lam{f} \psi_x(\phi(x,f)))\\
    (\lam{\phi}\lam{x}\lam{f} \psi_x(\phi(x,f))) \circ \yon^C_a &= \yon^D_a \circ \psi_a
  \end{align*}
\end{lem}
\begin{proof}
  The first is simple $\beta$-reduction: both sides equal $\lam{\phi} \psi_a(\phi(a,\idarr a))$.
  In the second, the left-hand side equals $\lam u \lam x \lam f \psi_x(\covtr f u)$ while the right-hand side equals $\lam u \lam x \lam f \covtr f \psi_a(u)$; thus the equality follows from \cref{prop:natural-fibration} (and function extensionality).
\end{proof}

\begin{defn}\label{defn:yoneda-embedding}
When $a,a':A$ are terms in a Segal type, we refer to the map
\[   \yon^{\hom_A(a',-)}_a  :  \hom_A(a',a) \to \Parens{\prod_{x:A} \hom_A(a,x) \to \hom_A(a',x)}\]
as the \textbf{Yoneda embedding}. 
\end{defn}

\begin{rmk}\label{rmk:rep-nat-trans} Because the Yoneda embedding is an equivalence, we know that any fiberwise map $\phi : \prod_{x:A} \hom_A(a,x) \to \hom_A(a',x)$ between covariant representables for a Segal type $A$ is equal to a post-composition function. Namely, if $u \defeq \evid_a^{\hom_A(a',-)}(\phi) : \hom_A(a',a)$ then by \cref{ex:representable-transport}, 
\[\phi = \yon_a^{\hom(a',-)}(u) = \lam{x}\lam{f} \covtr f u = \lam{x}\lam{f} f \circ u,\] which is to say the natural transformation $\phi$ is given by precomposition with the arrow $u : \hom_A(a',a)$.
\end{rmk}

From a type-theoretic perspective, the Yoneda lemma is a ``directed'' version of the ``transport'' operation for identity types.
This suggests a ``dependently typed'' generalization of the Yoneda lemma, analogous to the full induction principle for identity types.
Recall from \cref{sec:mult-covar} that a type family $C:\prod_{x:A} (\hom(a,x)\to \univtype)$ is called covariant if its uncurried version $C:\Parens{\sum_{x:A} \hom(a,x)} \to \univtype$ is covariant.

\begin{thm}[dependent Yoneda lemma]\label{thm:dep-yoneda}
  If $A$ is a Segal type, $a:A$, and $C:\prod_{x:A} (\hom(a,x)\to \univtype)$ is covariant, then the function
  \[ \evid^C_a \defeq \lam{\phi}\phi(a,\idarr a) : \Parens{\tprod_{x:A} \tprod_{f:\hom_A(a,x)} C(x,f)} \to C(a,\idarr a) \]
  is an equivalence.
\end{thm}

We will obtain this as a special case of a result about types with initial objects.

\begin{defn}
  A point $b:B$ is \textbf{initial} if for all $x:B$ the hom-type $\hom_B(b,x)$ is contractible.
\end{defn}

\begin{thm}\label{thm:initial-yoneda}
  If $b:B$ is initial and $C:B\to \univtype$ is covariant, then the function
  \[ \lam{\phi}{\phi(b)} : \Parens{\tprod_{x:B} C(x)} \to C(b) \]
  is an equivalence.
\end{thm}
\begin{proof}
  Since each $\hom_B(b,x)$ is contractible, it is in particular inhabited, so we have some $f:\tprod_{x:B} \hom_B(b,x)$.
  Moreover, since $\hom_B(b,b)$ is contractible, we have $f_b = \idarr b$.

  Now for an inverse to the above map, we send $u:C(b)$ to $\lam{x} \covtr{(f_x)}{u}$.
  In one direction we have
  \[ (\lam{x} \covtr{f_x}{u})(b) \jdeq \covtr{(f_b)}{u} = \covtr{(\idarr b)}{u} = u.\]
  In the other direction, for any $x:B$ we have $f_x : \hom_B(b,x)$, and thus for any $\phi:\tprod_{x:B} C(x)$ we have $\lam{t} \phi(f_x(t)) : \hom_{C(f_x)}(\phi(b),\phi(x))$.
  Since $C$ is covariant, $\sum_{v:C(x)} \hom_{C(f_x)}(\phi(b),v)$ is contractible, so \[(\phi(x),\lam{t} \phi(f_x(t))) = (\covtr{(f_x)}{\phi(b)},\istrans{f_x}{u}),\] and in particular $\phi(x) = \covtr{(f_x)}{\phi(b)}$.
\end{proof}

\begin{lem}\label{thm:slice-initial}
  For any Segal type $A$ and $a:A$, the type $\sum_{x:A}\hom_A(a,x)$ has an initial object $(a,\idarr a)$.
\end{lem}
\begin{proof}
  Let $x:A$ and $f:\hom_A(a,x)$; we must show that \[\hom_{\sum_{x:A}\hom_A(a,x)}((a,\idarr a),(x,f))\] is contractible. {By \cref{thm:exten-nonac}, this type is equivalent to
\[\exten{t:\two}{\hom_A(a,f(t))}{0}{\idarr a}\] 
which is contractible by \cref{prop:covariance-as-extension-type} since $\hom_A(a,-)$ is covariant.}
\end{proof}

\begin{proof}[Proof of \cref{thm:dep-yoneda}]
  By \cref{thm:slice-initial}, $\sum_{x:A}\hom_A(a,x)$ has an initial object $(a,\idarr a)$. Thus \cref{thm:initial-yoneda} specializes to the desired result.
\end{proof}

A formula for the inverse of the dependent $\evid^C_a$
\[ \yon^C_a  : C(a,\idarr a) \to \Parens{\tprod_{x:A} \tprod_{f:\hom_A(a,x)} C(x,f)} \]
can be extracted from the above proofs.
Under the equivalent description of $\hom_{\sum_{x:A}\hom_A(a,x)}((a,\idarr a),(x,f))$ in \cref{thm:slice-initial}, a specific inhabitant of it is given by the connection square $\connmin f$ from \cref{prop:connections}.
Thus, we can write
\begin{equation}
  \yon^C_a(u,x,f) \defeq \covtr{(\connmin f)}{u}.\label{eq:dep-yon}
\end{equation}

We say that a covariant type family $C : A \to \univtype$ over a Segal type $A$ is \textbf{representable} if there exists some $a :A$ and a family of equivalences over $A$:
\[ \prod_{x :A} (\hom_A(a,x) \simeq C(x)).\]

\begin{prop}\label{thm:representability} A covariant type family $C : A \to \univtype$ over a Segal type $A$ is representable if and only if the type
\[ \sum_{x \in A} C(x)\] has an initial object $(a,u)$, in which case
\[ \yon^C_a(u) : \prod_{x :A} (\hom_A(a,x) \to C(x))\] defines an equivalence.
\end{prop}
\begin{proof}
If $C$ is representable, then there is an equivalence 
\[ \phi : \prod_{x :A} (\hom_A(a,x) \simeq C(x))\] corresponding under the Yoneda lemma to a term
$\evid^C_a(\phi) : C(a)$ defined by $\evid^C_a(\phi) = \phi(a,\idarr a)$. By \cref{thm:slice-initial}, $(a, \idarr a)$ is initial in $\sum_{x:A}\hom_A(a,x)$. Transporting along the equivalence, we conclude that $(a, \phi(a,\idarr a)) : \sum_{x :A} C(x)$ is initial and moreover that \[ \yon^C_a(\phi(a,\idarr a)) = \yon^C_a( \evid^C_a(\phi)) = \phi\] defines the postulated equivalence $\phi$.

Conversely, if $(a,u) : \sum_{x :A} C(x)$ is initial then we argue that 
\[ \yon^C_a(u) : \prod_{x :A} (\hom_A(a,x) \to C(x))\] defines an equivalence by showing that its fibers are contractible. From the definition of $\yon^C_a(u)$ the fiber of 
\[ \yon^C_a(u)_b : \hom_A(a,b) \to C(b)\] over $v : C(b)$ is \[ \sum_{f : \hom_A(a,b)} f_*(u)=_{C(b)} v.\]
By \cref{thm:covtr-is-eq}  and \cref{thm:exten-nonac}
\begin{align*}
  \Parens{\sum_{f : \hom_A(a,b)} f_*(u)=_{C(b)} v}
  &\simeq \Parens{\sum_{f : \hom_A(a,b)} \hom_{C(f)}(u,v)}\\
  &\simeq \hom_{\sum_{x:A}C(x)}((a,u),(b,v)),
\end{align*}
 which is contractible since $(a,u)$ is initial.
\end{proof}

Yoneda lemmas for bisimplicial sets have recently been studied by~\cite{RV4,kv:yoneda-css,rasekh:yoneda-ss}, with similar conclusions.
For instance, \cref{thm:representability} above corresponds to~\cite[Theorem 5.6]{rasekh:yoneda-ss}.

\section{Rezk types}
\label{sec:Rezk-types}

A Segal type is a type $A$ whose hom-types $\hom_A : A \to A \to \univtype$ are enhanced by a homotopically unique composition operation. A Rezk type is a Segal type that is ``complete'' or ``univalent'' in the sense that the identity type $x  =_A y$ between any two terms is equivalent to the type of isomorphisms $x \cong_A y$ that we now introduce.

\subsection{Isomorphisms}
\label{sec:isos}

Let $A$ be a Segal type and consider $f : \hom_A(x,y)$. As in ordinary category theory, we say $f$ is an isomorphism if it has a two-sided composition inverse. However, as for functions in ordinary homotopy type theory \cite[Chapter 4]{hottbook}, more care is required to define a type of witnesses for the invertibility of $f$ in such a way that it is contractible if it is inhabited. Guided by that experience, we define:
\[\isiso f \defeq \Parens{\sum_{g:\hom_A(y,x)} g\circ f = \idarr x}
\times \Parens{\sum_{h:\hom_A(y,x)} f\circ h = \idarr y} \]
and say that $f$ is an isomorphism if this type is inhabited.

\begin{prop}
  Let $A$ be a Segal type and $f:\hom_A(x,y)$.
  Then $f$ is an isomorphism if and only if we have $g:\hom_A(y,x)$ with $g\circ f = \idarr x$ and $f\circ g = \idarr y$.
\end{prop}
\begin{proof}
  ``Only if'' is easy; take $h\defeq g$.
  Conversely, from an inhabitant of $\isiso f$ we get $g$ and $h$ with $g\circ f = \idarr x$ and $f\circ h = \idarr y$, and then we can show
  \[ g = g \circ \idarr y = g \circ (f\circ h) = (g\circ f) \circ h = \idarr x \circ h = h \]
  and therefore $f\circ g = \idarr y$ as well.
\end{proof}

\begin{prop} Let $A$ be a Segal type and $f : \hom_A(x,y)$. Then the type $\isiso f$ is a proposition.
\end{prop}
\begin{proof}
If $f$ is an isomorphism witnessed by a left inverse $g$ and right inverse $h$, then for any $k:\hom_A(z,x)$ we have $k = (g\circ f)\circ k = g\circ (f\circ k)$, and for any $\ell:\hom_A(z,y)$ we have $\ell = (f\circ h)\circ \ell = f\circ (h\circ \ell)$.
Therefore, the function $(f\circ -) : \hom_A(z,x) \to \hom_A(z,y)$ has both a left and a right inverse, and hence it is an equivalence \cite[4.3.3]{hottbook}.

Since $\sum_{h:\hom_A(y,x)} f\circ h = \idarr y$ is a fiber of $(f\circ -)$, it is therefore contractible.
Similarly, the function $(-\circ f) : \hom_A(y,z) \to \hom_A(x,z)$ is an equivalence, so its fiber $\sum_{g:\hom_A(y,x)} g\circ f = \idarr x$ is contractible, and hence $\isiso f$ is also contractible. In other words, if $\isiso f$ is inhabited, then it is contractible.
Therefore, it is a proposition.
\end{proof}

Thus it makes sense to define \textbf{the type of isomorphisms} from $x$ to $y$ to be
\[ (x\cong_A y) \defeq \sum_{f:\hom_A(x,y)}\isiso f. \]

Consider now a pair of functions $f,g : X \to A$ where $X$ is a type or shape and $A$ is a Segal type. For any natural transformation $\alpha : \nat XA(f,g)$, if $\alpha$ is an isomorphism in $X \to A$ then clearly its components $\alpha_x : \hom_A(f(x),g(x))$ are isomorphisms in $A$. Conversely:

\begin{prop} Let $X$ be a type or shape, let $A$ be a Segal type, and consider $\alpha : \nat XA(f,g)$. Then the map
\[ \isiso \alpha \to \prod_{x : X} \isiso {\alpha_x}\] is an equivalence. That is, a natural transformation is an isomorphism if and only if it is a pointwise isomorphism. 
\end{prop}
\begin{proof}
  Since both sides are propositions, it suffices to assume $\prod_{x : X} \isiso {\alpha_x}$ and prove that $\alpha$ is an isomorphism.
  To define $\beta:\nat XA(g,f)$, we must assume $t:\two$ and then $x:X$, and define $\beta_x(t)$; but since $\alpha_x$ is an isomorphism it has an inverse, so we can take $\beta_x(t) = (\alpha_x)^{-1}(t)$, i.e.\ $\beta_x = (\alpha_x)^{-1}$.
  To show $\beta\circ\alpha = \idarr f$, by function extensionality it suffices to show that $(\beta \circ \alpha)_x = (\idarr f)_x$; but this follows by \cref{prop:componentwise-composition} since $\beta_x \circ \alpha_x = \idarr{f(x)}$.
  Similarly we have $\alpha\circ\beta = \idarr g$.
\end{proof}

This gives ``isomorphism extensionality'':

\begin{cor}\label{thm:isomorphism-extensionality}
  For $X$ a type or shape, $A$ a Segal type, and $f,g: X\to A$, we have
  \[ (f \cong_{A^X} g) \simeq \prod_{x :X} (fx \cong_A gx).\]
\end{cor}
\begin{proof}
  We have
  \begin{align*}
    \tprod_{x :X} (fx \cong_A gx)
    &\jdeq \tprod_{x :X} \sum_{\alpha_x:\hom_A(f(x),g(y))} \isiso {\alpha_x}\\
    &\simeq \tsum_{\alpha:\prod_{x:X}\hom_A(f(x),g(x))} \tprod_{x:X} \isiso {\alpha_x}\\
    &\simeq \tsum_{\alpha:\nat XA(f,g)} \tprod_{x:X} \isiso {\alpha_x}\\
    &\simeq  \tsum_{\alpha:\nat XA(f,g)} \isiso{\alpha}\\
    &\jdeq (f\cong_{A^X} g).\qedhere
  \end{align*}
\end{proof}

\subsection{Rezk-completeness}
\label{sec:completeness}

Of course, $\idarr x$ is always an isomorphism.
Thus, path induction allows us to define
\begin{equation}
  \idtoiso:\prod_{x,y :A} (x =_A y) \to (x\cong_A y)\label{eq:idtoiso}
\end{equation}
by reducing to the case where $x \jdeq y$ and our equality is $\refl_x : x =_A x$, which we map to $\idarr x : (x\cong_A x)$.

\begin{defn}\label{defn:rezk-complete} A Segal type  $A$ is \textbf{Rezk-complete}  if~\eqref{eq:idtoiso} is an equivalence, in which case we say that $A$ is a \textbf{Rezk type}.
\end{defn}

When working with Rezk types, it is useful to observe that $\idtoiso$ mediates between the type-theoretic operations on paths and the category-theoretic operations on arrows.

\begin{lem}\label{thm:idtoiso-trans}
  If $A$ is Segal and $C:A\to\univtype$ is covariant, while $e:x=_A y$, then for any $u:C(x)$ we have
  \[ \covtr{\idtoiso(e)}{u}= \transport{C}{e}{u} \]
  (The left-hand side is covariant transport along an arrow, while the right-hand side is homotopy-type-theoretic transport along an equality.)
\end{lem}
\begin{proof}
  By path induction on $e$: when $e\jdeq \refl$, both sides are equal to $u$.
\end{proof}

\begin{lem}\label{thm:idtoiso-ap}
  If $A$ and $B$ are Segal and $f:A\to B$, while $e:x=_A y$, then
  \[ \extfn f(\idtoiso(e)) = \idtoiso(\ap_f(e)) \]
\end{lem}
\begin{proof}
  By path induction on $e$: when $e\jdeq \refl$, both sides are equal to $\idarr{fx}$.
\end{proof}

Rezk types, like Segal types, are closed under function spaces.

\begin{prop} If $A$ is a Rezk type so is $X \to A$ for any type or shape $X$.
\end{prop}
\begin{proof}
For $f,g:X\to A$, the map $\idtoiso_{f,g}$ factors through (perhaps relative) function extensionality and isomorphism extensionality (\cref{thm:isomorphism-extensionality}) and the maps $\idtoiso_{fx,gx}$ for $A$:
\[ 
(f =_{A^X} g) \xto{\simeq} \Parens{\prod_{x :X} (fx =_A gx)} \xto{\prod_{x:X} \idtoiso_{fx,gx}} \Parens{\prod_{x :X} fx \cong_A gx} \xto{\simeq} (f \cong_{A^X} g).\]
If $A$ is a Rezk type, the middle map is an equivalence, hence so is the composite.
\end{proof}

We now observe that if Rezk types are the ``categories'', then discrete types are the ``groupoids''.

\begin{prop}\label{prop:discrete-are-groupoid}
  A type is discrete if and only if it is Rezk and all its arrows are isomorphisms.
\end{prop}
\begin{proof}
  Note that the composite
  \[ (x=_A y) \xto{\idtoiso} (x\cong_A y) \longrightarrow \hom_A(x,y) \]
  is the map $\idtoarr$ from \cref{defn:discrete}.
  Since being an isomorphism is a proposition, the inclusion $(x\cong_A y) \to \hom_A(x,y)$ is an embedding, and hence is an equivalence if and only if it is surjective (i.e.\ all arrows are isomorphisms).
  This gives ``if'', since equivalences compose.
  On the other hand, if $A$ is discrete, then the composite is in particular surjective, hence so is the second factor. Thus this second factor is an equivalence, and hence so is the first factor; this gives ``only if''.
\end{proof}

\subsection{Representable isomorphisms}

As a corollary of the Yoneda lemma, we can prove:

\begin{prop}\label{prop:rep-iso} If given a pair of terms $a , a' : A$ in a Segal type  and a fiberwise equivalence
\[ \phi:\prod_{x : A} \hom_A(a,x) \simeq \hom_A(a',x)\] then the corresponding term $\evid^{\hom_A(a',-)}_a(\phi) : \hom_A(a',a)$ is an isomorphism. If $A$ is a Rezk type, then $a' =_A a$.
\end{prop}
\begin{proof}
  Let $u \defeq \evid^{\hom_A(a',-)}_a(\phi)$.
  Then as observed in \cref{rmk:rep-nat-trans}, we have
  \[ \phi = \yon^{\hom_A(a',-)}_a(u) = \lam{x} \lam{f} f \circ u. \] 
  By the same argument, for the fiberwise inverse equivalence $\phi_x^{-1}:\hom_A(a',x) \to \hom_A(a,x)$ we have
  \[ \phi^{-1} = \yon^{\hom_A(a',-)}_a(v) = \lam{x} \lam{f} f \circ v. \] 
  where $v \defeq \evid^{\hom_A(a',-)}_a(\phi^{-1})$.
  Since the composite of these fiberwise equivalences is equal to the identity function, we then have in particular that $\idarr a = (\idarr a \circ u) \circ v = \idarr a \circ  (u \circ v) = u \circ v$ by the associativity and identity laws; similarly, $\idarr {a'} = v \circ u$. Thus $u$ is an isomorphism and if $A$ is a Rezk type then $u$ proves that $a' =_A a$.
\end{proof}

\section{Adjunctions}
\label{sec:adjunctions}

In this section we introduce several types of adjunction data between a pair of types $A$ and $B$ and then investigate comparisons between these adjunction notions in the case where $A$ and $B$ are Segal or Rezk types. This extends the similar inquiry concerning data defining an equivalence between types in \cite[Chapter 4]{hottbook}.

\subsection{Notions of adjunction}
\label{sec:notions-adjunction}

In ordinary category theory, there are two ways of defining an adjunction: by a natural isomorphism of hom-sets, or in terms of a unit and counit satisfying the triangle identities.
For clarity, in this section we will refer to the first style as a \emph{transposing adjunction} and the latter as a \emph{diagrammatic adjunction}.
Transposing adjunctions generalize to our synthetic context fairly easily.

\begin{defn}
  A \textbf{transposing adjunction} between types $A,B$ consists of functors $f:A\to B$ and $u:B\to A$ and a family of equivalences
  \[ \prod_{\substack{a:A\\ b:B}} \hom_B(fa,b) \simeq \hom_A(a,ub). \]
  Similarly, a \textbf{transposing left adjoint} of a functor $u:B\to A$ consists of a functor $f:A\to B$ together with such a family of equivalences. A \textbf{transposing right adjoint} of a functor $f \colon A \to B$ is defined dually.
\end{defn}

On the other hand, in any sort of higher category theory, the triangle identities for a diagrammatic adjunction become data that can be asked to satisfy higher coherence laws as in~\cite{RVadj}.
We will indicate the absence of such coherence with the prefix ``quasi-'', intentionally recalling the use of ``quasi-inverse'' in~\cite{hottbook} for an incoherent homotopy inverse.

\begin{defn}
  A \textbf{quasi-diagrammatic adjunction} between types $A,B$ consists of:
  \begin{itemize}
  \item a functor $u : B\to A$,
  \item a functor $f : A \to B$, 
  \item a natural transformation $\eta : \nat AA(\idfunc A, uf)$, and
  \item a natural transformation $\epsilon : \nat BB(fu, \idfunc  B)$ together with
  \item a witness $\alpha : \nattwo BA(u, ufu, u, \eta u, u \epsilon, \idarr u)$ and
  \item a witness $ \beta : \nattwo AB(f,fuf,f, f\eta, \epsilon f, \idarr f)$.
  \end{itemize}
  Similarly, a \textbf{quasi-diagrammatic left adjoint} of a functor $u:B\to A$ consists of the last five data above, and dually.
\end{defn}

Note that if $A$ is Segal, then by \cref{thm:comp-htpy}, the last two data may be presented equally as homotopies $\alpha : u\epsilon\circ\eta u = \idarr u$ and $\beta : \epsilon f \circ f\eta = \idarr f$.
We will frequently pass back and forth between these two points of view in \cref{sec:adj-segal,sec:adj-rezk}.
We have phrased the definition using higher simplices because it makes the connection to the theory of~\cite{RVadj} clearer; the incoherence of a quasi-diagrammatic adjunction, for instance, corresponds to the fact that Example 4.2.4 of~\cite{RVadj} is not a ``parental subcomputad''.

One of the main results of~\cite{RVadj} is that while a fully coherent adjunction requires infinitely much data, that data is determined up to a contractible space of choices by various finite subcollections; these are the \emph{parental subcomputads} of the generic adjunction $\underline{\smash[b]{\mathrm{Adj}}}$.
The formal framework of~\cite{RVadj} uses simplicially enriched categories to represent $(\infty,2)$-categories, with their hom-spaces regarded as presenting quasi-categories.
This can quite easily be translated into our setting, because our function-types, being types, have simplicial structure, and composition of functions in type theory is even strictly associative and unital.

As described in Examples 4.2.3 and 4.2.5 of~\cite{RVadj}, the four simplest parental subcomputads correspond in our framework to the following data:
\begin{enumerate}
\item A functor $f:A\to B$ only.
\item Functors $f:A\to B$ and $u:B\to A$ and a transformation $\epsilon : \nat BB(fu, \idfunc  B)$.
\item Functors $f$ and $u$, transformations $\eta$ and $\epsilon$, and the 2-simplex $\beta$.
\item Functors $f$ and $u$, transformations $\eta$ and $\epsilon$, 2-simplices $\alpha$ and $\beta$, and 3-simplices $\omega$ and $\tau$ with one new common 2-simplex face $\mu$.
\end{enumerate}

Of these, the last is the first one that includes at least the data of a quasi-diagrammatic adjunction, so that it suffices to determine some kind of adjunction without further hypotheses.
We name the corresponding structure in our setting by analogy to the ``half-adjoint equivalences'' of~\cite[\S4.2]{hottbook}.
The simplices $\mu,\omega,\tau$ are (like $\alpha$ and $\beta$) named as in~\cite[\S\S1.1 and 4.2]{RVadj}.

\begin{defn}\label{defn:hadadj}
  A \textbf{half-adjoint diagrammatic adjunction} between types consists of a quasi-diagrammatic adjunction together with:
  \begin{itemize}
  \item A 2-simplex $\mu : \nattwo BB (fu,fufu,\idfunc A,f\eta u,\epsilon\ast\epsilon,\epsilon)$, where $\epsilon\ast\epsilon$ is the horizontal composite from \cref{sec:horiz}.
  \item Two 3-simplices $\omega$ and $\tau$ with the boundaries shown in \cref{fig:hadadj}, where the 2-simplex denoted $\epsilon$ is degenerate, and the 2-simplices $\mathsf{nat}_\epsilon^1$ and $\mathsf{nat}_\epsilon^2$ are the two halves of the Gray interchanger $\extfn\epsilon(\epsilon)$ from \cref{sec:horiz}.
  \end{itemize}
\end{defn}

\begin{figure}
  \centering
  \[
  \begin{tikzcd}
    & fufu\ar[dr,"\epsilon\ast\epsilon"] \ar[dd,"fu\epsilon" description] &&&&
    fufu \ar[dr,"\epsilon\ast\epsilon"] \ar[d,phantom,"\scriptstyle \mu"]\\
    fu\ar[ur,"f\eta u"] \ar[dr,equals] \ar[r,phantom,"\scriptstyle f\alpha"] &
    ~ \ar[r,phantom,"\scriptstyle \mathsf{nat}_\epsilon^1"] & \idfunc B
    & \xRightarrow{\omega} &
    fu\ar[ur,"f\eta u"] \ar[dr,equals] \ar[rr,"\epsilon" description] & ~ \ar[d,phantom,"\scriptstyle \epsilon"] & \idfunc B\\
    & fu\ar[ur,"\epsilon"'] &&&& fu\ar[ur,"\epsilon"']
  \end{tikzcd}
  \]
  \[
  \begin{tikzcd}
    & fufu\ar[dr,"\epsilon\ast\epsilon"] \ar[dd,"\epsilon fu" description] &&&&
    fufu \ar[dr,"\epsilon\ast\epsilon"] \ar[d,phantom,"\scriptstyle \mu"]\\
    fu\ar[ur,"f\eta u"] \ar[dr,equals] \ar[r,phantom,"\scriptstyle \beta u"] &
    ~ \ar[r,phantom,"\scriptstyle \mathsf{nat}_\epsilon^2"] & \idfunc B
    & \xRightarrow{\tau} &
    fu\ar[ur,"f\eta u"] \ar[dr,equals] \ar[rr,"\epsilon" description] & ~ \ar[d,phantom,"\scriptstyle \epsilon"] & \idfunc B\\
    & fu\ar[ur,"\epsilon"'] &&&& fu\ar[ur,"\epsilon"']
  \end{tikzcd}
  \]
  \caption{The 3-simplices in a half-adjoint diagrammatic adjunction}
  \label{fig:hadadj}
\end{figure}

If $A$ is Segal, then by \cref{thm:32horn-is-concat} the 3-simplices $\omega$ and $\tau$ can equivalently be regarded as equalities relating the following two concatenated equalities to the homotopy $(\epsilon\ast\epsilon) \circ f\eta u = \epsilon$ corresponding to $\mu$:
\begin{gather}
  (\epsilon\ast\epsilon) \circ f\eta u
  = (\epsilon \circ fu\epsilon) \circ f\eta u
  = \epsilon \circ (fu\epsilon \circ f\eta u)
  \overset{f\alpha}= \epsilon \circ \idarr{fu}
  = \epsilon \label{eq:omega}\\
  (\epsilon\ast\epsilon) \circ f\eta u
  = (\epsilon \circ \epsilon fu)\circ f\eta u
  = \epsilon \circ (\epsilon fu\circ f\eta u)
  \overset{\beta u}= \epsilon \circ \idarr{fu}
  = \epsilon.\label{eq:tau}
\end{gather}
That is, in a Segal type, the type of half-adjoint diagrammatic adjunctions extending a given quasi-diagrammatic adjunction is equivalent to
\[ \sum_{\mu:(\epsilon\ast\epsilon) \circ f\eta u = \epsilon} (\eqref{eq:omega}=\mu) \times (\eqref{eq:tau}=\mu). \]
By contracting a based path space, this is equivalent to simply $(\eqref{eq:omega}=\eqref{eq:tau})$, i.e.\ the type of 2-homotopies filling the following diagram:
\[
\begin{tikzcd}[row sep=small]
  & (\epsilon \circ fu\epsilon) \circ f\eta u \ar[r,-] &
  \epsilon \circ (fu\epsilon \circ f\eta u) \ar[r,-,"f\alpha"] &
  \epsilon \circ \idarr{fu} \ar[dr,-]\\
  (\epsilon\ast\epsilon) \circ f\eta u \ar[ur,-] \ar[dr,-] &&&& \epsilon\\
  & (\epsilon \circ \epsilon fu)\circ f\eta u \ar[r,-] &
  \epsilon \circ (\epsilon fu\circ f\eta u) \ar[r,-,"\beta u"'] &
  \epsilon \circ \idarr{fu} \ar[ur,-]
\end{tikzcd}
\]
Since the concatenation $\epsilon \circ fu\epsilon = \epsilon\ast\epsilon = \epsilon \circ \epsilon fu$ consists of the naturality squares for $\epsilon$ at the components of $\epsilon$, if we disregard the associativity and unit coherences we can write this as
\[
\begin{tikzcd}     \epsilon \circ fu\epsilon \circ f\eta u \ar[rr,-,"{\mathsf{nat}_\epsilon}"]  \ar[dr,-,"f\alpha"']  & &   \epsilon \circ \epsilon fu \circ f\eta u. \ar[dl,-,"\beta u"] \\ & \epsilon &
\end{tikzcd}
\]

We will see in \cref{sec:adj-segal} that half-adjoint diagrammatic adjunctions correspond to half-adjoint equivalences in the sense of~\cite[\S4.2]{hottbook}, justifying the name.
On the other hand, the notion of ``bi-invertible map'' from~\cite[\S4.3]{hottbook} suggests the following modification instead:

\begin{defn}
  A \textbf{bi-diagrammatic adjunction} between types $A,B$ consists of:
  \begin{itemize}
  \item a functor $u : B\to A$,
  \item a functor $f : A \to B$, 
  \item a natural transformation $\eta : \nat AA(\idfunc A, uf)$, and
  \item two natural transformations $\epsilon,\epsilon' : \nat BB(fu, \idfunc B)$ together with
  \item a witness $\alpha : \nattwo BA(u, ufu, u, \eta u, u \epsilon, \idarr u)$ and
  \item a witness $ \beta : \nattwo AB(f,fuf,f, f\eta, \epsilon' f, \idarr f)$.
  \end{itemize}
  Note that $\epsilon$ appears in $\alpha$ but $\epsilon'$ appears in $\beta$.
\end{defn}

Our goal in the rest of this section is to compare all these kinds of adjunction, for Segal and Rezk types.

\subsection{Adjunctions between Segal types}
\label{sec:adj-segal}

We begin by observing that quasi-diagrammatic adjunctions suffice to induce transposing adjunctions.
More precisely, we show that a quasi-diagrammatic adjunction corresponds exactly to the following.

\begin{defn}
  A \textbf{quasi-transposing adjunction} between types $A,B$ consists of functors $f:A\to B$ and $u:B\to A$ and a family of maps
  \[ \phi : \prod_{\substack{a:A\\ b:B}} \hom_B(fa,b) \to \hom_A(a,ub) \]
  equipped with quasi-inverses, i.e.\ a family of maps
  \[ \psi : \prod_{\substack{a:A\\ b:B}} \hom_A(a,ub) \to \hom_B(fa,b) \]
  and homotopies $\xi:\prod_{a,b,k} \phi_{a,b}(\psi_{a,b}(k)) = k$ and $\zeta:\prod_{a,b,\ell} \psi_{a,b}(\phi_{a,b}(\ell)) = \ell$.
\end{defn}

\begin{thm}\label{thm:qadj}
  Given Segal types $A,B$ and functors $f:A\to B$ and $u:B\to A$, the type of quasi-transposing adjunctions between $f$ and $u$ is equivalent to the type of quasi-diagrammatic adjunctions between $f$ and $u$.
\end{thm}
\begin{proof}
  Each of the two types is a dependent sum type with four components.
  We will show that each of the four components is equivalent to a corresponding one on the other side.

  By \cref{prop:covar-rep,rmk:covariant-pullbacks}, for any $a:A$ the type family $\hom_A(a,u-) :B \to\univtype$ is covariant.
  Thus, by the Yoneda lemma, the map
  \[ \yon^{\hom_A(a,u-)}_{fa} : \hom_A(a,ufa) \to \prod_{b:B} (\hom_B(fa,b) \to \hom_A(a,ub)) \]
  is an equivalence, and hence so is the induced map
  \begin{equation}
    \nat AA(\idfunc A,uf) \simeq \Parens{\prod_{a:A} \hom_A(a,ufa)} \to \prod_{\substack{a:A\\b:B}} (\hom_B(fa,b) \to \hom_A(a,ub))\label{eq:qadj-eta}
  \end{equation}
  which sends $\eta$ to $\phi^\eta$ defined by $\phi^\eta_{a,b}(k) \defeq u k \circ \eta_a$.
  Similarly, the map
  \[ \nat BB(fu,\idfunc B) \simeq \Parens{\prod_{b:B} \hom_B(fub,b)} \to \prod_{\substack{a:A\\b:B}} (\hom_A(a,ub) \to \hom_B(fa,b)), \]
  sending $\epsilon$ to $\psi^\epsilon$ defined by $\psi^\epsilon_{a,b}(\ell)\defeq \epsilon_b \circ f\ell$, is an equivalence.

  It remains, therefore, to show that if we fix $\eta$ and $\epsilon$, then we have equivalences
  \[ \nattwo BA(u, ufu, u, \eta u, u \epsilon, \idarr u) \simeq \prod_{a,b,k} (\psi^\epsilon_{a,b}(\phi^\eta_{a,b}(k)) = k) \]
  and dually.
  First note that since $A$ is Segal,
  \begin{align*}
    \nattwo BA(u, ufu, u, \eta u, u \epsilon, \idarr u)
    &\simeq (u\epsilon \circ \eta u = \idarr u)\\
    &\simeq \prod_{b:B} (u \epsilon_b \circ \eta_{ub} = \idarr{ub}).
  \end{align*}
  Thus, it suffices to construct an equivalence
  \begin{equation}\label{eq:alpha-goal}
    (u \epsilon_b \circ \eta_{ub} = \idarr{ub}) \simeq \prod_{a:A} \prod_{k:\hom_A(a,ub)} (\phi^\epsilon_{a,b}(\psi^\eta_{a,b}(k)) = k)
  \end{equation}
  for any $b:B$.

  Now by \cref{thm:eq-covar,rmk:covariant-pullbacks}, the family $\lam{a}\lam{k} (\phi^\epsilon_{a,b}(\psi^\eta_{a,b}(k)) = k)$ is contravariant.
  Thus, by the contravariant form of the dependent Yoneda lemma (\cref{thm:dep-yoneda}), we have an equivalence
  \begin{equation}
    \Parens{\phi^\epsilon_{ub,b}(\psi^\eta_{ub,b}(\idarr{ub})) = \idarr{ub}}
    \simeq  \prod_{a:A} \prod_{k:\hom_A(a,ub)} (\psi^\epsilon_{a,b}(\phi^\eta_{a,b}(k)) = k).\label{eq:alpha-depyon}
  \end{equation}
  Moreover, we have
  \begin{align*}
    \phi^\epsilon_{ub,b}(\psi^\eta_{ub,b}(\idarr{ub}))
    &\jdeq \phi^\epsilon_{ub,b}(\epsilon_b \circ f\idarr{ub})\\
    &\jdeq u(\epsilon_b \circ f\idarr{ub}) \circ \eta_{ub}\\
    &\jdeq u(\epsilon_b \circ \idarr{fub}) \circ \eta_{ub}\\
    &= u\epsilon_b \circ \eta_{ub}.
  \end{align*}
  Concatenating with this identity yields an equivalence
  \[ \Parens{u \epsilon_b \circ \eta_{ub} = \idarr{ub}} \simeq \Parens{\phi^\epsilon_{ub,b}(\psi^\eta_{ub,b}(\idarr{ub})) = \idarr{ub}} \]
  which in combination with~\eqref{eq:alpha-depyon} gives~\eqref{eq:alpha-goal}.

  The second equivalence, involving $\beta$ and $\zeta$, is defined similarly, by combining a dependent Yoneda equivalence
  \begin{equation*}
    \Parens{\psi^\epsilon_{a,fa}(\phi^\eta_{a,fa}(\idarr{fa})) = \idarr{fa}}
    \simeq  \prod_{b:B} \prod_{\ell:\hom_B(fa,b)} (\phi^\epsilon_{a,b}(\psi^\eta_{a,b}(\ell)) = \ell).\label{eq:beta-depyon}
  \end{equation*}
  and concatenation with the equality
  \begin{align*}
    \psi^\epsilon_{a,fa}(\phi^\eta_{a,fa}(\idarr{fa}))
    &\jdeq \epsilon_{fa} \circ f(u\idarr{fa} \circ \eta_a)\\
    &= \epsilon_{fa} \circ f\eta_a.\qedhere
  \end{align*}
\end{proof}

\begin{cor}\label{thm:qadj-transp}
  Any quasi-diagrammatic adjunction between Segal types induces a transposing adjunction.
\end{cor}
\begin{proof}
  This follows directly from \cref{thm:qadj} and the fact that any quasi-inverse can be improved to a coherent equivalence,~\cite[Theorem 4.2.3]{hottbook}.
\end{proof}

Our next goal is to show that the two improved versions of diagrammatic adjunctions really are suitably ``coherent''.
One of them is very straightforward.

\begin{thm}\label{thm:bda}
  Given Segal types $A,B$ and functors $f:A\to B$ and $u:B\to A$, the type of bi-diagrammatic adjunctions between them is equivalent to the type of transposing adjunctions.
\end{thm}
\begin{proof}
  In the definition of transposing adjunction, as is usual in homotopy type theory, we did not specify exactly which coherent notion of ``equivalence'' was meant, since given function extensionality all of them are equivalent.
  For the purposes of this theorem, we take it to mean bi-invertible maps.
  We can then unwind the definition of transposing adjunction to consist of a family of maps
  \[ \phi : \prod_{\substack{a:A\\ b:B}} \hom_B(fa,b) \to \hom_A(a,ub) \]
  and \emph{two} families of maps
  \[ \psi,\psi' : \prod_{\substack{a:A\\ b:B}} \hom_A(a,ub) \to \hom_B(fa,b) \]
  together with homotopies $\prod_{a,b,k} \phi_{a,b}(\psi_{a,b}(k)) = k$ and $\prod_{a,b,\ell} \psi'_{a,b}(\phi_{a,b}(\ell)) = \ell$.
  The same arguments as in \cref{thm:qadj} then identify these data with those of a bi-diagrammatic adjunction.
\end{proof}

The other requires a bit more work.

\begin{thm}\label{thm:haa}
  Given Segal types $A,B$ and functors $f:A\to B$ and $u:B\to A$, the type of half-adjoint diagrammatic adjunctions between them is equivalent to the type of transposing adjunctions.
\end{thm}
\begin{proof}
Just as in \cref{thm:bda} we formulated transposing adjunctions using bi-invertible maps, here we formulate them using half-adjoint equivalences.
  Thus, it remains to show that the type of coherence data $\mu$ in a half-adjoint diagrammatic adjunction is equivalent, over the equivalences constructed in \cref{thm:qadj}, to the type of families of half-adjoint coherence data for a quasi-transposing adjunction $(\phi,\psi,\xi,\zeta)$:
  \begin{equation}
    \tprod_{b:B} \tprod_{a:A} \tprod_{\ell:\hom_A(a,ub)}
    \Parens{\zeta^\beta_{a,b}(\psi^{\epsilon}_{a,b}(\ell)) = \ap_{\psi^\epsilon_{a,b}}(\xi^{\alpha}_{a,b}(\ell))}\label{eq:tradj-mu}
  \end{equation}
  Applying \cref{thm:eq-covar,rmk:covariant-pullbacks} twice, we see that
  \[\lam{a}\lam{\ell} \Parens{\zeta^\beta_{a,b}(\psi^\epsilon_{a,b}(\ell)) = \ap_{\psi^\epsilon_{a,b}}(\xi^\alpha_{a,b}(\ell))}\]
  is contravariant.
  Thus, by the dependent Yoneda lemma,~\eqref{eq:tradj-mu} is equivalent to
  \begin{equation}
    \tprod_{b:B} \Parens{\zeta^\beta_{ub,b}(\psi^\epsilon_{ub,b}(\idarr{ub}))
      = \ap_{\psi^\epsilon_{ub,b}}(\xi^\alpha_{ub,b}(\idarr{ub}))}.\label{eq:tradj-mu2}
  \end{equation}

  Now we need to analyze $\xi$ and $\zeta$ more carefully in terms of $\alpha$ and $\beta$.
  By definition, $\lam{a}\xi^\alpha_{a,b}$ is the image of the concatenated equality
  \begin{equation}
  \phi^\epsilon_{ub,b}(\psi^\eta_{ub,b}(\idarr{ub})) \jdeq 
  u(\epsilon_b \circ f\idarr{ub}) \circ \eta_{ub}
  = u\epsilon_b \circ \eta_{ub} \overset\alpha= \idarr{ub}\label{eq:pre-xi}
  \end{equation}
  under the inverse dependent Yoneda map
  \[ \yon_{ub} :
  \Parens{\phi^\epsilon_{ub,b}(\psi^\eta_{ub,b}(\idarr{ub})) = \idarr{ub}}
  \to  \prod_{a:A} \prod_{k:\hom_A(a,ub)} (\psi^\epsilon_{a,b}(\phi^\eta_{a,b}(k)) = k).
  \]
  which implies that $\evid_{ub}(\lam{a}\xi^\alpha_{a,b})$, i.e.\ $\xi^\alpha_{ub,b}(\idarr{ub})$, is equal to~\eqref{eq:pre-xi}.
  Now $\psi^\epsilon_{a,b}(k) \jdeq \epsilon_b \circ fk$, so $\ap_{\psi^\epsilon_{ub,b}} = \ap_{(\epsilon_b\circ -)} \circ \ap_{\extfn f}$.
  Thus the right-hand side of~\eqref{eq:tradj-mu2} is equal to the concatenation
  \begin{align}
    \psi^\epsilon_{ub,b}(\phi^\epsilon_{ub,b}(\psi^\eta_{ub,b}(\idarr{ub})))
    &\jdeq \epsilon_b \circ f(u(\epsilon_b \circ \idarr{fub}) \circ \eta_{ub}) \notag\\
    &= \epsilon_b \circ f(u\epsilon_b \circ \eta_{ub}) \notag\\
    &\overset{\mathclap{\epsilon_b \circ f\alpha}}= \epsilon_b \circ f\idarr{ub} \label{eq:ep-f-alpha}\\
    &\jdeq \epsilon_b \circ \idarr{fub}\notag
  \end{align}
  in which the two non-judgmental equalities are obtained by $\ap_{\extfn f}$ followed by $\ap_{(\epsilon_b\circ -)}$ from those in~\eqref{eq:pre-xi}.
  There are other ways to define an equality meriting the name ``${\epsilon_b \circ f\alpha}$'', but they can all easily be shown to be equal.

  Similarly, $\zeta^\beta_{a,b}$ is the image of the concatenated equality
  \begin{equation}\label{eq:pre-zeta}
    \psi^\epsilon_{a,fa}(\phi^\eta_{a,fa}(\idarr{fa}))
    \jdeq \epsilon_{fa} \circ f(u\idarr{fa} \circ \eta_a)
    = \epsilon_{fa} \circ f\eta_a
    \overset{\beta}= \idarr{fa}
  \end{equation}
  under the inverse dependent Yoneda map
  \begin{equation*}
    \yon_{fa} : \Parens{\psi^\epsilon_{a,fa}(\phi^\eta_{a,fa}(\idarr{fa})) = \idarr{fa}}
    \to \prod_{b:B} \prod_{\ell:\hom_B(fa,b)} (\phi^\epsilon_{a,b}(\psi^\eta_{a,b}(\ell)) = \ell).
  \end{equation*}
  As we saw in \cref{sec:yoneda-lemma}, the latter can be defined by $\yon_{fa}(e,b,\ell) \defeq \covtr{(\connmin \ell)}(e)$.

  Now, for any $b:B$ and $\ell:\hom_B(fa,b)$, define
  \begin{align*}
    C(a,b,\ell)
    &\defeq\Parens{\phi^\epsilon_{a,b}(\psi^\eta_{a,b}(\ell)) = \ell}\\ 
    &\jdeq
    \Parens{\epsilon_{b} \circ f(u\ell\circ \eta_a) = \ell}
    \notag\\
    D(a,b,\ell)
    &\defeq
    \sum_{m:\hom_A(a,ub)}
    \homtwo{A} (a,ufa,ub,\eta_a,u\ell,m)\times
    \homtwo{B}(fa,fub,b,{fm},\epsilon_{b},\ell).
  \end{align*}
  Then we have $C(a,b,\ell)\simeq D(a,b,\ell)$, by contracting the first two components of $D(a,b,\ell)$, which are contractible since $A$ is Segal, and then applying \cref{thm:comp-htpy}.
  Specifically, the map $D(a,b,\ell)\to C(a,b,\ell)$ takes $(m,\gamma,\delta)$ to the concatenated equality
  \[\epsilon_{b} \circ f(u\ell\circ \eta_a) \overset {\epsilon_b \circ f\gamma}= \epsilon_b \circ fm \overset\delta= \ell.\]
  (This can be proven easily by assuming that $\gamma$ comes from an equality and doing path induction on it.)
  In particular, since $C$ is covariant by \cref{thm:eq-covar,rmk:covariant-pullbacks}, so is $D$.

  We are interested in two cases of this equivalence:
  \begin{itemize}
  \item For any $a:A$, we have $(\eta_a,s_0\eta_a,\beta_a):D(a,fa,\idarr{fa})$.
    The corresponding element of $C(a,fa,\idarr{fa})$ is~\eqref{eq:pre-zeta}.
    Thus, the element of $D(a,b,\ell)$ corresponding to $\yon_{fa}(e,b,\ell) \defeq \covtr{(\connmin \ell)}(e)$, when $e$ is~\eqref{eq:pre-zeta}, is the covariant transport $\covtr{(\connmin \ell)}(\eta_a,s_0\eta_a,\beta_a)$ in the type family $D(a,-,-)$.
  \item For any $b:B$, we have $(\idarr{ub},\alpha_b,s_1 \epsilon_b):D(ub,b,\epsilon_b)$.
    Transporting this along the equality $p : \epsilon_b = \epsilon_b \circ \idarr{fub} \jdeq \psi^\eta_{ub,b} (\idarr{ub})$ to obtain an element of $D(ub,b,\psi^\eta_{ub,b} (\idarr{ub}))$, and passing back into $C(ub,b,\psi^\eta_{ub,b} (\idarr{ub}))$, we obtain our computation~\eqref{eq:ep-f-alpha} of $\ap_{\psi^\epsilon_{ub,b}}(\xi^\alpha_{ub,b}(\idarr{ub}))$.
  \end{itemize}
  Our goal, therefore, is to identify the equality type 
  \[\covtr{(\connmin {\epsilon_b\circ \idarr{fub}})}(\eta_{ub},s_0\eta_{ub},\beta_{ub})
  = \transport{}{p}{(\idarr{ub},\alpha_b,s_1 \epsilon_b)}.
  \]
  or equivalently by \cref{thm:idtoiso-trans} the simpler
  \[\covtr{(\connmin {\epsilon_b})}(\eta_{ub},s_0\eta_{ub},\beta_{ub})
  = (\idarr{ub},\alpha_b,s_1 \epsilon_b).
  \]
  Now by \cref{thm:covtr-is-eq}, this is equivalent to
  \begin{equation}
    \hom_{D(ub)(\epsilon_b,\connmin{\epsilon_b})}((\eta_{ub},s_0\eta_{ub},\beta_{ub}), (\idarr{ub},\alpha_b,s_1 \epsilon_b)).\label{eq:two-prisms}
  \end{equation}
  The notation $D(ub)(\epsilon_b,\connmin{\epsilon_b})$ means we consider dependent arrows in the type family $D(ub,-,-)$ over the arrow
  \[ (\epsilon_b,\connmin{\epsilon_b}):\hom_{\sum_{b':B} \hom_B(fub,b')}((fub,\idarr{fub}),(b,\epsilon_b)).\]
  Compiling this out, and using the equivalences of \cref{sec:equiv-exten}, we see that an element of~\eqref{eq:two-prisms} consists of two ``triangular prisms'' $\Delta^2\times\Delta^1 \to A$ and $\Delta^2\times\Delta^1 \to B$, with some elements of their boundary fixed as shown in \cref{fig:two-prisms}.

  \begin{figure}
    \centering
    \[
  \begin{tikzcd}
    & ufub \ar[dr,"u\epsilon_b"] & & &     & ufub \ar[dr,"u\epsilon_b"] \ar[d,phantom,"\scriptstyle \alpha_b"]\\
    ub \ar[ur,"\eta_{ub}"] && ub & &     ub \ar[ur,"\eta_{ub}"] \ar[rr,equals] & ~ & ub\\
    & ufub \ar[dr,equals] \ar[uu,equals] \ar[ur,phantom,"\scriptstyle{u\connmin{\epsilon_b}}"] \ar[d,phantom,"\scriptstyle \eta_{ub}"]   & &  \Rightarrow  & & \\
    ub \ar[ur,"\eta_{ub}" description] \ar[rr,"\eta_{ub}"'] \ar[uu,equals] \ar[ruuu,phantom,"\scriptstyle\eta_{ub}"] &~& ufub\ar[uu,"u\epsilon_b"'] & & ub \ar[rr,"\eta_{ub}"'] \ar[uu,equals] \ar[rruu,phantom,"\scriptstyle \varrho"] && ufub\ar[uu,"u\epsilon_b"']\\    & fub \ar[dr,"\epsilon_b"] & & &     & fub \ar[dr,"\epsilon_b"] \ar[d,phantom,"\scriptstyle \epsilon_b"]\\
    fub \ar[ur,equals] && b & &     fub \ar[ur,equals] \ar[rr,"\epsilon_b" description] & ~ & b\\
    & fufub \ar[dr,"\epsilon_{fub}" description] \ar[uu,"fu\epsilon_b" description] \ar[ur,phantom,"\scriptstyle \mathsf{nat}_\epsilon" description] \ar[d,phantom,"\scriptstyle \beta_{ub}"] & & \Rightarrow & &\\
    fub \ar[ur,"f\eta_{ub}" description] \ar[rr,equals] \ar[uu,equals] \ar[ruuu,phantom,"\scriptstyle f\varrho"] &~& fub\ar[uu,"\epsilon_b"'] & & 
    fub \ar[rr,equals] \ar[uu,equals] \ar[rruu,phantom,"\scriptstyle \connmin{\epsilon_b}"] && fub\ar[uu,"\epsilon_b"']
  \end{tikzcd}
  \]
    \caption{Two prisms}
    \label{fig:two-prisms}
  \end{figure}
  The square $\varrho$ is not fixed, but must be the same in both prisms; and none of the interior simplices (not shown) are fixed.
  Squares or 2-simplices marked with the name of an arrow are constant/degenerate in the other direction, and the square denoted $\mathsf{nat}_\epsilon$ is the naturality square for $\epsilon$ at itself, as constructed in \cref{sec:naturality}.

  Now, the boundary data of the top prism (in $A$) that is fixed consists of a ``trough''
  \[(\Lambda^2_1 \times \Delta^1) \cup_{(\Lambda^2_1 \times \partial\Delta^1)} (\Delta^2 \times \partial\Delta^1) \to A. \]
  The inclusion of the trough into the prism $\Delta^2\times \Delta^1$ is the pushout product of $\Lambda^2_1 \to\Delta^2$ and $\partial\Delta^1 \to \Delta^1$.
  Thus, by \cref{thm:pop-anodyne}, the type of fillers (consisting of $\varrho$ and all the inner simplices in the top prism) is contractible.
  Thus, it does not affect the homotopy type of~\eqref{eq:two-prisms}, so in identifying the latter we are free to fix any particular such filler.
  We choose the following one:
  \[\begin{array}{ccl}
      \Delta^2 \times \Delta^1 &\to& A\\
      \sh{\pair{t_1,t_1}:\two\times\two}{t_2\le t_1} \times \sh{t_3:\two}{\top} &\to& A\\
      \pair{\pair{t_1,t_2},t_3} &\mapsto&
                                          \begin{cases}
                                            \alpha_{b}(t_2,t_1) &\quad t_2\le t_3\\
                                            \alpha_{b}(t_3,t_1) &\quad t_3\le t_2
                                          \end{cases}
    \end{array}
  \]
  It is straightforward to verify that this has the correct boundary.
  It determines $\varrho$ to be the following square:
  \[
  \begin{tikzcd}
    ub \ar[r,equals] & ub \\
    ub \ar[u,equals] \ar[ur,equals] \ar[r,"\eta_{ub}"'] & ufub \ar[u,"u\epsilon_b"'] \ar[ul,phantom,"\scriptstyle\alpha_b" very near start]
  \end{tikzcd}
  \]

  Now the second prism has its entire boundary fixed.
  As noted in \cref{sec:sub-simplices}, a prism consists of three 3-simplices glued along two common boundary 2-simplices.
  When the boundary of the prism is fixed, the ``upper'' of these 3-simplices has a 3-1-horn on its boundary fixed, the ``lower'' one has a 3-2-horn on its boundary fixed, while the ``middle'' one has only two faces of its boundary fixed.
  By \cref{thm:3horn-anodyne}, the types of 3-simplex fillers for 3-1-horns and 3-2-horns are contractible, so in determining the homotopy type of prisms we may assume a particular filler for the upper and/or lower horns.

  \begin{figure}
    \centering
\[
  \begin{tikzcd}[ampersand replacement=\&]
    \& fub \ar[dr,"\epsilon_b"] \& \& \&    \& fub \ar[dr,"\epsilon_b"] \ar[d,phantom,"\scriptstyle \epsilon_b"]\\
    fub \ar[ur,equals] \&\& b\& \&     fub \ar[ur,equals] \ar[rr,"\epsilon_b" description] \& ~ \& b\\
    \& \& \& \Rightarrow \&    \& \ar[ul,phantom,"\scriptstyle \epsilon_b"]\\
    fub \ar[uu,equals] \ar[ruuu,equals] \ar[rruu,"\epsilon_b"']\& \& \& \&    fub \ar[uu,equals] \ar[rruu,"\epsilon_b"']
\\
    \& fub \ar[dr,"\epsilon_b"]\& \& \&   \& fub \ar[dr,"\epsilon_b"] \ar[dd,phantom,"\scriptstyle \epsilon_b" near start] \& ~\\
    ~\& ~ \ar[r,phantom,"\scriptstyle \mathsf{nat}^1_{\epsilon}"] \& b \& \&      \& ~ \& b\\
    \& fufub \ar[ul,phantom,"\scriptstyle \alpha_b" near start] \ar[uu,"fu\epsilon_b" description] \ar[ur,"\epsilon\ast\epsilon"'] \& \& \xRightarrow{\omega} \&    \& fufub \ar[ur,"\epsilon\ast\epsilon"'] \ar[uur,phantom,"\scriptstyle \mu_b" pos=.2] \\
    fub \ar[ur,"f\eta_{ub}"']  \ar[ruuu,equals,bend left] \&~\& \& \&  fub \ar[uuur,equals,bend left] \ar[rruu,bend left,"\epsilon_b" description] \ar[ur,"f\eta_{ub}"']\\
    \&~\& b \& \&     \& \& b\\
    \& fufub \ar[dr,"\epsilon_{fub}" description] \ar[ur,"\epsilon\ast\epsilon" description] \ar[u,phantom,"\scriptstyle \mu_b" pos=.3]
    \ar[d,phantom,"\scriptstyle \beta_{ub}"]  \ar[r,phantom,"\scriptstyle \mathsf{nat}^2_{\epsilon}"] \& ~ \& \xRightarrow{\tau} \&    \& \ar[dr,phantom,"\scriptstyle \epsilon_b"]\& \\
    fub \ar[ur,"f\eta_{ub}" description] \ar[rr,equals] \ar[rruu,bend left,"\epsilon_b"] \&~\& fub\ar[uu,"\epsilon_b"']
\& \& 
    fub \ar[rr,equals] \ar[rruu,bend left, "\epsilon_b"] \&\& fub \ar[uu,"\epsilon_b"']
  \end{tikzcd}
\]
\caption{The three 3-simplices in a prism}
    \label{fig:three-3simplices}
  \end{figure}
  In our case, the upper 3-1-horn has an obvious filler given by a doubly degenerate 3-simplex on $\epsilon_b$.
  If we fill this in, the remaining two 3-simplices and their common boundary 2-simplex are exactly the data of $\omega$, $\tau$, and $\mu$ from \cref{defn:hadadj} (evaluated at $b:B$).
  \cref{fig:three-3simplices} shows all three 3-simplices roughly as they sit inside the prism.
  Thus, the type of such prisms is equivalent to the type of half-adjoint diagrammatic adjunctions.
\end{proof}

We can therefore conclude:

\begin{cor}\label{thm:segal-adj-prop}
  Given Segal types $A,B$ and functors $f:A\to B$ and $u:B\to A$ along with a natural transformation $\eta:\nat AA (\idfunc A,uf)$, the following types are equivalent propositions.
  \begin{enumerate}[label=(\roman*)]
  \item The type of witnesses that $\lam{k} u k \circ \eta_a : \hom_B(fa,b) \to \hom_A(a,ub) $ is an equivalence for all $a,b$.\label{item:sap1}
  \item The type of $(\epsilon,\epsilon',\alpha,\beta)$ extending $(f,u,\eta)$ to a bi-diagrammatic adjunction.\label{item:sap2}
  \item The type of $(\epsilon,\alpha,\beta,\mu,\omega,\tau)$ extending $(f,u,\eta)$ to a half-adjoint diagrammatic adjunction.\label{item:sap3}
  \item The propositional truncation of the type of $(\epsilon,\alpha,\beta)$ extending $(f,u,\eta)$ to a quasi-diagrammatic adjunction.\label{item:sap4}
  \end{enumerate}
\end{cor}
\begin{proof}
  We have seen that when~\ref{item:sap1} is expressed using bi-invertibility it is equivalent to~\ref{item:sap2}, and that when it is expressed using half-adjoint equivalences it is equivalent to~\ref{item:sap3}.
  But~\ref{item:sap1} is always a proposition, however expressed.
  Finally,~\ref{item:sap4} is a proposition by definition, which implies~\ref{item:sap2} and is implied by~\ref{item:sap3}.
\end{proof}

In other words, if a given transformation $\eta$ is the unit of an adjunction, then that adjunction is uniquely determined up to a contractible space of choices.
This corresponds to the dual of the fact mentioned in \cref{sec:notions-adjunction} that $(f,u,\epsilon)$ is a parental subcomputad of $\underline{\smash[b]{\mathrm{Adj}}}$.
Similarly, the fact that $(f,u,\eta,\epsilon,\beta)$ is a parental subcomputad corresponds to the dual of the following:

\begin{cor}\label{thm:alpha-parental}
  Given data $(f,u,\eta,\epsilon,\alpha)$ as in a quasi-diagrammatic adjunction, the following types are equivalent propositions:
  \begin{enumerate}[label=(\roman*)]
  \item The type of $(\epsilon',\beta)$ extending $(f,u,\eta,\epsilon,\alpha)$ to a bi-diagrammatic adjunction.\label{item:ap1}
  \item The type of $(\beta,\mu,\omega,\tau)$ extending $(f,u,\eta,\epsilon,\alpha)$ to a half-adjoint diagrammatic adjunction.\label{item:ap2}
  \item The propositional truncation of the type of $\beta$ extending $(f,u,\eta,\epsilon,\alpha)$ to a quasi-diagrammatic adjunction.\label{item:ap3}
  \end{enumerate}
\end{cor}
\begin{proof}
  Since there is a map from~\ref{item:ap2} to~\ref{item:ap1} (take $\epsilon'\defeq\epsilon$) that becomes an equivalence when summed over $\epsilon$ and $\alpha$, it is already an equivalence.
  Moreover, the proof of \cref{thm:bda} actually shows that given $(f,u,\eta)$, the types of $(\epsilon,\alpha)$ and $(\epsilon',\beta)$ are equivalent to the types $\mathsf{linv}(\phi^\eta)$ and $\mathsf{rinv}(\phi^\eta)$ of left and right inverses to $\phi^\eta$ respectively (see~\cite[Definition 4.2.7]{hottbook}).
  Since these types are both contractible as soon as they are both inhabited~\cite[Lemma 4.2.9]{hottbook}, it follows that~\ref{item:ap1} is a proposition, hence so is~\ref{item:ap2}.
  Finally,~\ref{item:ap3} is a proposition by definition, which implies~\ref{item:ap1} and is implied by~\ref{item:ap2}.
\end{proof}

We would also like to know that if a given functor $u$ has a left adjoint, then the entire adjunction is likewise uniquely determined, corresponding to the dual of the fact that $(f)$ itself is already a parental subcomputad.
However, since uniqueness of a functor $f:A\to B$ involves equalities in $B$, for this we need to assume our types are not just Segal but Rezk.

\subsection{Adjunctions between Rezk types}
\label{sec:adj-rezk}

Under the additional hypothesis that the domain of a functor is Rezk and not just Segal, we can prove:

\begin{thm}\label{thm:rezk-adj-prop}
  Given a Segal type $A$ and a Rezk type $B$, and a functor $u:B\to A$, the following types are equivalent propositions.
  \begin{enumerate}[label=(\roman*)]
  \item The type of transposing left adjoints of $u$.\label{item:rap1}
  \item The type of functors $f:A\to B$ and transformations $\eta:\nat AA (\idfunc A,uf)$ such that $\lam{k} u k \circ \eta_a : \hom_B(fa,b) \to \hom_A(a,ub) $ is an equivalence for all $a,b$.\label{item:rap1a}
  \item The type of half-adjoint diagrammatic left adjoints of $u$.\label{item:rap2}
  \item The type of bi-diagrammatic left adjoints of $u$.\label{item:rap3}
  \item The propositional truncation of the type of quasi-diagrammatic left adjoints of $u$.\label{item:rap4}
  \end{enumerate}
\end{thm}
\begin{proof}
  The equivalence between~\ref{item:rap1} and~\ref{item:rap1a} follows by passing across the single Yoneda equivalence~\eqref{eq:qadj-eta}, while \cref{thm:segal-adj-prop} implies that~\ref{item:rap1a}, \ref{item:rap2}, and~\ref{item:rap3} are equivalent.
  And~\ref{item:rap4} is a proposition that is implied by~\ref{item:rap2}, and will imply~\ref{item:rap3} as soon the latter is a proposition.
  
  Thus, it suffices to show that~\ref{item:rap1a} is a proposition.
  But this is equivalent to
  \[ \tsum_{f:A\to B} \tsum_{\eta:\prod_{a:A} \hom_A(a,ufa)} \tprod_{a:A} \tprod_{b:B} \isequiv{\lam{k} u k \circ \eta_a}\]
  and this is equivalent to
  \[ \tprod_{a:A} \tsum_{f_a:B} \tsum_{\eta_a:\hom_A(a,uf_a)} \tprod_{b:B} \isequiv{\lam{k} u k \circ \eta_a}.\]
  Thus, since a product of propositions is a proposition, it suffices to prove that 
  \[ \tsum_{f_a:B} \tsum_{\eta_a:\hom_A(a,uf_a)} \tprod_{b:B} \isequiv{\lam{k} u k \circ \eta_a}\]
  is a proposition for all $a:A$.

  Note that this is the type of ``universal arrows'' from $a$ to the functor $u$; thus we are now reduced to essentially the usual proof of uniqueness of such universal arrows.
  Let $(f_a,\eta_a,\omega)$ and $(f_a',\eta_a',\omega')$ be two elements of this type.
  Since $\omega$ and $\omega'$ belong to propositions, we can ignore them for purposes of proving equality; what they give is us that the maps
  \begin{align*}
    \lam{k} u k \circ \eta_a &: \hom_B(f_a,b) \to \hom_A(a,ub)\\
    \lam{k} u k \circ \eta_a' &: \hom_B(f_a',b) \to \hom_A(a,ub)
  \end{align*}
  are equivalences for any $b:B$.
  Taking $b\defeq f_a'$ in the first equivalence, and applying its inverse to $\eta_a'$, we obtain $m:\hom_B(f_a,f_a')$ such that $um\circ\eta_a = \eta_a'$.
  Then taking $b\defeq f_a$ in the second equivalence, and applying its inverse to $\eta_a$, we obtain $n:\hom_B(f_a',f_a)$ such that $un\circ\eta_a'= \eta_a$.
  Thus, $u(m\circ n) \circ\eta_a' = \eta_a'$ and $u(n\circ m) \circ \eta_a = \eta_a$, so by the injectivity of equivalences, $m$ and $n$ are inverse isomorphisms in $B$.

  Now since $B$ is Rezk, we have $e:f_a=f_a'$ such that $\idtoiso(e) = m$.
  By the characterization of equalities in $\Sigma$-types, it suffices to show that
  \[\transport{\lam{b} \hom_A(a,ub)}{e}{\eta_a} = \eta_a'.\]
  But using \cref{thm:idtoiso-trans,thm:idtoiso-ap}, we have
  \begin{align*}
    \transport{\lam{b} \hom_A(a,ub)}{e}{\eta_a}
    &= \transport{\lam{x} \hom_A(a,x)}{\ap_u(e)}{\eta_a}\\
    &= \idtoiso(\ap_u(e)) \circ \eta_a\\
    &= u(\idtoiso(e)) \circ \eta_a\\
    &= um \circ \eta_a\\
    &= \eta_a'.\qedhere
  \end{align*}
\end{proof}

In other words, for Rezk types (regarded as synthetic $(\infty,1)$-categories), adjoints are literally unique, not just ``unique up to isomorphism''.
This should be compared with~\cite[Lemma 9.3.2]{hottbook}, which proves an analogous fact for 1-categories \emph{defined} internally to ordinary homotopy type theory (rather than axiomatized synthetically) and satisfying a similar Rezk-completeness condition.

\appendix

\section{Semantics of simplicial type theory}
\label{sec:semantics}

In this section we review the model of homotopy type theory in the category of Reedy fibrant bisimplicial sets from \cite{elreedy} and describe how this category also models the simplicial type theory of cubes, topes, and shapes.
We will not give a complete proof, but only sketch the main ideas.
We then prove that Segal types correspond exactly to the \emph{Segal spaces} in this model, while Rezk types correspond to the complete Segal spaces~\cite{css}, which are also called \emph{Rezk spaces}.

\subsection{Reedy fibrations of bisimplicial sets}

The category $\sSet \defeq\Set^{\DDelta^\op}$ of simplicial sets embeds in two ``orthogonal'' ways into the category $\ssSet \defeq \Set^{\DDelta^\op \times \DDelta^\op}$ of bisimplicial sets. Via the isomorphism $\ssSet \cong \sSet^{\DDelta^\op}$ that expresses a bisimplicial set $X$ as a simplicial space, we regard $X_{m,n}$ as the set of $n$-simplices in the $m$th space of the simplicial object $X \colon \DDelta^\op \to \sSet$. 

To define these two embeddings use the \textbf{external product bifunctor}
\[ \sSet \times \sSet \xrightarrow{\square} \ssSet \qquad \qquad (A \square B)_{m,n} \defeq A_m \times B_n.\]
Note that $\Delta^m \square \Delta^n$ is the functor represented by the object $(m,n) \in \DDelta \times \DDelta$. In particular, using exponential notation for the internal hom in $\ssSet$, we have \[ (Y^X)_{m,n} = \ssSet(X \times (\Delta^m \square \Delta^n), Y).\]

\begin{defn}[the discrete and constant embeddings]
Fixing one variable to be the point, we obtain embeddings
\[ \mathrm{disc} \colon \sSet \xrightarrow{-\square \Delta^0} \ssSet \qquad \qquad \mathrm{const} \colon \sSet \xrightarrow{\Delta^0\square-} \ssSet\] of simplicial sets as \textbf{discrete} and \textbf{constant} bisimplicial sets, respectively. The discrete simplicial spaces have the form of functors $\DDelta^\op \to \Set \hookrightarrow \sSet$, while the constant simplicial spaces have the form of functors $\DDelta^\op \to \mathbbe{1} \to \sSet$. The discrete embedding positions the data of a simplicial set in the ``categorical'' direction, while the constant embedding positions the data in the ``spacial'' direction.
\end{defn}

The bifunctor $-\square -$ is biclosed. Under the identification $\ssSet \cong \sSet^{\DDelta^\op}$ described above, the left closure
\[ A \square B \to X \qquad \leftrightsquigarrow \qquad B \to \{A, X\}\] is the limit of $X$ weighted by the simplicial set $A \in \Set^{\DDelta^\op}$; in particular, $\{\Delta^m,X\} \cong X_m$, the $m$th column of $X$.

\begin{defn} A bisimplicial set $X \to Y$ is a \textbf{Reedy fibration} if and only if for all $m \geq 0$ the induced map
\[ \{ \Delta^m,X\} \to \{\partial\Delta^m,X\} \times_{\{\partial\Delta^m,Y\}} \{\Delta^m , Y\}\] on weighted limits is a Kan fibration in $\sSet$. A bisimplicial set $X$ is \textbf{Reedy fibrant} just when the unique map $X \to 1$ is a Reedy fibration, which is the case when
\[ \{ \Delta^m, X\} \to \{\partial\Delta^m,X\}\] is a Kan fibration.
\end{defn}

Any bifunctor, such as $\square$, whose codomain has pushouts has an associated pushout product; in our case this defines a biclosed bifunctor
\[  \sSet^\two \times \sSet^\two \xrightarrow{\widehat{\square}} \ssSet^\two.\] The set of maps 
\[ \{ (\partial\Delta^m\hookrightarrow\Delta^m)\mathbin{\widehat{\square}}(\partial\Delta^n\hookrightarrow\Delta^n)\}_{m,n \geq 0}\] defines a set of generating Reedy cofibrations for $\ssSet$. A map of bisimplicial sets is a \textbf{Reedy trivial fibration} if and only if it has the right lifting property with respect to this set of maps.

\begin{thm}[{Shulman \cite{elreedy}}] The Reedy model structure on bisimplicial sets defined relative to the Quillen model structure on simplicial sets models intensional type theory with dependent sums, dependent products, identity types, and as many univalent universes as there are inaccessible cardinals greater than $\aleph_0$.
\end{thm}

In the bisimplicial sets model, a dependent type family $C \colon A \to \univtype$ is modeled by a Reedy fibration $C \twoheadrightarrow A$, which we denote using an arrow ``$\twoheadrightarrow$'' for emphasis. 

The Reedy fibrations enjoy the following important ``Leibniz closure'' property.

\begin{lem}\label{rmk:leibniz-reedy} If $i\colon U \rightarrow V$ is a monomorphism (equivalently, a cofibration) of bisimplicial sets and $p \colon X \twoheadrightarrow Y$ is a Reedy fibration then the map
  \[ \pair{X^i, p^V}\colon X^V \to X^U \times_{Y^U} Y^V,\] which we denote by $\widehat{\{i,p\}}$, is a Reedy fibration, whose domain and codomain are Reedy fibrant if $X$ and $Y$ are, and which is a weak equivalence if $p$ is.
\end{lem}
\begin{proof}
  By usual adjunction arguments, it suffices to prove that if $i\colon U \rightarrow V$ and $j:A\to B$ are cofibrations of bisimplicial sets, then the pushout product map $i\mathbin{\widehat{\times}}j$ is a cofibration that is acyclic (i.e.\ a levelwise weak equivalence) if $j$ is.
  \begin{equation*}
    \begin{tikzcd}
      U\times A \ar[r,"i\times 1_A"] \ar[d,"1_U\times j"] & V\times A \ar[ddr,bend left,"1_V\times j"] \ar[d,"k"] \\
      U\times B \ar[drr,bend right,"i\times 1_B"] \ar[r] & \bullet \ar[dr,dashed,"i\mathbin{\widehat{\times}}j"]\\
      && V\times B
    \end{tikzcd}
  \end{equation*}
  All the solid arrows in this diagram are monomorphisms and the outer square is a pullback; thus so is the dashed arrow, being a ``union of subobjects'' of $V\times B$.

  If $j$ is acyclic, then since products of simplicial sets preserve weak equivalences, so do products of bisimplicial sets; hence $1_U\times j$ and $1_V\times j$ are weak equivalences. {Thus the map $1_U \times j$ is an acyclic cofibration so its pushout, the map denoted $k$ in the diagram, is again a weak equivalence.}
  Thus, by the 2-out-of-3 property, $i\mathbin{\widehat{\times}}j$ is a weak equivalence as well.
\end{proof}

\subsection{Modeling type theory with shapes}
\label{sec:models-type-theory}

The usual approach to modeling dependent type theory in a category $\C$ is to exhibit a \emph{comprehension category} over $\C$, which is a Grothendieck fibration $\T\to\C$ equipped with a functor over $\C$:
\[
\begin{tikzcd}
  \T \ar[rr] \ar[dr] && \C^\two \ar[dl,"\mathrm{cod}"] \\ &\C
\end{tikzcd}
\]
that preserves cartesian arrows.
In the homotopy-theoretic context, $\T \hookrightarrow \C^\two$ is the subcategory of fibrations.
The categorical structure of $\C$ of interest induces similar structure on $\T$, to which one applies a \emph{coherence theorem} such as~\cite{LW} to obtain ``strictly stable'' structure on a split comprehension category equivalent to $\T$.
Finally, one constructs a similar split comprehension category with strictly stable structure out of the syntax of type theory, taking the base category $\C$ to be the contexts and the total category $\T$ to be the types-in-context, and proves an ``initiality theorem'' that it is the initial such, and hence maps uniquely into the one constructed from the desired model $\C$.

Of these steps, the initiality theorem is commonly neglected; the proofs in known cases are universally expected to generalize to all other cases, but there is as yet no general theorem.
Similarly, the coherence method of~\cite{LW} is not yet a general theorem but has to be proven separately for each kind of type-theoretic structure.
As our goal here is only to give a sketch of the semantics, we will omit both of these proofs; we confine ourselves to describing informally the relevant comprehension categories and explaining how both the syntax and the semantics yield examples.

To start with, since our type theory has three layers, our comprehension categories must also have three layers.
The cube and tope layers have no ``intra-layer dependencies'', so they do not require a full comprehension category structure individually; instead we can make do with a simple category with products.
This does involve blurring the line between context extension and cartesian product of types (i.e.\ we identify $t:I,s:J$ with $\pair{t,s}:I\times J$), but it is common and unproblematic.\footnote{Otherwise we could talk about ``cartesian multicategories''.}
The dependency between levels is encoded with fibrations as in~\cite{J}.
This leads to:

\begin{defn}
  A \textbf{comprehension category with shapes} is a tower of fibrations
  \[
  \begin{tikzcd}
    \T \ar[rr] \ar[dr] && (\C_2)^\two_{\C_1} \ar[dl,"\mathrm{cod}"] \\ &\C_2 \ar[d, "\pi_2"]\\ & \C_1 \ar[d,"\pi_1"] \\ & \C_0
  \end{tikzcd}
  \]
  in which $\C_0$ has finite products, $\C_1$ has fiberwise finite products (i.e.\ its fibers have finite products preserved by reindexing), $(\C_2)^\two_{\C_1}$ denotes the category of arrows in $\C_2$ that map to identities in $\C_1$ (and all commutative squares between them), and $\T\to(\C_2)^\two_{\C_1}$ preserves cartesian arrows.
\end{defn}

Our type theory with shapes as described in \cref{sec:shape-type-theory} yields a comprehension category with shapes in which:
\begin{itemize}
\item The objects of $\C_0$ are the contexts of cubes, and the morphisms are tuples of terms modulo the equivalence relation of derivable equality in tope logic, i.e.\ if $\Xi\types (t\jdeq s)$ then $t$ and $s$ represent the same morphism in $\C_0$.
\item The objects of $\C_1$ are contexts of topes-in-context, i.e.\ lists $\phi_1,\dots,\phi_n$ where $\Xi\types \phi_i\tope$ for each $i$, with reindexing by substitution along cube-morphisms.
  The morphisms in each fiber are entailments $\Xi\mid\Phi\types\psi$.
\item The objects of $\C_2$ are contexts of types in context, i.e.\ the judgment that we wrote as $\Xi\mid\Phi\types \Gamma\ctx$, and its morphisms are tuples of terms in the type theory modulo judgmental equality.
\item The objects of $\T$ are types-in-context, i.e.\ judgments $\Xi\mid\Phi\mid\Gamma\types A\type$, and its morphisms are terms.
  The functor $\T \to (\C_2)^\two_{\C_1}$ extends a context by a type.
\end{itemize}

On the other hand, the bisimplicial set model yields a comprehension category with shapes in which:
\begin{itemize}
\item $\C_0$ is the category of simplicial \emph{sets} of the form $(\Delta^1)^n$, regarded as spatially-discrete bisimplicial sets.
\item $\C_1$ is the category of monomorphisms of simplicial sets (regarded as spatially-discrete bisimplicial sets) whose codomain is of the form $(\Delta^1)^n$, with the projection $\C_1\to\C_0$ the codomain functor.
\item $\C_2$ is the category of diagrams $\Gamma \twoheadrightarrow \Phi\rightarrowtail I$ where $\Phi\rightarrowtail I$ is an object of $\C_1$ and $\Gamma\twoheadrightarrow \Phi$ is any Reedy fibration of bisimplicial sets.
\item $\T$ is the category of diagrams $A \twoheadrightarrow \Gamma \twoheadrightarrow \Phi\rightarrowtail I$, where $\Gamma \twoheadrightarrow \Phi\rightarrowtail I$ is as in $\C_2$ and $A \twoheadrightarrow \Gamma$ is a Reedy fibration.
\end{itemize}

The discrete embedding of sets in simplicial sets admits both adjoints, providing left and right adjoints to the inclusion of discrete simplicial spaces in bisimplicial sets:
\[
\begin{tikzcd}[row sep=large]
\Set \arrow[r, hook, "\perp", "\perp"'] & \sSet \arrow[l, bend right, "\pi_0"'] \arrow[l, bend left, "\mathrm{ev}_0"] & \rightsquigarrow & \Set^{\DDelta^\op} \arrow[r, hook, "{\mathrm{disc}}" description] & \sSet^{\DDelta^\op} \arrow[l, bend right, "(\pi_0)_*"', "\perp"] \arrow[l, bend left, "\mathrm{ev}_{-,0}", "\perp"'] 
\end{tikzcd}
\]
Hence the subcategory of discrete simplicial spaces is closed under all limits and colimits, which tells us that all the cubes, simplices, and more general shapes are discrete simplicial spaces. In particular, the conclusion of the following lemma applies to all of the shapes in the simplicial type theory and the functions between them.

\begin{lem}\label{lem:discrete-reedy} Any map of discrete bisimplicial sets is a Reedy fibration. In particular, any discrete simplicial space is Reedy fibrant.
\end{lem}
\begin{proof}
First note that any map of discrete simplicial sets is a Kan fibration, for if $S \to T$ is a map of discrete simplicial sets then the displayed lifting problems are transposes:
\[
\begin{tikzcd} \Lambda^n_k \arrow[r] \arrow[d] & S \arrow[d]  & \arrow[d, phantom, "{\leftrightsquigarrow}" description] & \pi_0 \Lambda^n_k \arrow[d, equals] \arrow[r] & S \arrow[d] \\ \Delta^n \arrow[ur, dashed] \arrow[r] & T & ~ & \pi_0 \Delta^n \arrow[r] \arrow[ur, dashed, "{\exists !}"'] & T
\end{tikzcd}
\]

Now the discretely embedded subcategory $\Set^{\DDelta^\op}\hookrightarrow \sSet^{\DDelta^\op}$ is reflective and coreflective and thus closed under weighted limits with any weight $W \in \Set^{\DDelta^\op}$. If $X \to Y$ is a map of discrete bisimplicial sets, then
\[ \{ \Delta^m,X\} \to \{\partial\Delta^m,X\} \times_{\{\partial\Delta^m,Y\}} \{\Delta^m , Y\}\] is a map of discrete simplicial sets, and thus is a Kan fibration.
\end{proof}

In particular, the objects $I,\Phi,\Gamma,A$ in the semantic model are Reedy fibrant objects.
This is not necessary for us here, but in other situations it can be useful to know.

We now describe the structure on a comprehension category with shapes that corresponds to our type theory with shapes from \cref{sec:shape-type-theory}.
No additional structure is required on $\C_0$; the finite products that encode context extension are also sufficient to model product cubes.
On $\C_1$ we require:

\begin{defn}\label{defn:cc-topelogic}
  We say that a comprehension category with shapes has \textbf{pseudo-stable coherent tope logic} if
  \begin{itemize}
  \item the fibers of $\C_1$ are preorders that are equivalent to distributive lattices, 
  \item with meets and joins preserved up to isomorphism by reindexing, 
  \item in which reindexings along diagonal maps in $\C_0$ have left adjoints satisfying the Beck-Chevalley condition, and 
  \item moreover the analogue of~\eqref{eq:tope-eq-jdeq} holds.\footnote{We will not be precise about what~\eqref{eq:tope-eq-jdeq} means categorically, since it holds in the syntactic model by definition, while it holds in the semantic model since there $\jdeq$ means literal equality of morphisms.}
  \end{itemize}
\end{defn}

The rules in \cref{fig:topes} ensure that \cref{defn:cc-topelogic} holds for the syntactic model; the connectives $\top,\land,\bot,\lor$ give the distributive lattice structure.
The left adjoint to reindexing along $I\to I\times I$ takes $t:I \types \phi\tope$ to $t:I,s:I \types (t\jdeq s)\land \phi \tope$; that this corresponds to the usual rules of equality is an observation of Lawvere~\cite{lawvere}; see also~\cite{J}.

In the semantic model, since the category of simplicial sets is coherent {and is closed in bisimplicial sets under all conical limits and colimits,} its subobject posets are distributive lattices with meets and joins preserved by pullback, and left adjoints to pullback of monomorphisms along any monomorphism (such as a diagonal map) are given by composition.

\begin{rmk}
\cref{defn:cc-topelogic} is called ``pseudo-stable'' because the meets and joins in fibers of $\C_1$ are preserved up to isomorphism by reindexing.
This accords with the terminology of~\cite{LW}, although they consider mainly type constructors whose rules do not suffice to determine them uniquely up to isomorphism, so that such pseudo-stability has to be asserted as a structure.
The method of~\cite{LW} (which we will not describe in detail here) then applies to make such structure \emph{strictly} preserved by reindexing, as needed to model type theory.\footnote{We will not need to consider the more generally ``weakly stable'' structure of~\cite{LW}, since all the additional operations of our type theory correspond categorically to objects with a universal property that determines them up to isomorphism.}
\end{rmk}

The structure on $\C_2$ consists, firstly, of analogues of the usual structure for modeling dependent type theory with $\Sigma$, $\Pi$, identity types, and so on, as described in~\cite{LW}.
This exists in both the syntactic and the semantic model for the usual reasons in each case.

Secondly, we have compatibility with the coherent logic.
As usual in a comprehension category, we write $\T(\Gamma)$ for the fiber of $\T$ over $\Gamma\in\C_2$, and $\Gamma.A\to\Gamma$ for the image of such an object in $(\C_2)^\two_{\C_1}$.

\begin{defn}\label{defn:cc-reclor}
  A comprehension category with shapes and pseudo-stable coherent tope logic has \textbf{type eliminations for tope disjunction} if the following hold:
  \begin{itemize}
  \item If $\pi_2(\Gamma)$ is the bottom element of its fiber in $\C_1$, then $\Gamma$ is an initial object of $\C_2$.
  \item If $\pi_2(\Gamma) = \phi\lor\psi$ in a fiber of $\C_1$, with injections $i:\phi\to\phi\lor\psi$ and $j:\psi\to\phi\lor\psi$ and $k:\phi\land\psi\to\phi\lor\psi$, then the following square of reindexings is a pushout in $\C_2$:
    \[
    \begin{tikzcd}
      k^*\Gamma \ar[r] \ar[d] & i^*\Gamma \ar[d] \\ j^*\Gamma \ar[r] & \Gamma
    \end{tikzcd}
    \]
  \end{itemize}
\end{defn}

This appears somewhat different from the rules of \cref{fig:tope-or}, which talk about terms $\Gamma\types a:A$, hence sections of a comprehension $\Gamma.A\to \Gamma$.
But if \cref{defn:cc-reclor} holds then we can define such sections using the universal property of a pushout as in the following diagram.\footnote{The vertical dotted arrows denote the action of $\pi_2 : \C_2\to \C_1$, rather than an actual morphism in a category, although in the semantic model there is such a morphism.}
\[
\begin{tikzcd}[row sep=small,column sep=small]
& & & \Gamma.A \arrow[dd, two heads] \\ 
k^*\Gamma \arrow[rr] \arrow[dr] \arrow[dd, dotted] & & j^*\Gamma \arrow[dr] \arrow[ur, dashed, "a_\psi"'] \arrow[dd,dotted] \\ & i^*\Gamma \arrow[rr, crossing over] \arrow[uurr, bend left, crossing over, dashed, "a_\phi"] & & \Gamma \arrow[dd,dotted] \arrow[uu, bend right, dotted, "{\rec_\lor^{\phi,\psi}(a_\phi,a_\psi)}"'] \\ 
{\phi\land \psi} \arrow[rr, tail] \arrow[dr, tail] \arrow[drrr, tail, "k"] & & {\psi} \arrow[dr, tail, "j"]  \\ 
& {\phi} \arrow[rr, tail, "i"'] \arrow[from=uu, dotted, crossing over] & & {\phi \lor \psi}
\end{tikzcd}
\]
Conversely, \cref{defn:cc-reclor} holds in the syntactic model since morphisms of contexts are tuples of sections of dependent types, so a universal property relating to the latter implies one relating to the former.

For the semantic model, the first condition in \cref{defn:cc-reclor} is easy since bottom elements of subobject lattices are initial objects, and initial objects in $\ssSet$ are strict (i.e.\ any map with initial codomain has initial domain).
The second condition similarly follows from the facts that in a coherent category, unions of subobjects are pushouts under their intersections, and such pushouts are preserved by pullback, and the inclusion of discrete bisimplicial sets preserves colimits.

Finally, there are the extension types.

\begin{defn}\label{defn:cc-exten}
  A comprehension category with shapes has \textbf{pseudo-stable extension types} if whenever we have the following data:
  \begin{equation}\label{eq:ccexten}
  \begin{tikzcd}
    & (\psi^*\Gamma).A \ar[d,->>]\\
    \phi^*\Gamma \ar[r] \ar[d,dotted] \ar[ur,"a"] & \psi^*\Gamma \ar[r] \ar[d,dotted] & \Gamma \ar[d,dotted] & \in \C_2 \ar[d,"\pi_2",shift left]\\
    \Phi\times \phi \ar[r] \ar[dr,dotted] &\Phi \times \psi \ar[r] \ar[d,dotted] & \Phi \ar[d,dotted] & \in \C_1 \ar[d,"\pi_1",shift left]\\
    &\Xi\times I \ar[r] & \Xi & \in\C_0
  \end{tikzcd}
  \end{equation}
  with $A\in\T(\psi^*\Gamma)$ there exists an object $\ccexten{\psi}{A}{\phi}{a} \in \T(\Gamma)$ whose comprehension
  $\Gamma.\ccexten{\psi}{A}{\phi}{a} \to \Gamma$
represents the functor $(\C_2/\Gamma)^\op\to\Set$ that sends $\sigma:\Theta\to\Gamma$ to the set of sections $b$ of $\sigma^*A$ that extend $\sigma^*a$, i.e.\ liftings in the following square:
  \begin{equation}
  \begin{tikzcd}
    \phi^*\Theta \ar[r,"\sigma^*a"] \ar[d] & (\psi^*\Theta).(\sigma^*A) \ar[d,->>]\\
    \psi^*\Theta \ar[r,equals] \ar[ur,dashed,"b"] & \psi^*\Theta
  \end{tikzcd}\label{eq:ccexten-lift}
  \end{equation}
\end{defn}

Comparing this to \cref{fig:exten}, the above diagram of data corresponds exactly to the premises of the first (formation) rule.
The second (introduction) rule says that given any $b$ as in \cref{defn:cc-exten} there is an induced map to $\ccexten{\psi}{A}{\phi}{a}$.
The third and fourth (elimination) rules say that $\ccexten{\psi}{A}{\phi}{a}$ comes with a universal such $b$, and the fifth and sixth ($\beta$-reduction and $\eta$-conversion) rules say that any $b$ is induced by the universal one and that the corresponding map is uniquely determined.
In particular, the syntactic model has pseudo-stable extension types.

The fact that the semantic model also has pseudo-stable extension types is the least trivial part of the semantics.
Although our primary interest is in the bisimplicial sets model, it is hardly any more work to prove a more general theorem.

Let $\fT$ be a propositional coherent theory, i.e.\ a set of axiomatic cubes, cube terms, topes, and tope entailments in the first two layers of our type theory from \cref{sec:shape-type-theory}, such as the simplicial type theory of \cref{sec:strict-interval} or the cubical type theory mentioned in \cref{rmk:cubical}.
This gives rise to a syntactic fibration $\fT_1 \to \fT_0$ as in the first two layers of a comprehension category with shapes, which has coherent tope logic.
A \emph{model} of $\fT$ in a topos (or more generally a coherent category) $\V$ is a morphism of fibrations from this syntactic one to the fibration $\mathrm{Mono}(\V) \to \V$ which preserves finite products in the base and the lattice structure in the fibers.

\begin{defn}\label{defn:model-shapes}
  A \textbf{model category with $\fT$-shapes} consists of:
  \begin{itemize}
  \item A right proper Cisinski model category $\M$, i.e.\ a right proper cofibrantly generated model structure on a Grothendieck topos whose cofibrations are the monomorphisms;
  \item A model of $\fT$ in a coherent category $\V$;
  \item A  coherent functor $\varpi : \V\to\M$;
  \item Such that for any object $U\in \V$, the functor $(\varpi U \times -):\M\to\M$ preserves acyclicity of cofibrations.
  \end{itemize}
\end{defn}

Note that since $\varpi$ preserves finite limits, it preserves monomorphisms, i.e.\ it takes them to cofibrations in $\M$.
Thus, essentially the same proof of \cref{rmk:leibniz-reedy} implies that for any monomorphism $i:U\to V$ in $\V$ and any fibration $p:X\twoheadrightarrow Y$ in $\M$, the induced map $\widehat{\{i,p\}} \colon X^{\varpi V} \to X^{\varpi U} \times_{Y^{\varpi U}} Y^{\varpi V}$ is a fibration, which is acyclic if $p$ is.

Our primary class of examples is the following.

\begin{ex}\label{thm:model-shapes}
  Let $\V = \Set^{I^\op}$ be a presheaf topos containing a model of $\fT$.
  For instance, it might be the classifying topos of $\fT$, if that happens to be a presheaf topos.
  Let $\N$ be a right proper Cisinski model category, and give $\M = \N^{I^\op}$ with the injective model structure, with cofibrations and weak equivalences levelwise; note $\M$ is again a right proper Cisinski model category.
 Let $\varpi : \V = \Set^{I^\op} \to \N^{I^\op}= \M$ be induced by the unique cocontinuous functor $\Set\to\N$, which takes a set $U$ to the coproduct $\coprod_U 1$ of that many copies of the terminal object.
    This is a coherent functor since it is the inverse image of a geometric morphism $\M\to\V$.
    The final condition follows since the cartesian product in $\M = \N^{I^\op}$ is levelwise, as are its acyclic cofibrations, and acyclic cofibrations are closed under $\Set$-indexed copowers in any model category.

  In particular, taking $I= \DDelta$ with the universal strict interval in $\Set^{I^\op} = \sSet$, and $\N = \sSet$ with the Quillen model structure, we recover the bisimplicial sets model considered above.\footnote{Since $\DDelta$ is an elegant Reedy category in the sense of~\cite{BR}, the Reedy model structure on $\N^{\DDelta^\op}$ coincides with the injective model structure.}
  More generally, we can take $I= \DDelta$ with the same universal strict interval, but $\N$ any right proper Cisinski model category; this yields a synthetic theory of ``internal $(\infty,1)$-categories in $\N$''.
\end{ex}

However, the following class of examples is also somewhat interesting.

\begin{ex}\label{thm:model-discrete}
  Let $\V = \sSet$ with the universal strict interval, and let $\M = \sSet^{J^\op}$ be a topos of simplicial presheaves, with some left Bousfield localization of the injective model structure associated to the Quillen model structure, and assume that $\M$ is right proper and a simplicial model category.
  Any locally cartesian closed locally presentable $(\infty,1)$-category, such as a Grothendieck $(\infty,1)$-topos, can be presented by such a model category $\M$.
 Let $\varpi$ be restriction along the projection $\DDelta \times J \to \DDelta$.
  Then for $U\in\V$ and $X\in \M$, the product ${\varpi U}\times X$ is equivalently the simplicial copower of the simplicial enrichment of $\M$, so the final condition in \cref{defn:model-shapes} follows from the axioms of a simplicial model category.

  In the resulting model of simplicial type theory, all types are ``discrete'' in the sense of \cref{sec:discrete-types}, since $A^\two$ is just the simplicial path-object.
  However, this is not completely pointless, since compared to the identity types of ordinary homotopy type theory, the hom-types $\hom_A(x,y)$ of simplicial type theory have strictly functorial behavior, yielding some (but not all) of the advantages of cubical type theory.
\end{ex}

From any model category with $\fT$-shapes, we construct a comprehension category with shapes as follows.
\begin{itemize}
\item $\C_0 = \V$ and $\C_1 = \mathrm{Mono}(\V)$.\footnote{In the bisimplicial sets model, we restricted to the subcategory of $\V$ consisting of cubes, but this is immaterial since the image of the interpretation functor from the syntactic model will automatically land in that subcategory anyway.}
\item $\C_2$ is the category of diagrams $\Gamma \twoheadrightarrow \varpi\Phi\rightarrowtail \varpi I$ where $\Phi\rightarrowtail I$ is an object of $\C_1$ and $\Gamma\twoheadrightarrow \varpi\Phi$ is any fibration in $\M$.
\item $\T$ is the category of diagrams $A \twoheadrightarrow \Gamma \twoheadrightarrow \varpi\Phi\rightarrowtail \varpi I$, where $\Gamma \twoheadrightarrow \varpi\Phi\rightarrowtail \varpi I$ is as in $\C_2$ and $A \twoheadrightarrow \Gamma$ is a fibration in $\M$.
\end{itemize}
For the same reasons described above for the bisimplicial set model, this comprehension category has pseudo-stable coherent tope logic with type eliminations for tope disjunction.
The latter uses the fact that $\varpi$ is a coherent functor.
It remains to prove:

\begin{thm}
  For any model category with $\fT$-shapes, the above comprehension category with shapes has pseudo-stable extension types.
\end{thm}
\begin{proof}
A shape inclusion $t:I \mid \phi \types \psi$ is modeled by a monomorphism $i \colon \phi \rightarrowtail \psi$ in $\V$. In the case where $A$ and $a$ are defined in the empty context, then the extension type $\ccexten{\psi}{A}{\phi}{a}$ is constructed simply by the pullback:
 \[
 \begin{tikzcd} {\ccexten{\psi}{A}{\phi}{a}} \arrow[r] \arrow[d, two heads] \arrow[dr, phantom, "\lrcorner" very near start] & A^{\varpi\psi} \arrow[d, two heads, "A^{\varpi i}"] \\ 1 \arrow[r, "a"] & A^{\varpi\phi}
 \end{tikzcd}
 \]
where $A^{\varpi i}$ is a fibration by the observation after \cref{defn:model-shapes}.
In the general case, extension types are again constructed by a similar pullback, though there is some delicacy in expressing the context dependence correctly.

Recall that the context $\Xi\mid\Phi \types \Gamma \ctx$ is modeled by a fibration $\Gamma \twoheadrightarrow \varpi\Phi$, where $\Phi \rightarrowtail \Xi$ is another monomorphism in $\V$.  The type $A$ is then a further fibration $p \colon A \twoheadrightarrow \Gamma \times \varpi\psi$, and the dependent term $a : A$ is a section as shown in the diagram below, most of which is just~\eqref{eq:ccexten} specialized to the model in question.
For brevity we omit $\varpi$ from the notation, identifying objects and morphisms in $\V$ with their images in $\M$.
Importantly, note that the exponentials are in $\M$ and not in any slice category thereof.
\[
\begin{tikzcd}
& A \arrow[d, two heads, "p"] & {\ccexten{\psi}{A}{\phi}{a}} \arrow[d, two heads] \arrow[dr, phantom, "\lrcorner" very near start] \arrow[r] & A^{\psi} \arrow[d, two heads, "{\widehat{\{ i,p\}}}"] \\ 
\Gamma \times \phi \arrow[d, two heads] \arrow[dr, phantom, "\lrcorner", very near start] \arrow[r, "\Gamma \times  i"]  \arrow[ur, dashed, "a"] & \Gamma \times \psi \arrow[d, two heads] \arrow[dr, phantom, "\lrcorner" very near start] \arrow[r, "\pi"] & \Gamma \arrow[d, two heads] \arrow[r, "{(a, \eta)}"] & {A^\phi \times_{(\Gamma \times \psi)^{\phi}} (\Gamma \times \psi)^{\psi}} \\ (\Phi \times \phi) \arrow[r, tail, "(\Phi \times i)"] \ar[dr,tail,two heads] & (\Phi \times \psi) \arrow[r, "\pi"] \arrow[d, tail,two heads] & \Phi \arrow[d, tail,two heads] \\ & (\Xi\times I) \ar[r] & \Xi
\end{tikzcd}
\]
The extension type $\ccexten{\psi}{A}{\phi}{a}$ is constructed using the pullback square in the upper right of this diagram.
By the observation after \cref{defn:model-shapes}, the map $\widehat{\{i,p\}}$ is a fibration.
The map $a \colon \Gamma \to A^\phi$ is a transpose of the partial section $a$ of $p$ displayed on the left, while the map $\eta \colon \Gamma \to (\Gamma \times \psi)^\psi$ is the transpose of the identity.
Since $p \circ a = \Gamma \times i$, this pair indeed defines a cone over the pullback ${A^\phi \times_{(\Gamma \times \psi)^{\phi}} (\Gamma \times \psi)^\psi}$; thus the pullback is well-formed.
Of course, a pullback of a fibration is a fibration, so $\ccexten{\psi}{A}{\phi}{a}$ defines an element of $\T(\Gamma)$.
Thus, it remains to argue that it has the correct universal property.

The universal property of ${A^\phi \times_{(\Gamma \times \psi)^{\phi}} (\Gamma \times \psi)^\psi}$ is that it classifies commutative squares from $i \colon \phi \rightarrowtail\psi$ to $p \colon A \twoheadrightarrow \Gamma \times \psi$, and $\widehat{\{i,p\}}$ classifies lifts in such squares.
That is, maps from $\Theta$ into this pullback correspond to commutative squares of the form
\[
\begin{tikzcd}
  {\Theta \times \phi} \arrow[d, tail, "{\Theta \times i}"'] \arrow[r] & A \arrow[d, two heads, "p"] \\ {\Theta \times \psi} \arrow[r] \ar[ur,dashed] & {\Gamma \times \psi}
\end{tikzcd}
\]
(i.e.\ pairs of maps $\Theta\times\phi\to A$ and $\Theta\times\psi\to\Gamma\times\psi$ making the square commute), while lifts of such maps along $\widehat{\{i,p\}}$ correspond to diagonal fillers in such squares.
The classifying map of such a commutative square factors through $(a, \eta)$ just when the bottom map is of the form $\sigma\times 1_\psi$ and the top of the form $a\circ (\sigma\times 1_\phi)$ for some $\sigma:\Theta\to\Gamma$.
Thus, lifts of a given $\sigma:\Theta\to\Gamma$ to $\ccexten{\psi}{A}{\phi}{a}$ classify lifts in the square
\[
\begin{tikzcd} {\Theta \times \phi} \arrow[d, tail, "{\Theta \times i}"'] \arrow[r, "a\circ (\sigma\times 1_\phi)"] & A \arrow[d, two heads, "p"] \\ {\Theta \times \psi} \arrow[r, "\sigma\times 1_\psi"'] \arrow[ur, dashed] & {\Gamma \times \psi}
\end{tikzcd}
\]
Factoring such a square through the pullback of $p$ along $\sigma\times 1_\psi$, we obtain exactly~\eqref{eq:ccexten-lift}.
\end{proof}

This proof is the semantic reason for requiring the shape inclusion $i:I \mid \phi\types\psi$ in \cref{fig:exten} to be defined in the empty context rather than allowed to depend on $\Xi$ and $\Phi$ as well.
Specifically, \cref{rmk:leibniz-reedy} and its generalizations do \emph{not} extend to exponentials in slice categories, so if we allowed such dependence then the analogue of $\widehat{\{i,p\}}$ would not necessarily be a fibration.

We also require:

\begin{thm}
  The pseudo-stable extension types in any model category with $\fT$-shapes satisfy relative function extensionality.
\end{thm}
\begin{proof}
  A fibration has ``contractible fibers'' in type theory just when it is an acyclic fibration model-categorically.
  Thus, \cref{ax:extfunext} holds for the same reason that $\ccexten{\psi}{A}{\phi}{a}\to\Gamma$ is a fibration, since $\widehat{\{i,p\}}$ is acyclic if $p$ is.
\end{proof}

In conclusion, we have:

\begin{thm}
  The comprehension category with shapes constructed from any model category with $\fT$-shapes has pseudo-stable coherent tope logic with type eliminations for tope disjunction, and also pseudo-stable extension types satisfying relative function extensionality.
\end{thm}

Thus, by applying the coherence methods of~\cite{LW}, we can construct from it a strict comprehension category with shapes having strictly stable coherent tope logic with type eliminations for tope disjunction and strictly stable extension types satisfying relative function extensionality.
An initiality theorem will then imply that the syntactic model maps into it uniquely, thereby interpreting our type theory with shapes into any model category with shapes, and in particular into bisimplicial sets.

\begin{rmk}
  A comprehension category with shapes does \emph{not} necessarily have ``universe types'', and in a general model category with shapes there is no obvious way to construct these.
  In~\cite{elreedy} it was shown that the Reedy model structure on bisimplicial sets \emph{does} have the requisite structure to model universe types, which moreover satisfy the univalence axiom, and those universes carry over to our type theory with shapes.
  So in the case of the primary motivating model there is no additional difficulty here, but in the cases of \cref{thm:model-shapes,thm:model-discrete} there may not be universes.
  However, as noted in \cref{sec:notation}, in this paper we did not really use the universe in any essential way; so at the expense of a bit more cumbersome notation our results apply just as well to these examples.
\end{rmk}

In the general case, the axioms of the theory $\fT$ are satisfied by assumption.
For the theory of simplices and the bisimplicial set model, by construction we have $\two \defeq \Delta^1 \square \Delta^0$ as the categorically-embedded 1-simplex, with $0,1 : \two$ the elements of $\two_{0,0}$ corresponding, as usual, to the 1st and 0th face maps.
The discretely embedded inclusion $\Delta^2 \to \Delta^1\times \Delta^1$ models the inequality tope $t :\two,s:\two \types (t \leq s) \tope$.
The fact that this satisfies the theory of a strict interval is part of the theorem, mentioned in \cref{sec:simplices}, that simplicial sets are the classifying topos of that theory.

\subsection{Segal spaces and Rezk spaces}

In this section we show that the Segal types of \cref{sec:Segal-types} correspond exactly to the Segal spaces in the bisimplicial sets model. A similar argument proves that the Rezk types of \cref{sec:Rezk-types} also correspond to the Rezk spaces. 

\begin{defn} A Reedy fibrant bisimplicial set $X$ is a \textbf{Segal space} if and only if for all $m \geq 2$ and $0 < i < m$ the induced map
\[ \{ \Delta^m,X\} \to \{\Lambda^m_i,X\}\] on weighted limits is a trivial fibration in $\sSet$.
\end{defn}

\begin{prop}\label{prop:joyal} A Reedy fibrant bisimplicial set $X$ is a Segal space if and only if the induced map
\begin{equation}\label{eq:segal} X^{\Delta^2 \square \Delta^0} \to X^{\Lambda^2_1 \square \Delta^0}\end{equation} is a Reedy trivial fibration.
\end{prop}
\begin{proof}
Transposing across the adjunction between the cartesian product and internal hom for bisimplicial sets, \eqref{eq:segal} is a Reedy trivial fibration if and only if $X$ has the right lifting property with respect to the set of maps
\[ \{ ( (\partial\Delta^m\hookrightarrow\Delta^m)\widehat{\square}(\partial\Delta^n\hookrightarrow\Delta^n)) \widehat{\times} (\Lambda^2_1 \square \Delta^0 \hookrightarrow \Delta^2 \square \Delta^0) \}_{m,n \geq 0}.\]
This set is isomorphic to
\[ \{ ( (\partial\Delta^m\hookrightarrow\Delta^m) \widehat{\times}(\Lambda^2_1\hookrightarrow\Delta^2))\widehat{\square}(\partial\Delta^n\hookrightarrow\Delta^n)  \}_{m,n \geq 0},\] where the left-hand product is now the cartesian product on $\sSet$. Transposing across the weighted limit adjunction, we see that \eqref{eq:segal} is a Reedy trivial fibration if and only if the induced map on weighted limits
\[ \{ \Delta^m \times \Delta^2, X\} \to \{ \Delta^m \times \Lambda^2_1 \bigcup\limits_{\partial\Delta^m\times \Lambda^2_1} \partial\Delta^m \times\Delta^2, X\}.\] is a trivial fibration of simplicial sets. By the following combinatorial lemma of Joyal, this precisely characterizes the Segal spaces.
\end{proof}

\begin{lem}[{Joyal \cite[2.3.2.1]{HTT}}]\label{lem:2.3.2.1} The following sets generate the same class of morphisms of simplicial sets under coproduct, pushout, retract, and sequential composition:
\begin{enumerate}
\item The inner horn inclusions $\Lambda^m_i \hookrightarrow\Delta^m$ for $m \geq 2$, $0< i < m$.
\item The collection of all inclusions
\[ \{ \Delta^m \times \Lambda^2_1 \bigcup\limits_{\partial\Delta^m\times \Lambda^2_1} \partial\Delta^m \times\Delta^2 \hookrightarrow \Delta^m \times \Delta^2\}_{m \geq 0}.\]
\end{enumerate}
\end{lem}

Let $E$ denote the simplicial set defined as the colimit of the diagram
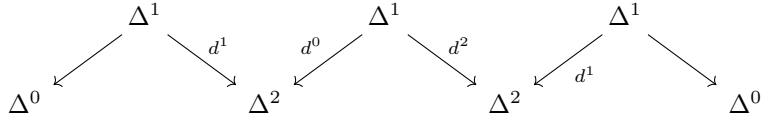
\[
\begin{tikzcd} & \Delta^1 \arrow[dl] \arrow[dr, "d^1"] & & \Delta^1 \arrow[dl, "d^0"'] \arrow[dr, "d^2"]  & & \Delta^1 \arrow[dl, "d^1"] \arrow[dr] \\ \Delta^0 & & \Delta^2 & & \Delta^2 & & \Delta^0
\end{tikzcd}
\]

The simplicial set $E$ together with its ``middle'' 1-simplex may be regarded as the ``free-living bi-invertible map'', equipped with left and right inverses.

\begin{defn}\label{defn:rezk-space} A Segal space $X$ is a \textbf{Rezk space} if and only if the map
\[ \{E,X\} \to \{\Delta^0,X\}\cong X_0\] on weighted limits induced by either vertex map $\Delta^0 \to E$ is a trivial fibration in $\sSet$.
\end{defn}

Our first task is to re-express the Rezk-completeness condition in the internal language of bisimplicial sets.

\begin{prop}\label{prop:rezk-space} A Segal space $X$ is a Rezk space if and only if the induced map
\begin{equation}\label{eq:rezk} X^{E \square \Delta^0} \to X^{\Delta^0 \square \Delta^0} \cong X
\end{equation}
is a Reedy trivial fibration.
\end{prop}
\begin{proof}
As in the proof of Proposition \ref{prop:joyal}, the condition that \eqref{eq:rezk} is a Reedy trivial fibration
transposes to the condition that  the induced map on weighted limits
\[ \{ \Delta^m \times E, X\} \to \{ \Delta^m \times \Delta^0 \bigcup\limits_{\partial\Delta^m\times \Delta^0} \partial\Delta^m \times E, X\}\] is a trivial fibration of simplicial sets for all $m \geq 0$. In the case $m=0$ this is the condition of Definition \ref{defn:rezk-space} so we see that the lifting property \eqref{eq:rezk} implies the completeness condition.

For the converse we appeal to known model categorical results to avoid having to prove a combinatorial lemma. The inclusion $\Delta^0 \to E$ is a trivial cofibration in the Joyal model structure and the discrete embedding $-\square\Delta^0 \colon \sSet \to \ssSet$ of simplicial sets into bisimplicial sets is a left Quillen equivalence from the Joyal model structure to the Rezk model structure \cite[4.11]{JT}. As the Rezk model structure is cartesian monoidal with the Rezk spaces as its fibrant objects, it follows that if $X$ is a Rezk space, then the map \eqref{eq:rezk} is a trivial fibration.
\end{proof}

\cref{prop:rezk-space} is equivalent to the condition that the map $X \to X^{E \square\Delta^0}$ is an equivalence in the Rezk model structure or equivalently in the Reedy model structure. Here $X^{E\square\Delta^0}$ is a model for the total space of the type of isomorphisms introduced in \cref{sec:isos}. Thus, Proposition \ref{prop:rezk-space} corresponds to the Rezk-completeness condition of Definition \ref{defn:rezk-complete}.

\begin{rmk}\label{rmk:boavida}
  There is also a similar characterization of our covariant fibrations.
 By \cref{thm:covariance-as-pullback}, a Reedy fibration $\pi \colon C \to A$ of bisimplicial sets is a covariant fibration in our sense if and only if the square
  \begin{equation}
  \begin{tikzcd}
    C^\two \arrow[r, "{\pi^\two}"] \arrow[d, "{\mathsf{ev}_0}"'] &
    A^\two \arrow[d, "{\mathsf{ev}_0}"] \\
    C \arrow[r, "\pi"'] &
    A 
  \end{tikzcd}\label{eq:htpy-pullback-untest}
  \end{equation}
  is a homotopy pullback.
  If $A$ and $C$ are Rezk, then this corresponds exactly to the characterization of ``groupoidal cartesian fibrations'' in~\cite[Proposition 4.2.7]{RV4}, specialized to the $\infty$-cosmos of Rezk spaces.

  On the other hand, for arbitrary $A$ and $C$, \cref{rmk:leibniz-reedy} guarantees that the induced map to the pullback
  \begin{equation}\label{eq:htpy-pullback-test} \pair{ C^{i_0}, \pi^\two}  \colon C^\two \to C \times_A A^\two\end{equation}
  is a Reedy fibration; thus~\eqref{eq:htpy-pullback-untest} is a homotopy pullback square if and only if~\eqref{eq:htpy-pullback-test} is a Reedy trivial fibration.
  By~\cite[Lemma 2.1.3]{kv:yoneda-css}, this happens if and only if for each $n\ge 1$ the square of \emph{simplicial sets}
  \[ \begin{tikzcd} C_n \arrow[r, "\pi"] \arrow[d, "{\mathsf{ev}_0}"'] &
    A_n \arrow[d, "{\mathsf{ev}_0}"] \\
    C_0 \arrow[r, "\pi"'] &
    A_0
    \end{tikzcd}\]
  is a homotopy pullback; such maps are called \emph{left fibrations} in~\cite{kv:yoneda-css}.
  By~\cite[Proposition 1.7]{boavida:segr} and~\cite[Lemma 3.9]{rasekh:yoneda-ss}, if $A$ and $C$ are Segal spaces then it suffices to assert this for $n=1$.
  Moreover, by~\cite[Proposition 1.10]{boavida:segr}, left fibrations over a Segal space $A$ are the fibrant objects in a model structure that is Quillen equivalent to left fibrations over quasi-categories, and by~\cite[Theorem 4.8]{rasekh:yoneda-ss} this remains true for arbitrary $A$.
  Thus, just as our Rezk types coincide with Rezk spaces and hence model $(\infty,1)$-categories, our covariant fibrations model $\infty$-groupoid-valued $(\infty,1)$-functors.
\end{rmk}


\bibliographystyle{alpha}
\bibliography{stt}

\end{document}